\documentclass[a4paper]{amsart}

\usepackage{amscd,amsthm,amsfonts,amssymb,amsmath}
\usepackage[colorlinks=true]{hyperref}
\usepackage{fullpage}
\usepackage[all]{xy}

    \newcommand{\BA}{{\mathbb {A}}} 
    \newcommand{\BC}{{\mathbb {C}}} 
     \newcommand{\BF}{{\mathbb {F}}}

    \newcommand{\BQ}{{\mathbb {Q}}}

     \newcommand{\BZ}{{\mathbb {Z}}}

     \newcommand{\CL}{{\mathcal {L}}}
     \newcommand{\CN}{{\mathcal {N}}}
    \newcommand{\CO}{{\mathcal {O}}}

    \newcommand{\fa}{{\mathfrak{a}}} 
    \newcommand{\fc}{{\mathfrak{c}}}

    \newcommand{\fm}{{\mathfrak{m}}} 
     \newcommand{\fp}{{\mathfrak{p}}}

     \newcommand{\fX}{{\mathfrak{X}}}

    \newcommand{\Aut}{{\mathrm{Aut}}}

    \newcommand{\Cl}{{\mathrm{Cl}}}

    \newcommand{\disc}{{\mathrm{disc}}}

    \newcommand{\End}{{\mathrm{End}}}

      \newcommand{\Fitt}{{\mathrm{Fitt}}}

    \newcommand{\Gal}{{\mathrm{Gal}}} \newcommand{\GL}{{\mathrm{GL}}}
    
    \newcommand{\Hom}{{\mathrm{Hom}}}
    
    \renewcommand{\Im}{{\mathrm{Im}}}

    \newcommand{\ord}{{\mathrm{ord}}} \newcommand{\rank}{{\mathrm{rank}}}
     \newcommand{\Pic}{\mathrm{Pic}}

    \renewcommand{\mod}{\ \mathrm{mod}\ }
    
    \newcommand{\ram}{{\mathrm{ram}}}
    
    \newcommand{\Res}{{\mathrm{Res}}}
    \newcommand{\Sel}{{\mathrm{Sel}}} 
    
    \newcommand{\supp}{{\mathrm{supp}}}

      \newcommand{\Supp}{{\mathrm{Supp}}}
    \newcommand{\tor}{{\mathrm{tor}}}

    \newcommand{\vol}{{\mathrm{vol}}}

\newcommand{\matrixx}[4]{\begin{pmatrix}
#1 & #2 \\ #3 & #4
\end{pmatrix} }        % 2*2 matrix

    \font\cyr=wncyr10

    \newcommand{\Sha}{\hbox{\cyr X}}
    \newcommand{\wh}{\widehat}
    
    \newcommand{\pair}[1]{\langle {#1} \rangle}

    \newcommand{\ov}{\overline}

    \newcommand{\lra}{\longrightarrow}
    \newcommand{\ra}{\rightarrow}

    \theoremstyle{plain}
    \newtheorem{thm}{Theorem}[section] \newtheorem{cor}[thm]{Corollary}
    \newtheorem{lem}[thm]{Lemma}  \newtheorem{prop}[thm]{Proposition}
    \newtheorem {conj}[thm]{Conjecture} \newtheorem{defn}[thm]{Definition}

\theoremstyle{remark} \newtheorem{remark}[thm]{Remark}
\theoremstyle{remark} 
\theoremstyle{remark} 

    \newcommand{\cO}{\mathcal O}
    \numberwithin{equation}{section}

\begin{document}

\title{Horizontal variation of Tate--Shafarevich groups}

\author{Ashay A. Burungale, Haruzo Hida, and Ye Tian}

\address{Ashay A. Burungale:  Institute for Advanced Study,
Einstein Drive, Princeton NJ 08540} \email{ashayburungale@gmail.com}

\address{Haruzo Hida: Department of Mathematics, UCLA, 
Los Angeles, CA 90095-1555, U.S.A.} \email{hida@math.ucla.edu}

\address{Ye Tian: Academy of Mathematics and Systems
Science, Morningside center of Mathematics, Chinese Academy of
Sciences, Beijing 100190 and \newline
School of Mathematical Sciences, University of Chinese Academy of Sciences, Beijing 10049} \email{ytian@math.ac.cn}

\maketitle
\begin{abstract}
Let $E$ be an elliptic curve over $\BQ$. 
Let $p$ be an odd prime and $\iota: \ov{\BQ}\hookrightarrow \BC_p$ an embedding. 
Let $K$ be an imaginary quadratic field and $H_{K}$ the corresponding Hilbert class field. 
For a class group character $\chi \in \widehat{\Cl}_{K}$, let $\BQ(\chi)$ be the field generated by $\Im(\chi)$ and 
$\fp_{\chi}$ the prime of $\BQ(\chi)$ above $p$ determined via $\iota_p$.
Under mild hypotheses, we show that the number of $\chi \in \widehat{\Cl}_{K}$ 
such that the $\chi$-isotypic Tate--Shafarevich group $\Sha(E/H_{K})^{\chi}$ is finite with trivial $\fp_{\chi}$-part increases with the absolute value of the discriminant of $K$.

We also show an analogous result simultaneously for $\Sha(E/H_{K})^{\chi}$ and the analytic Sha of a $\GL_2$-type abelian variety over $\BQ$. Here analytic Sha refers to the order of the conjecturally finite Tate--Shafarevich group as predicted by the Birch--Swinnerton-Dyer conjecture.
%We consider the non-vanishing of the family of central-critical Rankin-Selberg L-values $L(\frac{1}{2},f\otimes \lambda\chi)$,
%as $\chi$ varies over the class group characters of $K$.

The approach is based on the arithmetic of automorphic forms on a quaternion algebra along with Euler system of toric periods and Heegner points.
An underlying role is played by Michel's Equidistribution
%/ Zariski density
of special points on a Shimura set arising from a definite quaternion algebra and 
Zariski density of special points on a Shimura curve arising from an indefinite quaternion algebra.

\end{abstract}

\tableofcontents

\section{Introduction}

Tate--Shafarevich group is a mysterious arithmetic invariant
associated to a motive $\mathcal{M}$ over a number field $F$. It measures the failure of local to global principle for $\mathcal{M}$ and is conjectured to be finite. In view of the Bloch--Kato conjecture, the conjectural finite size
is expected to be closely related to a special value or a derivative of the $L$-function associated to $\mathcal{M}$.
As the motive varies in a well-controlled family, variation of the Tate--Shafarevich group in the family
seems to be a fundamental question.

An instructive setup arises from twists of $\mathcal{M}$ by Hecke characters over $K$ for a finite  extension $K/F$. Families arise in numerous ways. For example, we may fix $K$ and consider Hecke characters with fixed infinity type but allow the conductor to grow. 
This is usually referred as an Iwasawa or a vertical variation and it continues to be investigated in various situations. 
On the other hand, we may vary $K$ with fixed degree and consider Hecke characters with fixed infinity type and bounded conductor. 
We may refer this as a horizontal variation and it seems to be investigated in only few situations. 
In any of these families, we can ask for variation of the $p$-part of the Tate--Shafarevich group with $p$ a fixed prime.

For a fixed $\GL_2$-type abelian variety over the rationals and a prime $p$,  we hope to initiate the study of horizontal variation of the $p$-part of Tate-Shafarevich groups.
In this article, we consider the $\GL_2$-case for a specific horizontal family.
The family arises from the twists of the abelian variety by ideal class characters over imaginary quadratic fields.
%In view of the modular parametrisation of the abelian variety by a Shimura curve, the family arises naturally in the context of CM points on the Shimura curve.

Let us begin with a brief discussion of the $\GL_1$-case regarding ideal class groups. 
We first recall some history. 
Kummer defined $p$ to be a regular prime if $p \nmid |\Cl(\BQ(\zeta_{p}))|$ for the class group $\Cl(\BQ(\zeta_{p}))$ of the $p^{th}$-cyclotomic field $\BQ(\zeta_{p})$. 
Kummer showed that $p$ is an irregular prime if and only if $p$ divides the numerator of the Bernoulli number $B_k$ for some even integer $k$ with $2 \leq k \leq p-3$. 
After several decades, Herbrand and Ribet showed a refined version incorporating 
$\Gal(\BQ(\zeta_{p})/\BQ)$-action on the class group $\Cl(\BQ(\zeta_{p}))$. 
They showed that the $\omega_{p}^{1-k}$-isotypic class group $\Cl(\BQ(\zeta_{p}))^{\omega_{p}^{1-k}}$ has non-trivial $p$-part if and only if $p$ divides the numerator of the Bernoulli number $B_k$. 
Here 
%$\Cl(\BQ(\zeta_{p}))_{p}$ denotes the $p$-primary subgroup and 
%$\Cl(\BQ(\zeta_{p}))$ is a $\BZ[\Gal(\BQ(\zeta_{p})/\BQ)]$-module and we have 
$\omega_{p}$ denotes the Teichmuller character. 
As $p \nmid |\Gal(\BQ(\zeta_{p})/\BQ)|$, the action of $\Gal(\BQ(\zeta_{p})/\BQ)$ on $\Cl(\BQ(\zeta_{p}))$ is semi-simple.
In these results, the Bernoulli number $B_{k}$ can be replaced with the special zeta-value $\zeta(1-k)$.
% of the Riemann--Zeta function $\zeta(s)$.

We introduce some notation to state the results. 
Let $G$ be an abelian group and $\wh{G}$ the group of $\ov{\BQ}^\times$-valued characters of $G$.
For $\chi \in \wh{G}$, let $\BZ[\chi]$ be the ring generated over $\BZ$ by the image of $\chi$ and $\BQ(\chi)$ its fraction field. For any $\BZ[G]$-module $M$, let $M^\chi$ denote the $\chi$-isotypic part of $M\otimes_\BZ \BZ[\chi]$.   Fix embeddings $\iota_{\infty}:\overline{\BQ}\hookrightarrow \BC$,
$\iota_{p}:\overline{\BQ}\hookrightarrow \BC_{p}$ and an isomorphism $\iota: \BC \ra \BC_p$ such that $\iota_p=\iota\circ \iota_\infty$. Let $v_p$ be the valuation on $\BC_p$ with $v_p(p)=1$. For $\chi\in \wh{G}$, let $\fp_\chi$ be the prime ideal of $\BQ(\chi)$ determined by $\iota_p$.

%Fix a prime $p$ and an embedding $\iota_p : \overline{\BQ}\hookrightarrow \overline{\BQ_{p}}$. 
We now introduce an analogue of Kummer's notion of regularity for Dirichlet characters. 
Let $\chi$ be a $\ov{\BQ}^\times$-valued Dirichlet character. Let $\BQ_\chi$ be the field cut out by $\chi$ and $\Cl(\BQ_\chi)$ its ideal class group with the action of $\Gal(\BQ_\chi/\BQ)$. Suppose that $p\nmid [\BQ_\chi: \BQ]$.
% for $\BQ_\psi$
%being the field cut out by $\psi$.
Following Kummer, we say that
\begin{itemize}
\item[] $\chi$ is a {\em $p$-regular} Dirichlet character  if  the $\chi$-isotypic class group $\Cl(\BQ_\chi)^{\chi}$
 has trivial $\fp_\chi$-part. \end{itemize} By Mazur--Wiles (\cite{MW}) and Wiles (\cite{W}),  it is equivalent to the Dirichlet $L$-value $L(0,\chi^{-1})$ being $\fp_{\chi}$-indivisible or
 % by a fixed prime $\fp$ in $\overline{\BQ}$, or
equivalently to the $p$-indivisibility of the Bernoulli number $B_{1,\chi^{-1}}$.
%In Iwasawa setup, the question has been investigated in \cite{FW}, \cite{Wa} and \cite{Si}. For all but finitely many characters in the Iwasawa tower, the class numbers are shown to be $p$-indivisible. The structure of Iwasawa tower plays an underlying role in the approach and seems responsible to deduce the exceptional set being `thin'. Somewhat surprisingly, the case of Dirichlet characters with arbitrary modulus was not considered until recently.
Recently, Burungale--Sun obtained results towards the number of $p$-regular Dirichlet characters 
%$\chi$ as $\chi$ varies 
%under the condition $p\nmid [\BQ_\chi: \BQ]$ 
(\cite{BuSu}). We refer to \cite[\S1]{BuSu} for their main results. In the Iwasawa setup, the results go back to Washington (\cite{Wa}) and Sinnott (\cite{Si}).

%\begin{thm}
%\label{Dirichlet}
%Let $p$ be an odd prime.
%Let $N >1$ be a positive integer with $p \nmid \phi(N)$ and $\lambda: (\BZ/N\BZ)^{\times} \rightarrow \overline{\BQ}^\times$ a Dirichlet character.
%For any $\epsilon > 0$ we have
%$$
%\#\big{\{} \chi \in \widehat{(\BZ/F\BZ)^{\times}} \big{|}\ \text{$\lambda\chi$ is $p$-regular} \big{\}}\gg_\epsilon \phi(F)^{\frac{1}{2}-\epsilon}.
%$$
%Here $F$ varies over positive integers such that
%\begin{itemize}
%\item[] $p\nmid \phi(F)$.
%\item[(ii)] either $p$ splits in $K$ or the root number $\epsilon(E, \chi_0)$ of $L(s, E, \chi_0)$ is $+1$ for some $\chi_0\in \wh{\Cl}_K$.
%\end{itemize}

%\end{thm}

We now turn to $\GL_2$-type abelian varieties over the rationals. In the introduction, we restrict to elliptic curves for simplicity.

Let $K$ be an imaginary quadratic field and $D_{K}$ the discriminant. 
Let $K_\BA$ be the ring of adeles. Let $G_{K}$ denote the absolute Galois group over $K$. 
Let $\Cl_K$ be the ideal class group of $K$ and $h_K$ the ideal class number.   
Let $H_K$ be the Hilbert class field of $K$. 

Let $E$ be an elliptic curve over $\BQ$. There is a natural action of $\Cl_K$ on the Tate--Shafarevich group
$\Sha(E/H_{K})$ via Artin's reciprocity isomorphism $\Gal(H_{K}/K)\simeq \Cl_K$.  A basic question is the following.
\begin{itemize}
\item[(Q)] For a prime $p$, how does the $\BZ[\Cl_{K}]$-module $\Sha(E/H_{K})[p^{\infty}]$ vary as $K$-varies over imaginary quadratic fields?
\end{itemize}
 A fundamental conjecture of Tate asserts that $\Sha(E/L)$ is finite for a number field $L$.  We are not aware of any conjecture regarding the distribution of the above horizontal variation on $\Sha$. The horizontal variation seems to be complementary to the variation considered in Delaunay (\cite{De}) and Bhargava--Kane--Lenstra--Poonen--Rains (\cite{BKLPR}). 
%For a $\ov{\BQ}$-valued character $\chi$ of $\Cl_K$, let $\BZ[\chi]$ be the ring generated over $\BZ$ by the image of $\chi$ and $\BQ(\chi)$ its fraction field. For any $\Cl_K$-module $M$, let $M^\chi$ denote the $\chi$-isotypic part of $M\otimes_\BZ \BZ[\chi]$.  Let $\wh{\Cl}_K$ denote the group of $\ov{\BQ}$-valued characters of $\Cl_K$. Fix embeddings $\iota_{\infty}:\overline{\BQ}\hookrightarrow \BC$ and
%$\iota_{p}:\overline{\BQ}\hookrightarrow \BC_{p}$. Let $v_p$ be the valuation on $\BC_p$ with $v_p(p)=1$.
%Let $E$ be an elliptic curve over $\BQ$. There is a natural action of $\Cl_K$ on the Tate-Shafarevich group
%$\Sha(E/H_{K})$ via the Artin's reciprocity law isomorphism $\Gal(H_{K}/K)\simeq \Cl_K$.  A basic question is the following.
%\begin{itemize}
%\item[(Q)]
The question is closely related to the variation of the $\BZ[\chi]$-module $\Sha(E/H_{K})^{\chi}[p^{\infty}]$ as the pair $(K,\chi)$-varies with $\chi \in \widehat{\Cl}_{K}$.
%\end{itemize}

% It is  conjectured that $\Sha(E/H_K)^\chi$ is always finite, but we are even not aware of any further conjecture regarding the above horizontal variation on $\Sha$.
Recall that $\fp_\chi\subset \BZ[\chi]$ denotes the prime above $p$ of $\BQ(\chi)$ determined via the embedding $\iota_p$. We consider the horizontal variation of the $\fp_\chi$-primary part of $\Sha(E/H_K)^\chi$ in this article. 
%From practice, 
This seems to be a slightly approachable question than the original (Q).

One of our main results is the following.
\begin{thm}\label{main0}
Let $E$ be a non-CM elliptic curve over $\BQ$.
Then, there exists an explicit finite subset $\Sigma_{E}$ of primes only dependent on $E$, such that for any fixed prime $p \notin \Sigma_{E}$ and $\epsilon > 0$ we have
$$
\#\bigg{\{} \chi \in \widehat{\Cl}_{K} \bigg{|}\ \Sha(E/H_K)^\chi\ \text{is finite with trivial $\fp_\chi$-part} \bigg{\}}\gg_\epsilon (\ln |D_K|)^{1-\epsilon}.
$$
Here $K$ varies over imaginary quadratic fields such that
\begin{itemize}
\item[(i)] $p\nmid h_K$, and
\item[(ii)] either $p$ splits in $K$ or the root number $\epsilon(E, \chi_0)$ of the Rankin--Selberg convolution $L(s, E, \chi_0)$ equals $+1$ for some $\chi_0\in \wh{\Cl}_K$.
\end{itemize}
\end{thm}
\begin{remark}
(1). Under mild hypothesis, a precise description of the exceptional set $\Sigma_E$ is given in Theorem \ref{main2}.

(2). Cohen-Lenstra heuristics predicts that the imaginary quadratic fields satisfying the hypotheses in the theorem have positive density.
It is known that there exist infinitely many imaginary quadratic fields $K$ satisfying the hypotheses
(\cite{Br} and \cite{W}). For a recent quantitative refinement of these results, we refer to \cite{Be}.
\end{remark}

 As for the horizontal variation, we may ask the same question for analytic Sha. Recall that the BSD conjecture predicts the size for the conjecturally finite Tate--Shafarevich groups in terms of the behaviour of L-function around the center and arithmetic invariants of the underlying abelian variety. We refer to this quantity as analytic Sha. 
  In this setup, relevant abelian variety turns out to be the following.
 For $\chi\in \wh{\Cl}_K$, let $E_\chi=E_K\otimes_\BZ \BZ[\chi]$ denote the Serre tensor where the absolute Galois group $G_K$ acts on $\BZ[\chi]$ via $\chi$ (\cite{MRS}). Then $E_\chi$ is an abelian variety over $K$ with the following properties:
 $$\Sha(E_\chi/K)\otimes_\BZ \BZ[h_K^{-1}]\cong \Sha(E/H_K)^\chi\otimes_\BZ \BZ[h_K^{-1}], \qquad L(s, E_\chi)=L(s, E, \chi)\in \BC\otimes_\BQ \BQ(\chi).$$
 Let $\CL(E_\chi)$ be the analytic Sha of $E_\chi$ (part (2) of Conjecture \ref{BSD}), which is conjectured to be a non-zero integral ideal of $\BQ(\chi)$ as predicted by the BSD conjecture for the abelian variety $E_{\chi}$ over $K$.
  \begin{defn}
  A character $\chi\in \wh{\Cl}_K$ is said to be
  %$p$-minimal
  $p$-regular for $E$ if the following conditions hold.
 \begin{itemize}
 \item $\rank_{\BZ[\chi]} E_\chi(K)=\ord_{s=1}L(s, E_\chi/K)$, and
\item  Let $\fp_\chi|p$ be the prime ideal of $\BZ[\chi]$ induced via the fixed embedding $\iota_p$, then
\begin{enumerate}
\item[(a)]  $\Sha(E_\chi)$ is finite with trivial $\fp_\chi$-part, and
        \item[(b)]  $\CL(E_\chi)$ has trivial $\fp_\chi$-part.
            \end{enumerate}
            \end{itemize}
            \end{defn}
In particular, for a $p$-minimal $\chi\in \wh{\Cl}_K$, the $\fp$-part of the full BSD conjecture holds for $E_\chi$. 
%As $\fp$-part of the full BSD is not yet established, we introduce the simultaneous $\fp$-triviality conditions in the above definition. 

Our result regarding $p$-minimality is the following.
\begin{thm}\label{main1}
Let $E$ be a non-CM elliptic curve over $\BQ$. Then, there exists an explicit finite subset $\Sigma_{E}$ of primes only dependent on $E$, such that for any prime $p \notin \Sigma_{E}$ and $\epsilon > 0$ we have
$$
\#\bigg{\{} \chi \in \widehat{\Cl}_{K} \bigg{|}\ \chi\ \text{is $p$-regular for $E$}\ \bigg{\}}\gg_\epsilon (\ln |D_K|)^{1-\epsilon}.
$$
Here $K$ varies over imaginary quadratic fields satisfying
\begin{itemize}
\item[(i)] $p\nmid h_K$, and
\item[(ii)] either $p$ splits in $K$ or the root number $\epsilon(E, \chi_0)$ of the Rankin--Selberg convolution $L(s, E, \chi_0)$ equals $+1$ for some $\chi_0\in \wh{\Cl}_K$.
\end{itemize}
\end{thm}
Note that the above theorem is stronger than Theorem \ref{main0} in appearance.

%\begin{remark}
The set $\Sigma_{E}$ is intrinsically related to the arithmetic of $E$. More precisely, it encodes the Tamagawa numbers of $E$, Galois image of the $p$-adic Galois representation $\rho_{E, p}$ and
the local Galois representations arising from $\rho_{E,p}$ at primes dividing the conductor $N$ of $E$.
%If $E$ is semi-stable, the set $\Sigma_{E}$ can be taken to consist of primes $p$ satisfying one of the following.
%\begin{itemize}
%\item $p|6N\prod_{\ell |N}c_\ell $ for $c_\ell $ being the Tamagawa number at $\ell$.
%\item The $p$-adic Galois representation $\rho_{E,p}:G_{\BQ}\rightarrow \GL_{2}(\BZ_p)$ is not surjective.
%\end{itemize}
%\end{remark}
%Let $\pi$ be the cuspidal automorphic representation of $\GL_{2}(\BA)$ associated to $E$ and $\pi_v$ its component at a place $v$.
For $\ell |N$, let $c_\ell$ be the Tamagawa number of $E$ at $\ell$.
Let $d_\ell\geq 1$ be the greatest common divisor of $[M_\ell: \BQ_\ell]$ for all extensions $M_{\ell}/\BQ_{\ell}$ such that the base change of $E$
%$\pi_{\ell}$
to $M_{\ell}$ has either good or multiplicative reduction.
%representation.
For example, if $\ell\| N$, then $d_\ell=1$.
%Under mild hypotheses on the imaginary quadratic fields, the exceptional set $\Sigma_{E}$ can be made explicit.
An explicit version of Theorem \ref{main1}
%under mild hypotheses
is the following.
\begin{thm}\label{main2}
Let $E/\BQ$ be a non-CM elliptic curve over $\BQ$ with conductor $N$. Let $p\nmid 6N \prod_{\ell |N} (\ell^2-1)c_\ell d_\ell$ be a prime such that the $\rho_{E, p}$ is surjective. Then for any
$\epsilon > 0$,
$$
\#\bigg{\{} \chi \in \widehat{\Cl}_{K} \bigg{|}\
 \text{$\chi$ is $p$-regular for $E$} \bigg{\}}\gg_\epsilon
(\ln |D_K|)^{1-\epsilon}.
$$
Here $K$ varies over imaginary quadratic extensions satisfying
\begin{itemize}
\item[(i)] $p\nmid h_K$,
\item[(ii)] $(N,D_{K})=1$,
 \item[(iii)] $E$ is semi-stable at any prime factor of $N$ inert in $K$;
\item[(iv)] either $p$ splits in $K$  or the number of prime factors of $N$ inert in $K$ is odd.
%$\epsilon(E,\chi')=1$ for some $\chi' \in  \widehat{\Cl}_{K}$.
\end{itemize}
\end{thm}

In the article, we also prove analogous results for $\GL_2$-type abelian varieties over the rationals (\S4). By Khare--Wintenberger (\cite{KhWi}),  such abelian varieties are modular. In turn, modular forms are integral to our approach.

We now describe the approach. It is based on the arithmetic of automorphic forms on a quaternion algebra along with Euler system of toric periods and Heegner points.
The Euler system controls both the Tate--Shafarevich group and analytic Sha.
Our point is to first establish $p$-indivisibility of the Euler system. 
A finer analysis of the relation of the Euler system with Tate--Shafarevich group and analytic Sha then leads to the main result.
An underlying role is played by the following.
\begin{itemize}
\item Michel's equidistribution
%/ Zariski density
of special points arising from varying imaginary quadratic fields on a Shimura set arising from a definite quaternion algebra (\cite{M}).
\item Zariski density of CM points arising from varying imaginary quadratic fields
on a Shimura curve arising from an indefinite quaternion algebra. %play an underlying role.
\end{itemize}

As $\chi \in \wh{\Cl}_{K}$, the Rankin--Selberg convolution $L(s,E,\chi)$ corresponding to the pair
$(E,\chi)$ is self-dual. Let $\epsilon(E,\chi)$ denote the global root number.
Let
$$
\fX_{K}^{\pm}=\bigg{\{} \chi \in \widehat{\Cl}_{K}\bigg{|} \epsilon(E,\chi)=\pm 1\bigg{\}}.
$$
The approach crucially relies on the root number. Let $S$ be the set of places of $\BQ$ dividing $N\infty$. For each character $\chi=\otimes \chi_v$ of $\Cl_K$ viewed as a Hecke character of $K_\BA^\times$ via class field theory, let $\chi_S=\otimes_{v\in S}\chi_v: \prod_{v\in S} K_v^\times \ra \ov{\BQ}^\times$ be its $S$-component. 
Some of the notation used here and in the sketch below is not followed in the rest of the article. 

  The set $\fX_{K}^{+}$ admits a finite partition according to the $S$-local component  of characters in $\fX_K^+$. For a fixed type $\chi_{0, S}$, there is a definite quaternion algebra $B/\BQ$ and a so-called $\chi_{0, S}$-toric line $V$ in $\pi_B$,  
  %the  Jacquet-Langlands transfer of $\pi$ to $B_{\BA}^{\times}$, 
  such that $L(1, E, \chi)\neq 0$ if and only if the period integral
$$P_f(\chi):=\int_{K_\BA^\times/K^\times \BA^\times} f(t) \chi(t) dt$$
is non-vanishing for any non-zero vector $f\in V$. 
This is an inexplicit version of Waldspurger formula (\cite{Wa1} and \cite{YZZ}). 
Here $\pi_B$ denotes the Jacquet-Langlands transfer of $\pi$ to $B_{\BA}^{\times}$ for $B_{\BA}=B \otimes \BA$ with $\BA$ the adeles over $\BQ$. 
%Here for any character $\chi\in \fX_K^+$ with a fixed type $\chi_{0, S}$ 
Moreover, we fix an embedding of $K$ into $B$ dependent on the $S$-component. 
%and by which view $K^\times$ as a $\BQ$-subalgebra of $B$.

The space $V$ contains non-zero vectors and we fix such a form $f$ from now. 
Let $U$ be an open compact subgroup of $B_{\BA}^{(\infty),\times}$ such that $f$ is $U$-invariant. 
Here $B_{\BA}^{(\infty)}$ denotes the finite part. 
Let $X_U$ be the corresponding definite Shimura set with level $U$. We can choose $U$ such that for any pair $(K, \chi)$ with the fixed type $\chi_{0, S}$, the chosen embedding $K_\BA$ into $B_\BA$ induces a map $\varphi_K: \Cl_{K}\ra X_U$. The image of the map is referred as the special points arising from
the class group $\Cl_{K}$.
%Let $U$ be a open compact subgroup of $\wh{B}^\times$ such that forms in $V$ are $U$-invariant and let $X_U=B^\times \bs \wh{B}^\times/U$ be the  Shimura set of level $U$. We can choose $U$ such that for any pair $(K, \chi)$ with the fixed type $\chi_{0, S}$, the chosen embedding $K$ into $B$ induces a map $\varphi_K: \Cl_K\ra X_U$. The image of the map is referred as the special points arising from $\Cl_K$.
Let $\fp_0|p$ be the prime of $\BZ[\chi_{0, S}]$ induced via the embedding $\iota_p$. Without loss of generality, we may suppose that the image of $f: X_U\ra \BZ[\chi_{0, S}]$ generates an ideal prime to $\fp_0$.

A key point is to study the non-vanishing of toric periods $P_{f}(\chi)$ modulo $\fp_\chi$ as the pair
$(K,\chi)$-varies with $S$-component $\chi_{0,S}$ (thus $\fp_\chi|\fp_0$).
The $\fp_\chi$-indivisibility of $P_f(\chi)$ begins with Fourier analysis on $\Cl_{K}$ and its relation with the induced map $\varphi_{K}:\Cl_{K}\rightarrow X_{U}$. Recall that $X_U$ is a finite set whereas the class number grows with the discriminant of $K$.
Based on Shimura's reciprocity law, we reduce the $\fp_\chi$-indivisibility to equidistribution of special points on the Shimura set $X_{U}$.
%These special points arise from the subgroups of the class group $\Cl_K$ with bounded index and the equidistribution concerns the variation of $K$.
The equiditribution concerns the distribution of the special points
$$
\left\{\varphi_{K}(H_{K})\ \Big{|}\  H_{K} \subset \Cl_{K}\right\} \subset X_{U}
$$
with $H_{K} \subset \Cl_{K}$ a subgroup such that $[\Cl_{K}: H_{K}] \ll |D_{K}|^{\delta}$ for some $\delta > 0$.
 Such an equidistribution is essentially due to Michel (\cite{M}). It is  based on the subconvex bound due to Michel--Venkatesh \cite{MV} and the explicit Waldspurger formula in Cai--Shu--Tian (\cite{CST}). The hypothesis on the $p$-indivisibility of the class number of the imaginary quadratic fields seems essential for the Fourier analysis. 

In view of the the work on the Euler system of toric periods
due to Bertolini--Darmon \cite{BD} as refined by Nekov\'a\v{r} \cite{N},
the non-vanishing of the toric period $P_{f}(\chi)$ implies the finiteness of $\Sha(E/H_{K})^{\chi}$.
Moreover, there exists a finite set of primes $\Sigma_{E}$ such that the size of
the $\fp_\chi$-part $\Sha(E/H_{K})^{\chi}[\fp_{\chi}^{\infty}]$ is controlled by the $\fp_{\chi}$-divisibility of the toric period $P_{f}(\chi)$ for $p \notin \Sigma_{E}$.
More precisely, we have
$$
P_{f}(\chi)\cdot \fp_{\chi}^{C_{\chi, \fp_\chi}}\cdot \Sha(E/H_{K})^{\chi}[\fp_{\chi}^{\infty}]=0
$$
for a non-negative integer $C_{\chi, \fp_\chi}$.
Our point is to investigate horizontal variation of the constants $C_{\chi, \fp_\chi}$.
We show the existence of a finite set of primes $\Sigma_{E,1}$
such that the constant $C_{\chi, \fp_\chi}$ vanishes as $(K, \chi)$-varies for $p \notin \Sigma_{E,1}$. 
The hypothesis on the $p$-indivisibility of the class number of the imaginary quadratic fields seems essential for this control. 
Along with the $\fp_{\chi}$-indivisibility of toric periods $P_{f}(\chi)$, this leads to the proof Theorem \ref{main0} in the setup.

The situation is analogous for analytic Sha.
Based on the explicit Waldspurger formula (\cite[\S1.3]{CST}), we have $\mathcal{L}(E,\chi)\in \BQ(\chi)^{\times}$ as long as the analytic rank equals zero. More precisely, we have
$$
\mathcal{L}(E_{\chi})=D_{\chi}\cdot P_{f}(\chi)^{2}
$$
for a fractional ideal $D_{\chi}\subset \BZ[\chi]$.
Our point is to investigate horizontal variation of the ideals $D_{\chi}$.
We show the existence of a finite set of primes $\Sigma_{E,2}$ such that
$v_{p}(D_{\chi})=0$ for $p \notin \Sigma_{E,2}$.
Along with the $\fp_{\chi}$-indivisibility of toric periods and the horizontal variation of the constants $C_{\chi,\fp_{\chi}}$, this leads to the proof of Theorem \ref{main1} in the setup.

Under the hypotheses of Theorem \ref{main2}, a finer analysis of the above approach especially that
of the constants $C_{\chi,\fp_{\chi}}$ and the ideals $D_{\chi}$, leads to the proof of Theorem \ref{main2} in the setup. 

The lower bound in Theorem \ref{main0} and Theorem \ref{main2} is based on Shimura's reciprocity law and an elementary group theoretic lemma. Roughly speaking, the reciprocity law implies that the set $S_{K}$ of $p$-regular characters is invariant under a Galois action. The Fourier analysis implies that the subgroup generated by $S_{K}$ has size bounded below by $|\Cl_{K}|^{\delta'}$ for some $\delta' > 0$. An elementary argument then shows that the size of $S_{K}$ is bounded below by $(\log(|\Cl_{K}|))^{1-\epsilon}$ for $\epsilon > 0$.

The $p$-adic Waldspurger formula due to Bertolini--Darmon--Prasanna (\cite{BDP1}), Brooks (\cite{Bro}) and Liu--Zhang--Zhang (\cite{LZZ}) allows an analogous approach in the rank one case.

The set $\fX_{K}^{-}$ admits a finite partition according to the $S$-local component  of characters in
$\fX_K^-$. For a fixed type $\chi_{0, S}$, there is an indefinite quaternion algebra $B/\BQ$ and a so-called $\chi_{0, S}$-toric line $V$ in $\pi_B$, 
% the  Jacquet-Langlands transfer of $\pi$ to $B_{\BA}^{\times}$, 
such that $L'(1, E, \chi)\neq 0$ if and only if the Heegner point
$$P_f(\chi):=\int_{K_\BA^\times/K^\times \BA^\times} f(t) \chi(t) dt$$
is non-torsion for any non-zero vector $f\in V$. 
This is an inexplicit version of Gross--Zagier formula (\cite{GZ} and \cite{YZZ}). 
Here $\pi_B$ denotes the Jacquet-Langlands transfer of $\pi$ to $B_{\BA}^{\times}$. 
%Here for any character $\chi\in \fX_K^+$ with a fixed type $\chi_{0, S}$ 
Moreover, we fix an embedding of $K$ into $B$ dependent on the $S$-component. 
In the above expression for the Heegner point, $f$ is viewed as a modular parametrisation $f: X_{U} \ra E_{\chi_{0,S}}$ for Shimura curve $X_U$ arising from $B$ with level $U$ specified below and the Serre tensor $E_{\chi_{0,S}}$.

The space $V$ contains non-zero vectors and we fix such a form $f$ from now. 
Let $U$ be an open compact subgroup of $B_{\BA}^{(\infty),\times}$ such that $f$ is $U$-invariant.
Let $X_U$ be the corresponding Shimura curve with level $U$. We can choose $U$ such that for any pair $(K, \chi)$ with the fixed type $\chi_{0, S}$, the chosen embedding $K_\BA$ into $B_\BA$ induces a map $\varphi_K: \Cl_{K}\ra X_U$. The image of the map is referred as the CM points arising from
the class group $\Cl_{K}$.
Without loss of generality, we may suppose that
%Let $\fp_0|p$ be the prime of $\BZ[\chi_{0, S}]$ induced by $\iota_p$.
%The space $V$ contains a non-zero  
$f$ is a non-zero $p$-primitive form.

A key point is to study the non-vanishing of toric periods $P_{\chi}(g)$ modulo $\fp_\chi$ as the pair
$(K,\chi)$-varies with $S$-component $\chi_{0,S}$.
% (thus $\fp_\chi|\fp_0$).
Here $g$ is a weight zero $p$-adic modular form given by $d^{-1}(f^{(p)})$ for $d$ the Katz $p$-adic differential operator and $f^{(p)}$ the $p$-depletion of $f$. In view of the $p$-adic Waldspurger formula, the $p$-adic logarithm of the Heegner point $P_{f}(\chi)$ is a $p$-unit if and only if a normalisation of the toric period $P_{g}(\chi)$ is a $p$-unit. In this sense, the $\fp_{\chi}$-indivisibility of the Heegner point $P_{f}(\chi)$ is closely related to the $p$-indivisibility of the toric period
$P_g(\chi)$. 

Strictly speaking, $P_{g}(\chi)$ is a toric period on an Igusa tower over the Shimura curve. Let $\pi: Ig \ra X_{U}$ be the $p$-ordinary Igusa tower. 
Recall that $Ig$ is a formal scheme defined over $W$ the Witt ring of the algebraic closure $\BF$ of $\BF_p$ arising from the embedding $\iota_p$.
We also recall that $g$ is a function on the Igusa tower $Ig$.
% over the Shimura curve $X_U$.
The $p$-indivisibility of $P_g(\chi)$ begins with Fourier analysis on $\Cl_{K}$ and its relation with the induced map $\varphi_{K}':\Cl_{K}\rightarrow Ig$.
Here $\varphi_{K}'$ is a lift of $\varphi_K$.
%Recall that $X_U$ is a finite set whereas the class number grows with the discriminant of $K$.
Based on Shimura's reciprocity law, we reduce the $p$-indivisibility to Zariski density of the mod $p$ reduction of special points on the special fiber $Ig_{\BF}$.
%These special points arise from the subgroups of the class group $\Cl_K$ with bounded index and the equidistribution concerns the variation of $K$.
The Zariski density concerns the density of the special points
$$
\left\{\varphi_{K}'(H_{K})\ \Big{|}\  H_{K} \subset \Cl_{K}\right\} \subset Ig_{\BF}
$$
with $H_{K} \subset \Cl_{K}$ a subgroup of bounded index as $|D_{K}| \rightarrow \infty$. 
The Zariski density is a consequence of CM theory and the geometry of Igusa tower modulo $p$. 
The hypothesis on the $p$-indivisibility of the class number of the imaginary quadratic fields seems essential for the Fourier analysis. 

In view of the the work on the Euler system of Heegner points
due to Kolyvagin \cite{Ko} as refined by Nekov\'a\v{r} \cite{N0},
the non-triviality of the Heegner point $P_{\chi}(f)$ implies the finiteness of $\Sha(E/H_{K})^{\chi}$.
Moreover, there exists a finite set of primes $\Sigma_{E}'$ such that the size of
the $\fp_\chi$-part $\Sha(E/H_{K})^{\chi}[\fp_{\chi}^{\infty}]$ is controlled by the $\fp_{\chi}$-divisibility of the Heegner point $P_{f}(\chi)$ for $p \notin \Sigma_{E}'$.
More precisely, we have
$$
\big{[}E(H_{K})^{\chi}:\BZ[\chi]\cdot P_{f}(\chi)+E(H_{K})^{\chi}_{tor}\big{]}\cdot \fp_{\chi}^{C'_{\chi, \fp_\chi}}\cdot \Sha(E/H_{K})^{\chi}[\fp_{\chi}^{\infty}]=0
$$
for a non-negative integer $C'_{\chi, \fp_\chi}$.
Our point is to investigate horizontal variation of the constants $C'_{\chi, \fp_\chi}$.
We show that there exists a finite set of primes $\Sigma'_{E,1}$
such that the constant $C'_{\chi, \fp_\chi}$ vanishes as $(K, \chi)$-varies for $p \notin \Sigma'_{E,1}$. 
The hypothesis on the $p$-indivisibility of the class number of the imaginary quadratic fields seems essential for this control. 
Along with Galois image result for $E$ and the $\fp_{\chi}$-indivisibility of the Heegner points $P_{f}(\chi)$, this leads to the proof Theorem \ref{main0} in the setup.

The situation is analogous for analytic Sha.
Based on an explicit Gross--Zagier formula in Cai--Shu--Tian (\cite[\S1.3]{CST}), we have $\mathcal{L}(E_{\chi})\in \BQ(\chi)^{\times}$ as long as the analytic rank equals one. More precisely, we have
$$
\mathcal{L}(E_{\chi})=D'_{\chi}\cdot \big{[}E(H_{K})^{\chi}:\BZ[\chi]\cdot P_{f}(\chi)+E(H_{K})^{\chi}_{tor}\big{]}^{2}
$$
for a fractional ideal $D'_{\chi}\subset \BZ[\chi]$.
Our point is to investigate horizontal variation of the ideals $D'_{\chi}$.
We show that there exists a finite set of primes $\Sigma'_{E,2}$ such that
$v_{p}(D'_{\chi})=0$ for $p \notin \Sigma'_{E,2}$.
Along with Galois image result for $E$, the $\fp_{\chi}$-indivisibility of Heegner points and the horizontal variation of the constants $C'_{\chi,\fp_{\chi}}$, this leads to the proof of Theorem \ref{main1} in the setup.

Under the hypotheses of Theorem \ref{main2}, a finer analysis of the above approach especially that
of the constants $C'_{\chi,\fp_{\chi}}$ and the ideals $D'_{\chi}$, leads to the proof of Theorem \ref{main2} in the setup.

As before, the lower bound in Theorem \ref{main0} and Theorem \ref{main2} is based on Shimura's reciprocity law and the elementary group theoretic lemma. 
%Roughly speaking, the reciprocity law allows implies that the set $S$ of $p$-regular characters is invariant under a Galois action. The Fourier analysis implies that the subgroup generated by $S$ has size bounded below by $|\Cl_{K}|^{\epsilon}$. An elementary argument then shows that the size of $S$ is bounded below by $(\log(|\Cl_{K}|))^{1-\epsilon}$.

As evident from the sketch, there is a visible analogy among the case of root number $1$ and the case of root number $-1$. The analogy seem to resonate resemblance among the Waldspurger formula and the $p$-adic Waldspurger formula. 
On the other hand, there are striking differences. Here we only mention that the 
equidistribution of special points relies on a subconvex bound for Rankin--Selberg L-functions 
whereas the Zariski density of CM points relies on the geometry of Igusa tower modulo $p$.  

Based on Jochnowitz congruence (\cite{V} and \cite{BD}), the case of root number $-1$ can be deduced from the case of root number $1$ for variation over imaginary quadratic fields satisfying extra hypothesis that there exists a fixed prime inert in them. 
 
%{\bf{Analytic difference rank 1, rank 0}}

Our main results can be considered as a mod $p$ analogue of the characteristic zero non-vanishing results for Rankin--Selberg L-functions due to Michel--Venkatesh (\cite{MV1} and \cite{MV}) and Templier (\cite{Te}). 
As far as we know, these are the first results towards a mod $p$ analogue. 
In the case of root number $1$ (resp. $-1$), these article obtain quantitative results towards the number of class group characters with the corresponding Rankin--Selberg central L-value (resp. central derivative of the L-function) non-vanishing. The approach in these articles seems to have more analytic/ ergodic flavour.  
In Burungale--Tian (\cite{BT}), non-quantitative results are obtained in the case of root number $-1$ based on an algebro-geometric approach which relies on Andr\'e--Oort conjecture. 
It seems instructive to compare our current approach with the ones in \cite{MV1}, \cite{MV}, \cite{Te}, \cite{BH} and \cite{BT}. Here we only mention that Michel's equidistribution also plays an underlying role for the characteristic zero non-vanishing result due to Michel--Venkatesh (\cite{MV}). It seems striking that the characteristic zero and mod $p$ non-vanishing share a key ingredient in the proof. 

In the Iwasawa setup, the non-vanishing goes back to Cornut (\cite{C}) and Vatsal (\cite{V}, \cite{V1}). 
From the consideration of approach to non-vanishing, there seem to be fundamental differences in Cornut--Vatsal and our horizontal setup. To begin with, Iwasawa algebra plays a pivotal role in Cornut--Vatsal to reduce the generic non-vanishing to equidstribution of a class of special points on a definite Shimura set. The equidsitribution is proven based on Ratner's theory. The equidistribution also crucially relies on Iwasawa algebra as it leads to an action of a $p$-adic Lie group on the class of special points. In the horizontal setup, there does not seem to be a naive analogue of the Iwasawa algebra as the number fields in consideration vary horizontally. 
Moreover, the special points do not seem to admit a group action.

The article is organised as follows.
In \S2, we consider the root number $+1$ case and analyse certain aspects of toric periods.
In \S2.1, we introduce the setup and fix a definite quaternion algebra $B$.
In \S2.2, we consider the horizontal mod $p$ non-vanishing of toric periods of a toric form on $B$.
In \S2.3, we consider the horizontal variation of the constants appearing in the Euler system method to bound the Tate--Shafarevich group.
In \S2.4, we consider the horizontal variation of the constants appearing in the explicit Waldspurger formula alluded to above.
In \S3, we consider the root number $-1$ case and analyse certain aspects of Heegner points.
In \S3.1, we introduce the setup and fix an indefinite quaternion algebra $B$.
In \S3.2, we recall generalities regarding CM points on Shimura curves arising from $B$.
In \S3.3, we consider the horizontal mod $p$ non-vanishing of toric periods of a weight zero $p$-adic modular form $g$ alluded to above.
In \S3.4, we recall the $p$-adic Waldspurger formula.
In \S3.5, we consider the horizontal variation of the constants appearing in the Euler system method to bound the Tate--Shafarevich group.
In \S3.6, we consider the horizontal variation of the constants appearing in the explicit Gross--Zagier formula alluded to above.
In \S4, we prove the main results based on \S2 and \S3 
along with the existence of suitable torus embeddings into the underlying quaternion algebras.

\section*{Acknowledgement}
We thank Li Cai for his assistance and Hae-Sang Sun for stimulating conversations regarding horizontal non-vanishing. 
We also thank Henri Darmon, Najmuddin Fakhruddin, Dimitar Jetchev, Mahesh Kakde, Chandrashekhar Khare, Chao Li, Philippe Michel, C.-S. Rajan, Nicolas Templier, Jacques Tilouine, Vinayak Vatsal, Xinyi Yuan, Shou-Wu Zhang and Wei Zhang 
for instructive conversations about the topic.

\section*{Notation} 
We use the following notation unless otherwise stated.

For a finite abelian group $G$, let $\widehat{G}$ denote $\overline{\BQ}^\times$-valued character group of $G$. For a $\BZ$-module $A$, let $\widehat{A}=A \otimes_{\BZ} \widehat{\BZ}$ for $\widehat{\BZ}=\varprojlim_{n} \BZ/n$.

For a number field $L$, let $\cO_L$ be the corresponding integer ring and $D_L$ the discriminant.
Let $L_+$ denote the totally positive elements in $L$.
Let $G_{L}=\Gal(\overline{\BQ}/L)$ denote the absolute Galois group over $L$. Let $\BA_L$ denote the adeles over $L$. For a finite subset $S$ of places in $L$, let $\BA_{L}^{(S)}$ denote the adeles outside $S$ and $\BA_{L,S}$ the $S$-part.
When $L$ equals the rationals or the underlying totally real field, we drop the subscript $L$.
For a $\BQ$-algebra $C$, let $C_{\BA}=C\otimes_{\BQ} \BA$. Let $\widehat{C}^{(S)}$ (resp.
$C_{S}$) denote the part outside $S$ (resp. $S$-part) of
$C_{\BA}$.

For a place $v$ of $F$ for a totally real field $F$ and a quadratic extension $K_{v}/F_{v}$, let $\eta_{v}$ denote the corresponding quadratic character. For a quaternion algebra $B_{v}/F_{v}$, let $\epsilon(B_{v})$
denote the corresponding local invariant.
For an CM quadratic extension $K/F$ and an integral ideal $\fc$ of $F$, let $H_{K,\fc}$ be the ring class field with conductor $\fc$ and
$\Pic_{K/F}^{\fc}$ the relative ring class group with conductor $\fc$.
Let $h_{K}$ (resp. $h_{K,\fc}$) be the relative ideal class number of $K$ (resp. $H_{K,\fc}$) over $F$.

\section{Toric Periods}
In this section, we consider toric periods of a modular form on a definite quaternion algebra over the rationals. In \S 2.1, we introduce the setup. In \S2.2., we prove horizontal mod $p$ non-vanishing of toric periods. 
In \S2.3 and \S2.4, we deduce the consequences of the mod $p$ non-vanishing for Tate--Shafarevich group and analytic Sha, respectively. 

\subsection{Setup}
In this subsection, we introduce the setup.

Let $\phi\in S_2(\Gamma_0(N))$ be a newform with weight $2$ and level $\Gamma_{0}(N)$ for $N \geq 3$.
% and  Neben character $\omega$. Let $\omega$ also denote its associated Hecke character on $\BA^\times$. Let $c\geq 1$ be an integer such that for $\ell |N$, $\ord_\ell \omega_\ell\leq \ord_\ell c$. 
In particular, we consider newforms with trivial central characters. 
Let $L_\phi$ denote the Hecke field corresponding to the newform $\phi$.

Let $B$ be a definite quaternion algebra over $\BQ$ such that there exists an irreducible automorphic  representation $\pi_{B}$ on $B_\BA^\times$ whose Jacquet-Langlands transfer is the automorphic representation of $\GL_2(\BA)$ associated to $\phi$. Recall $B_{\BA}=B \otimes_{\BQ} \BA$

Let $S=\Supp (N \infty)$.  Let $K_{0, S}\subset B_S$ be a $\BQ_S$-subalgebra such that $K_{0, \infty}=\BC$ and $K_{0, v}/\BQ_v$ is semi-simple quadratic. 
As in the introduction, $B_S$ (resp. $\BQ_S$) denotes the $S$-part of $B_\BA$ (resp. $\BA$). 
For any $v\in S$, we say that $v$ is non-split if $K_{0, v}$ is a field and split otherwise. Let
$$ U_{0, S}:=\prod_{v\in S,\ \text{$v$  split}}\CO_{K_{0, v}}^\times \times \prod_{v\in S,\ \text{$v$ non-split}} K_{0, v}^\times.$$

Suppose we are given a finite order character $\chi_{0, S}: U_{0, S}\lra \ov{\BQ}^\times$ with conductor one such that the following holds.
\begin{itemize}
\item[(LC1)]
$
\omega \cdot \chi_{0, S}\big|_{\BQ_S^\times \cap U_{0, S}}=1.
$
\item[(LC2)]
%We also suppose that
$
\epsilon(\phi, \chi_{0,v})\chi_{0,v}\eta_v(-1)=\epsilon(B_v)
$
for all places $v|N \infty$ with the local root number $\epsilon(\phi,\chi_{0,v})$ corresponding to the Rankin--Selberg convolution.
\end{itemize}

Fix a maximal order $R^{(S)}$ of $B_{\BA}^{(S)}\cong M_2(\BA^{(S)})$.
Let $U^{(S)}=R^{(S)\times}$. Note that $U^{(S)}$ is a maximal compact subgroup of
$B_{\BA}^{(S)^{\times}}\cong \GL_2(\BA^{(S)})$.

Let $p$ be an odd prime. 
As in the introduction, we fix an embedding $\iota_{p}:\ov{\BQ}\hookrightarrow \BC_p$. 
Let $v_p$ be a $p$-adic valuation. 
Let $\BF$ be an algebraic closure of $\BF_p$.
%$$
%p\nmid \prod_{\ell | c,\\{split}} (\ell-1) \prod_{\ell |c, \\{non-split}} (\ell+1)
%$$ be a prime.

We introduce underlying imaginary quadratic fields. 
\begin{defn}\label{Theta}
Let $\Theta_{S}$ denote the set of imaginary quadratic fields $K$ such that
\begin{itemize}
\item[(i)] $p\nmid h_{K}$,
\item[(ii)] there exists an embedding $\iota_{K}:K\ra B$ with $K_S=K_{0, S}$, 
\item[(iii)] $\BA_{K}^{(S)}\cap R^{(S)}=\wh{\CO}_K^{(S)}$ under the embedding and 
\item[(iv)] $K \neq \BQ(i), \BQ(\omega)$ for $\omega$ a primitive third root of unity.
\end{itemize}
\end{defn}
There exist infinitely many imaginary quadratic extensions $K/\BQ$ with $K \in \Theta_S$
(\cite{Br}, \cite{W} and \cite{Be}). 
For $K \in \Theta_{S}$, we fix such an embedding $\iota_{K}$.

We introduce underlying Hecke characters over imaginary quadratic fields.
\begin{defn}\label{char}
For each $K\in \Theta$, let $\fX_{K,\chi_{0,S}}$ denote the set of finite order characters $\chi$ over $K$ such that \begin{itemize}
\item[(i)] $\chi|_{\BA^\times} \cdot \omega=1$,
\item[(ii)] $\chi_S=\chi_{0, S}$ via the embedding $\iota_K$,  and
\item[(iii)] $\chi$ is unramified outside $S$.
\end{itemize}
\end{defn}
Note that the conductor of $\chi \in \fX_{K,\chi_{0,S}}$ equals one. Moreover, we have
\begin{itemize}
\item[(RN)] $\epsilon(\phi,\chi)=1$.
\end{itemize}
Here $\epsilon(\phi,\chi)$ denotes the global root number of the Rankin--Selberg convolution corresponding to the pair $(\phi,\chi)$. 

In the rest of the section, we let $\Theta=\Theta_{S}$ and
$\fX_{K} = \fX_{K,\chi_{0,S}}$.

\begin{lem}\label{existence1}
The set $\fX_{K}$ is non-empty for all but finitely many imaginary quadratic fields $K$ with $K\in \Theta$. Moreover, it is a homogenous space for the class group $\Cl_{K}$.
\end{lem}
\begin{proof} 
Note that for all but finitely many imaginary quadratic fields $K$, we have that $\CO_K^\times=\BZ^\times$. We fix such an imaginary quadratic field from now. 
%For such a CM quadratic extension $K/F$, we now show that $\fX_K$ is 
%non-empty. 
%a homogenous space of $\Pic_{K/F}^\fc$.

In view of class field theory, the class group is given by
$\Cl_{K}=\BA_K^\times/K^\times K_\infty^\times \BA^\times U$.
%for $S_0$ the subset of non-split places.
%Suppose that $D_K\neq -3, -4$ so that  $\CO_K^\times=\BZ^\times$.
From the structure of $U$, there exists a non-trivial Hecke character $\epsilon$ of $\BA_K^\times$ with  conductor one, trivial on $K_\infty^\times$ and $\BA^\times$ for all but finitely many imaginary quadratic fields $K$ (\cite{KhKi}). 

Let $\chi_{0, S}'=\chi_{0, S} \cdot \epsilon|_{U_{0, S}}$. Then, there exists a character $\chi'$ of $\BA_K^\times/K^\times K_\infty^\times \BA^\times U$ extending $\chi_{0, S}'$. The character $\chi'\cdot \epsilon^{-1}$ is a desired one. This finishes the proof of first part. 

For $\chi \in \fX_K$, note that $\chi'\epsilon^{-1}\chi^{-1}$ factors through $\Cl_{K}$.

\end{proof}

Let $\CO$ be the ring of integers of the field $\BQ(\phi, \chi_{0, S})$ generated over $\BQ$ by the Hecke eigenvalues of the newform $\phi$ and the values of the local character $\chi_{0, S}$. Let $\fp$ be the prime ideal of $\CO$ corresponding to $\iota_p$. 
Let $\cO_{(\fp)}$ be the localisation of the integer ring $\cO$ at the prime ideal $\fp$.

We have the following existence of toric test vectors. 
\begin{lem}
\label{TV0}
There exists a non-zero $\cO_{(\fp)}$-valued form $f\in \pi_B$ such that the ideal generated by the image of $f$ is prime to $\fp$ and $f$ satisfies the following.
\begin{itemize}
\item[(F1)] The subgroup $U_{0, S}$ acts on $f$ via $\chi_{0, S}$ and
\item[(F2)] $f \in \pi^{U^{(S)}}$.
\end{itemize}
\end{lem}
\begin{proof}
In view of (RN), there exists a non-zero $\ell \in \Hom_{\BA_{K}^{\times}}(\pi_{B} \otimes \chi, \BC)$ (\cite{S}, \cite{T0} and \cite{T}). Here $\BA_{K}^\times$ acts trivially on $\BC$.

From \cite{CST}, we thus have $\varphi \in \pi_{B}^{U^{(S)}}$ such that $\ell(\varphi)\neq 0$.
We further choose $\varphi$ to be $\fp$-integral.
For $t \in U_{0,S}$, note that
$$
\ell(\chi_{0,S}(t)\varphi)=\chi_{0,S}(t)\varphi.
$$
Recall that 
$$\dim_{\BC}\Hom_{K_{\BA}^{\times}}(\pi_{B} \otimes \chi, \BC)=1.$$

We can thus take $f$ to be the automorphic form given by
$$
g \mapsto \int_{U_{0,S}/\BQ_{S}^{\times}\cap U_{0,S}} \chi_{0,S}(t)\varphi(gt) dt .
$$
We further normalise $f$ to be $\fp$-primitive by requiring that the ideal generated by the image of $f$ is prime to $\fp$.%The existence follows from the existence of test vector corresponding to the pair $(\pi,\chi)$.
\end{proof}
Note that $f$ is spherical outside $S$.

%We further normalise $f$ to be $\fp$-primitive in the sense that the ideal generated by the image of $f$ is prime to $\fp$.

Let $U$ be an open compact subgroup of ${B}_{\BA}^{(\infty),\times}$ such that 
$U=U_{S}U^{(S)}$ with $U_{S}\subset B_{S}^\times$ and $U_{S}=\prod_{v \in S} U_{v}$ such that\begin{itemize}
\item[(L1)] $f\in \pi_{B}^U$ and
\item[(L2)] $\CO_{0, v}^\times \subset U_v$ for all $v\in S \backslash \{\infty\}$. 
\end{itemize}

Let $X_{U}$ be the corresponding Shimura set of level
$U$ given by
$$
X_{U}=B^{\times} \backslash  B_{\BA}^{(\infty),\times} / U.
$$

Let
$$
\overline{f}: X_{U} \rightarrow \BF
$$
be the reduction of $f$ modulo $\fp$. 

 Let $\fp_{0}$ be the prime above $p$ in the Hecke field $L_\phi$ determined via the embedding $\iota_p$. 
Let $\rho_{\phi,\fp_{0}}: G_{\BQ} \ra \GL_{2}(L_{\phi,\fp_{0}})$ be the corresponding $p$-adic Galois representation. 
Let $\overline{\rho}_{\phi,\fp_{0}}$ be the reduction modulo $\fp_{0}$.

We have the following non-constancy regarding the reduction $\ov{f}$. 

\begin{lem}
\label{non-constancy}
Assume that the mod $p$ Galois representation $\overline{\rho}_{\phi,\fp_{0}}$ associated to the newform $\phi$ is absolutely irreducible.
%(\ref{irr}).
Let $f$ be a toric form as above with non-zero reduction modulo $\fp$.
Then, the mod $\fp$ reduction $\overline{f}$ is non-constant.
% modulo $\fp$.
\end{lem}
\begin{proof}
Suppose it is not the case. From the definition of $f$, it is a Hecke eigenform outside the Hecke operators arising from primes in $S$. Moreover, the Hecke eigensystem is the same as that of $\phi$.
From the constancy, the eigensystem outside $S$ is Eisenstein. This contradicts the irreducibility assumption.

\end{proof}

We identify $\BA_{K}^\times$ as a subgroup of $B_{\BA}^\times$ under the embedding $\iota_{K}: K \hookrightarrow B$.
As $\iota_{K}(\cO_{K}) \subset U$, the choice of embedding $\iota_{K}$ gives rise to a map
$$
\varphi_{K}:\Cl_{K}\rightarrow X_{U}.
$$
For $\sigma \in \Cl_{K}$, let $x_{\sigma} \in X_{U}$ be the corresponding point on the Shimura set.
This is usually referred as a special point on the Shimura set. In what follows, the map
$\varphi_{K}$ plays an underlying role.

For $\chi \in \widehat{\Cl}_{K}$,
let $P_{f}(\chi)$ be the toric period given by
\begin{equation}
P_{f}(\chi)=\frac{1}{h_{K}}\cdot\sum_{\sigma \in \Cl_{K}} \chi(\sigma)^{-1}f(x_{\sigma}).
\end{equation}
%Here $x_{\sigma}\in\widehat{K}^\times$ is a representative of $\sigma$.
As $f$ is $\fp$-integral and $p\nmid h_{K}$, the period is $\fp$-integral.
More precisely, $P_{f}(\chi)\in \cO_{(\fp)}[\chi]$. 
Here $\cO_{(\fp)}[\chi]$ denotes the extension of $\cO_{(\fp)}$ obtained by adjoining the values of $\chi$. 

The period
possibly depends on the choice of the torus embedding $\iota_K$.

\begin{remark}
(1). For $\chi \in \widehat{\Cl}_{K}$, the non-vanishing of $P_{f}(\chi)$ implies that $\chi \in \fX_{K}$.
This follows from $f$ being $\chi_{0,S}$-toric.
The observation will be used in the proof of Theorem \ref{toric1}. 

(2). We restrict to newforms with trivial central character and unramified Hecke characters over the imaginary quadratic fields for simplicity. 
\end{remark}

\subsection{Non-vanishing}
In this subsection, we prove the horizontal mod $p$ non-vanishing of toric periods of a modular form on a definite quaternion algebra.

Let the notation and assumptions be as in \S2.1. In particular, $\phi$ is a weight two newform and $L_\phi$ the corresponding Hecke field. Let $\fp_{0}$ be the prime above $p$ in $L_\phi$ determined via the embedding $\iota_p$. 
Recall that $\rho_{\phi,\fp_{0}}: G_{\BQ} \ra \GL_{2}(L_{\phi,\fp_{0}})$ denotes the corresponding $p$-adic Galois representation. 
Moreover, $\overline{\rho}_{\phi,\fp_{0}}$ denotes the reduction modulo $\fp_{0}$.

Our main result regarding the mod $p$ non-vanishing of toric periods is the following.
\begin{thm}
\label{toric1}
Let $\phi \in S_{2}(\Gamma_{0}(N))$ be a newform.
Let $p\nmid N$ be a prime such that
\begin{itemize}
\item[(irr)] the mod $p$ Galois representation $\ov{\rho}_{\phi, \fp_{0}}$ is absolutely irreducible.
\end{itemize}
Let $f$ be a toric test vector on a definite quaternion algebra as above.
Let $\Theta=\Theta_S$ be the set of imaginary quadratic extensions as in Definition \ref{Theta}. 
For $K \in \Theta$, let $\fX_K$ be the set of finite order Hecke character over $K$ as in Definition \ref{char}. For $\chi \in \fX_K$, let $P_{f}(\chi)$ be the toric period corresponding to the pair $(f,\chi)$. 

For $\epsilon > 0$, we have
$$
 \#\left\{\chi\in\fX_K\ \Big|\ v_{p}(P_{f}(\chi))=0 \right\}
\gg_{\epsilon} \log(|D_{K}|)^{1-\epsilon}.
 $$
 as $K$ varies in $\Theta$.
 % such that
% \begin{itemize}
% \item []$p \nmid |\Cl_{K}|$.
% \end{itemize}
\end{thm}
\begin{proof}
%We first note that the non-vanishing of $P_{f}(\chi)$ necessarily implies
%$\chi \in \fX_{K}$. This follows from $f$ being $\chi_{0,S}$-toric.

%Let $U$ be an open compact subgroup of $\widehat{B}$ such that $U=U_{S}U^{(S)}$ with $U_{S}\subset B_{S}^\times$ such that $f\in \pi^U$ and $\CO_{c, v}^\times \subset U_v$ for all $v\in S$. Let $X$ be the corresponding Shimura set given by
%$$
%X=B^{\times} \backslash  \widehat{B}^{\times} / U.
%$$
Recall that $f: X_{U} \rightarrow \cO_{(\fp)}$ is a function on the Shimura set $X$ and $\overline{f}$ its mod $\fp$ reduction. In view of the map
$\varphi_{K}:\Cl_{K}\rightarrow X_{U}$ of special points, we may regard $f$ and $\overline{f}$ as functions on the abelian group $\Cl_{K}$.

The approach is based on the Fourier analysis on $\Cl_{K}$ and its relation to the Shimura set $X_{U}$.

%For $\sigma \in \Pic_{K,c}$, the element $x_{\sigma}$ as above is nothing but the corresponding Gross point on $X$.

To begin with, the toric periods $P_{f}(\chi)$ modulo $\fp_{\chi}$ equal the $\chi$-th Fourier coefficients of the function
$$\sigma \mapsto f(x_{\sigma}) \mod{\fp}$$ on the finite abelian group $\Cl_{K}$.
Here $\fp_{\chi}$ (resp. $\fp$) is the prime above $p$ in $L_{\phi,\chi}$ (resp. $L_{\phi}'$) determined via the embedding $\iota_p$.
Here $L_{\phi}'$ denotes the Hecke field corresponding to the pair $(\phi,\chi_{0,S})$ and $L_{\phi,\chi}$ its extension obtained by adjoining the values of $\chi$.

In what follows, we thus consider Fourier analysis of the function
$$\overline{f}:\Cl_{K}\rightarrow \BF$$
and its variation for $K \in \Theta$ with $\BF$ an algebraic closure of $\BF_p$. 
As $p \nmid h_{K}$, the Fourier inversion works as usual.

Let us first consider the non-vanishing of at least one twist.
In view of the Fourier inversion, it suffices to show that
\begin{equation}
\label{pCM}
%v_{p}(f(x_{\sigma}))= 0
\overline{f}(x_{\sigma})\neq 0
\end{equation}
for at least one $\sigma \in \Cl_{K}$ for all but finitely many $K$. 

By a result of Iwaniec (\cite{I}), the special points
$$
 \left\{\ x_{\sigma} \ \Big|  \sigma \in \Cl_{K} \right\} \subset X_{U}
$$
become equidistributed on $X_{U}$ with respect to the probability measure on the finite set $X_U$ 
%$\nu/\nu(X)$
as $|D_{K}| \rightarrow \infty$. 
Strictly speaking, the equidistribution is a consequence of Iwanicec's main result in \cite{I} as noticed in Michel (\cite{M}).

In view of the $p$-primitivity of $f$ (Lemma \ref{TV0}),
%(\ref{p-optimal}),
the equidistribution readily implies (\ref{pCM}).
%In fact, the surjectivity of $\varphi_{K}$ for a sufficiently large discriminant suffices.

We now consider the growth in the number of non-vanishing twists.
%In what follows, we consider $f$ as a mod $p$ modular form on the Shimura set $X$.
Suppose that there exists a positive integer $\ell_{K}$ such that exactly $\ell_{K}\geq 1$ of the toric periods
$P_{f}(\chi)$
are non-vanishing modulo $\fp_\chi$, for $K$ with sufficiently large discriminant. Let
$$
\Xi_{K}=\{ \chi_{K,1}, ..., \chi_{K,\ell_{K}}\}
$$
 be the set of the non-vanishing modulo $\fp_\chi$ twists.

Note that
$$
\overline{P_{f}(\chi)} \neq 0 \iff \overline{P_{f}(\chi^{\sigma})}\neq 0
$$
for $\sigma \in \Gal(\BF/\BF_{p}(\overline{\phi},\overline{\chi_{0,S}}))$. 
Here $``\ov{\cdot}"$ denotes the reduction modulo a prime ideal above $p$ determined via the embedding $\iota_p$. 
As $\Im(\overline{f}) \subset \BF_{p}(\overline{\phi},\overline{\chi_{0,S}}))$, the above version of reciprocity law is evident in the setup.
We conclude that $\Xi_{K}$ is invariant under the action of $\Gal(\BF/\BF_{p}(\overline{\phi},\overline{\chi_{0,S}}))$.

Let $C_{K} \subset \widehat{\Cl}_{K}$ the subgroup generated by these characters and let $H_{K} \subset \Cl_{K}$ be the orthogonal complement of $C_{K}$. The orthogonal complement is with respect to the pairing on $\widehat{\Cl}_{K}$ and its character group.
Recall that $|H_{K}|=|\Cl_{K}|/|C_{K}|$.

In view of the Fourier inversion on $\Cl_{K}$ and the hypothesis on the periods, we have
\begin{equation}
\label{inversion1}
\overline{f}(x_{\sigma})=\sum_{i=1}^{\ell_{K}} c_{K,i}\chi_{K, i}(\sigma).
\end{equation}
Here $c_{K,i}\in \BF$ are the Fourier coefficients possibly dependent on $K$ and determined by
the pairs $(\overline{f},\chi_{K,i})$.
From the above formula on $H_{K}$ (\ref{inversion1}), we note that $f$ is constant on $H_{K}$. 

We now recall a refinement of Iwaniec's equidistribution. By a result of Michel (\cite[\S6]{M}), the special points
$$
 \left\{\ x_{\sigma} \ \Big|  \sigma \in H'_{K}, [\Cl_{K}: H'_{K}]\leq |D_{K}|^{\delta}\right\} \subset X_{U}
$$
become equidistributed on $X_{U}$ with respect to the probability measure
%$\nu/\nu(X)$
as $|D_{K}| \rightarrow \infty$ for an absolute constant $\delta > 0$ (remark below).
%$\delta=1/2115$.

It thus follows that
$$|C_{K}| \gg |D_{K}|^{\delta}.$$
Otherwise, this contradicts the non-constancy in Lemma \ref{non-constancy}.

From the Galois-stability of $\Xi_{K}$ and an elementary Lemma \ref{stability} below, we conclude
$$
\ell_{K}\gg_{\epsilon} \log(|\Cl_{K}|)^{1-\epsilon}
$$
for any $\epsilon > 0$.

This finishes the proof.
\end{proof}
\noindent
\begin{remark}
(1). In \cite[\S6]{M}, the equidistribution is proven for the special points arising from the maximal order $\cO_K$. In this case, $\delta$ can be taken to be $1/2115$. The hypothesis seems to be made for simplicity and partly arises from
%choice of torus embeddings and
the availability of explicit Waldspurger formula at the time. In \cite{CST}, the formula has been established unconditionally.
%The hypothesis on $c$ is for simplicity and can be removed directly following \cite{M}.
%The approach enables to establish that the periods $P_{f}(\chi)$ are often $\fp$-units with $\fp|p$ the prime in $\BQ(\phi,\chi)$ determined via the embeddings. For simultaneous non-vanishing corresponding to all the primes above $p$, perhaps a different approach is necessary.

(2). The approach only requires surjectivity of the map $$\varphi_{K}: H_{K} \rightarrow X_{U}$$
with $H_{K}$ as in the proof for a sufficiently large discriminant. 
%The surjectivity itself follows from the surjectivity in the case $c=1$.
% and the approach simplifies slightly.
%In particular, equidistribution of the special points is stronger than what is used.

(3). Under the assumptions in Theorem \ref{toric2}, it seems tempting to predict
$$
 \#\left\{\chi\in\fX_{K}\ \Big|\ v_{p}(P_{f}(\chi))=0 \right\}
\gg_{\epsilon} |\Cl_{K}|^{1-\epsilon}
 $$
for any $\epsilon > 0$. 

 (4). For related equidistribution results on a quaternionic Shimura variety, we refer to \cite{Zh}.
\end{remark}
The following elementary lemma is used in the above proof.
\begin{lem}
\label{stability}
Let $\BF_{q}$ be a finite field with characteristic $p$.
Let $G$ be a finite abelian group and $\widehat{G}$ the $\BF^\times$-valued character group.
Suppose that $p \nmid |G|$.
Let $S \subset \widehat{G}$ be a Galois-stable subset i.e.
\begin{itemize}
\item[(St)] If $\chi \in S$, then the Galois conjugates $\chi^{\sigma} \in S$ for any $\sigma \in \Gal(\overline{\BF}/\BF_{q})$.
\end{itemize}

Let $\widehat{H_{S}} \subset \widehat{G}$ be the subgroup generated by $S$.  Consider a sequence $(G_{i},S_{i})$ as above with $|G_{i}| \rightarrow \infty$. Suppose that there exists
$\delta$ with $ 0 < \delta < 1$ such that $|\widehat{H_{S_{i}}}| \gg |G_{i}|^{\delta}$. For $\epsilon >0$, we then have
$$
|S_{i}| \gg_{\epsilon} \ln(|G_{i}|)^{1-\epsilon}.
$$
\end{lem}
\begin{proof}
As $H_S$ is abelian, say
$$
H_{S} \simeq \BZ/(n_{1}) \times \BZ/(n_{2}) \times ... \times \BZ/(n_{d})
$$
with $n_{1} \geq n_{2} \geq ... \geq n_{d}$.

It follows that there exists an element $\chi_{1}$ in the generating set $S$ with order $n_{1}$.
The number of Galois conjugates of $\chi_{1}$ over $\BF_{q}$ equals the degree $[\BF_{q}(\mu_{n_{1}}):\BF_{q}]$. 
Here $\mu_l$ denotes a primitive $l^{th}$ root of unity over $\BF_p$. 
By duality, the subgroup generated by $\chi_{1}$ corresponds to a quotient
$H_{S}\rightarrow C_{1}$. Note that
$$
\widehat{H_{S}}/\widehat{C_{1}} \simeq \BZ/(n_{2}) \times ... \times \BZ/(n_{d}).
$$
We conclude that there exists an element $\chi_{2}$ in the generating set $S$ with order $n_{2}$.

Repeating the process, we have
$$
|S| \geq \sum_{i} [\BF_{q}(\mu_{n_{i}}):\BF_{q}]\geq \sum_{i} \frac{\varphi(n_{i})}{(n_{i},q-1)}.
$$

As
$$
\sum_{i} \frac{\varphi(n_{i})}{(n_{i},q-1)} \gg_{\epsilon} \bigg{(}\sum_{i} \ln(n_{i})\bigg{)}^{1-\epsilon}
$$
for any $\epsilon > 0$, this finishes the proof.
The last estimate follows from \cite{R}.
\end{proof}
\begin{remark}
(1). Considering arbitrary finite products of $\BZ/(2)$, the estimate is optimal in general.

(2). For a sequence $(G_{i})_{i}$ as in Lemma \ref{stability} with the number of prime factors $\nu(|G_{i}|)$ being bounded, the proof shows that the lower bound for a minimal Galois-stable subset of the character group can be improved significantly.  
\end{remark}

\subsection{Tate--Shafarevich Groups I}
In this subsection, we describe bounds for the Tate--Shafarevich group in the analytic rank zero case in terms of a toric period based on an analysis Euler system of toric periods due to Bertolini--Darmon (\cite{BD}) as refined by Nekov\'a\v{r} (\cite{N0}). 

For convenience, let us briefly recall the Birch and Swinnerton-Dyer (BSD) conjecture.

 In the beginning of this subsection, let $A$ be an abelian variety over a number field $K$ such that $\cO_{L} \subset \End(A)$ for a number field $L$.
 Let $A^\vee$ be the dual abelian variety.
Let $L(s, A)$ be the L-series arising from the Tate module of $A$ at the finite primes of
$\CO=\CO_L$.
In particular, it is without Euler factor at infinite places.

 For a torsion $\CO$-module $M$, let $\Fitt_{\cO}(M)$ denote the Fitting ideal of $M$.

We have the following fundamental

\begin{conj}[BSD]\label{BSD}
Let $L(s,A)$ be the L-series for an abelian variety $A$
%Serre-tensor $A_{\chi}$
over a number field $K$ such that $\cO=\cO_{L} \subset \End(A)$ for a number field $L$. Then, the following holds.

(1). The $L$-series has analytic continuation to $\BC$.

(2). We have $\ord_{s=1}L(s,A)=\rank_{\cO}A(K)$.

(3). Let
$L^*(1, A)\in L\otimes_\BQ \BC$ be the leading coefficient of the L-series.
Then,
$$\CL(A):=\frac{L^*(1, A)}{\Omega(A/K) R(A/K)} \cdot \prod_{v} \Fitt_{\cO}(c_v(A))^{-1} \cdot
\Fitt_{\cO}(A(K)_\tor) \Fitt_{\cO}(A^\vee(K)_\tor)\subset L\otimes \BC$$
 is a fractional ideal of $\CO$ for $r=\ord_{s=1}L(s,A)=\rank_{\cO}A(K)$,
 period $\Omega(A/K)$, regulator $R(A/K)$ (\cite{BCW}) and Tamagawa number $c_{v}$.
 Moreover,
$$\CL(A)=\Fitt_{\cO}(\Sha(A/K)).$$

\end{conj}
For a prime $\fp$ of $\cO_L$, we refer to the asserted equality in (3) for the $\fp$-parts as the $\fp$-part of the BSD (BSD$_\fp$).

We now consider our setup. Let the notation be as in \S2.2. 
In particular, $L_\phi$ denotes the Hecke field of the newform $\phi\in S_2(\Gamma_0(N))$.
%and $\CO_{L_\phi}$ its ring of integer.
Let $\pi$ the cuspidal automorphic representation of $\GL_{2}(\BA)$ corresponding to the newform 
$\phi$ so that we have the equality
$$L(s,\pi)=L(s,\phi)$$
of complex L-functions. 
Let $A$ be a $\GL_2$-type abelian variety  over $\BQ$ associated to  $\phi$ such that $\CO_{L_\phi}\subset \End (A)$.

%Let $c\geq 1$ be a positive integer such that for each $\ell |N$, we have
%$\ord_\ell(c)\geq \cond(\omega_\ell)$ for the local component $\omega_\ell$ of $\omega$.

To state the results, we introduce some notions. Let us begin with 
some results on the image of a $p$-adic Galois representation associated to a modular $\GL_2$-type abelian variety over the rationals 
(\cite[App. B]{N0}).
%For simplicity of notation, let us recall some results on the Galois image in the non-CM case.
For simplicity, we suppose that $A$ does not have CM.  

Let $\Gamma \subset \Aut(L_{\phi}/\BQ)$ be the group of inner twists of $\pi$.
Recall that an inner twist of $\pi$ is a pair $(\sigma,\chi)$ with $\sigma$ an embedding of $L_{\phi}$ and $\chi$ a Dirichlet character such that there exists an isomorphism
$$^{\sigma}\pi \simeq \pi \otimes \chi$$
with $^{\sigma}\pi$ the $\sigma$-conjugate. 
For such a $\sigma$, we have a unique Dirichlet character $\chi_{\sigma}$ such that the above isomorphism holds (\cite[B.3]{N}). 
Let $L^{\Gamma}$ be the fixed subfield of $L_{\phi}$ corresponding to $\Gamma$.
For $\sigma \in \Gamma$, we thus have an inner twist arising from $\chi_{\sigma}$.
Let $\BQ_{\Gamma}/\BQ$ be the extension corresponding to $\bigcap_{\sigma \in \Gamma}\ker(\chi_{\sigma})$.
Let $\frak{p}_{\phi}|p$ be a prime in $L_{\phi}$ and $\frak{p}_{\Gamma}$ the corresponding prime in $L^{\Gamma}$.

From the $\frak{p}$-adic Tate-module of $A$, we have the $\frak{p}_{\phi}$-adic Galois representation $\rho_{\frak{p}_{\phi}}: G_{\BQ} \rightarrow \GL_{2}(O_{L_{\phi},\frak{p}_{\phi}})$. It induces a Galois representation
$$\rho_{\fp_{\phi}}:  G_{\BQ_{\Gamma}}\lra \GL_{2}(O_{L^{\Gamma},\frak{p}_{\Gamma}}).$$
In view of the work of Ribet (\cite{Ri}) and Momose (\cite{Mo}), 
$$
\rho_{\frak{p}_{\phi}}(G_{\BQ_{\Gamma}}) \subseteq H_{\Gamma}:=\bigg{\{}x \in \GL_{2}(O_{L^{\Gamma},\frak{p}_{\Gamma}}) \bigg{|} \det(x) \in \BZ_{p}^{\times} \bigg{\}}
$$
is an open subgroup (for example,  see \cite[B.5.2]{N}). Moreover, the equality holds for all but finitely many primes $\frak{p}$.

We say that
$\rho_{\frak{p}_{\phi}}$ has maximal image if the equality holds in the above inclusion.
Here we only mention that $\rho_{\fp_{\phi}}$ is maximal for all but finitely many primes $\fp$.
If $\rho_{\fp_{\phi}}$ is maximal, then we evidently have
$$
\GL_{2}(\BZ_{p}) \subset \Im(\rho_{\frak{p}_{\phi}}).
$$

%By solution of Serre's conjecture, $A$ is modular. Let $\phi \in S_{2}(\Gamma_{0}(N))$ be the corresponding newform and $\pi$ the cuspidal automorphic representation of $\GL_{2}(\BA)$ generated by $\phi$.
For $\ell |N$, let $\pi_\ell $ be the corresponding local representation of $\GL_{2}(\BQ_{\ell})$.
Let $d_\ell\geq 1$ be such that $d_\ell \BZ$ is generated by $[M_\ell: \BQ_\ell]$ where $M_{\ell}$ varies over extensions of $\BQ_{\ell}$ such that the base change of $\pi_{\ell}$ to $M_{\ell}$ is either a spherical or a Steinberg representation of $\GL_{2}(M_{\ell})$. Note that, $d_\ell=1$ if $\ell \| N$.

Let $B$ be a definite quaternion algebra and $\chi\in \fX_K$ a character over an imaginary quadratic field $K$ as in \S2.1. Recall that $f$ is a toric form on $B_{\BA}^\times$ and
$U \subset (B_{\BA}^{(\infty)})^{\times}$
the toric level corresponding to $f$. Let $I$ be an ideal of $\cO_{L_{\phi}}$ given by
$$
I=\big{\langle}\ell+1-a_{\ell}\big{|} {\text{$\ell$ prime}}, \ell \nmid N, \ell \in
\ker(\widehat{\BQ}^{\times} \rightarrow \BQ^{\times}_{+}\backslash \widehat{\BQ}^{\times}/Nrd(U))
\big{\rangle}. 
$$
Here $a_{\ell}$ denotes the $T_{\ell}$-eigenvalue of the newform $\phi$ and $Nrd$ the reduced norm map.

In the rest of this subsection, let $L=L_{\phi,\chi}$. Let $A_\chi=A_K\otimes_{\CO_{L_\phi}} \CO_L$ be the Serre tensor where the absolute Galois group $G_K$ acts on $\CO_L$ via $\chi$. In regards to certain aspects of the arithmetic corresponding the pair $(A,\chi)$, it turns out that $A_{\chi}$ is a relevant abelian variety.

The main result of this subsection regarding the Tate--Shafarevich group of the Serre tensor $A_{\chi}$ is the following.
 \begin{prop} \label{BD}
 Let $\phi \in S_{2}(\Gamma_{0}(N))$ be a non-CM newform.
 Let $K/\BQ$ be an imaginary quadratic extension and $\chi$ an unramified finite order Hecke character over $K$ as above.
 Let $p$ be a prime and $\fp_\phi$ (resp. $\fp$) the prime above $p$ of the coefficient fields $L_\phi$ (resp. $L$) corresponding to the fixed embedding $\iota_p$. Let $A$ be an abelian variety corresponding to $\phi$ such that $\cO_{L_{\phi}} \subset \End(A)$. For $\ell |N$, let $c_{\ell}$ be the corresponding Tamagawa number. Suppose that the following holds.
 \begin{enumerate}
 \item The $p$-adic Galois representation $\rho_{\fp_\phi}$ has maximal image,
 \item $p\nmid 6N \cdot (\prod_\ell c_\ell d_{\ell})\cdot I \cdot h_{K}$ and 
 \item The toric period $P_f(\chi)$ corresponding to the pair $(f,\chi)$ is a $\fp$-unit with $f$ the $p$-primitive toric form on a definite quaternion algebra as in last subsection.
 \end{enumerate}
Then, 
$$\Sha(A_\chi/K)[\fp^\infty]=0 \qquad \text{for $|D_K|\gg 0$}.$$ 
As $p\nmid h_{K}$, we also have
$$
\Sha(A/H_{K})^\chi[\fp^\infty]\cong \Sha(A_\chi/K)[\fp^\infty]=0 \qquad \text{for $|D_K|\gg 0$}.
$$

 \end{prop}
\begin{proof}
When the toric period $P_{f}(\chi)$ does not vanish, the Bertolini--Darmon method to bound
$\Sha(A_\chi/K)[\fp^\infty]$ shows that
$$
\fp^{C_{\chi}+v_{p}(P_{f}(\chi))}\cdot \Sha(A_\chi/K)[\fp^\infty]=0
$$
for a non-negative integer $C_{\chi}$ (\cite[2.9.6]{N}).
The constant $C_{\chi}$ possibly depends on $p$.

Under the hypotheses (1) and (2), we show that
$$
C_{\chi}=0
$$
for all $\chi \in \fX_{K}$ as $|D_{K}|\gg 0$.

In the following exposition, we suppose some familiarity with \cite[\S2]{N0}.

In view of the description of $C_{\chi}$, we consider the following.
\begin{itemize}
\item[(i)] As the Galois representation $\rho_{A, \fp_\phi}$ has maximal image, 
only finitely many imaginary quadratic extensions are contained in $\BQ(A[\mathfrak{p}^\infty])$. 
Let $K$ be an imaginary quadratic field with $|D_{K}| \gg 0$. 

For $p \geq 5$, let $\matrixx{\lambda_{1}}{}{}{\lambda_{2}} \in \GL_{2}(\BZ_{p})$ such that
$\lambda_{1}=\pm 1$ and $\lambda_{2}^{2}-1$ is a $p$-adic unit. Since the image is maximal, there exists
$g_{K} \in G_{\BQ}$ with trivial action on $\BQ_{\chi}$ such that 
\begin{itemize}
\item $\rho_{\fp_{\phi}}(g_{K})=\matrixx{\lambda_{1}}{}{}{\lambda_{2}}$.
\item $g_{K}$ does not act trivially on $K$. 
%As only finitely many imaginary quadratic extensions are contained in $\BQ(A_{\chi}[\mathfrak{p}^\infty])$, it follows that $g$ does not act trivially on $K$ for $|D_{K}|\gg 0$.
\end{itemize}

In the notation of \cite[\S2]{N}, we thus have $C_{2}=0$ for $|D_{K}| \gg 0$.

\item[(ii)] We now show that
$$
\Im(O_{L,\mathfrak{p}}[G_{\BQ}]\rightarrow \End_{O_{L,\mathfrak{p}}}(A_{\chi}[\fp^{\infty}]))=
M_{2}(O_{L,\mathfrak{p}}).
$$

As $A_{\chi}[\fp]$ is absolutely irreducible $G_{\BQ}$-module, by Burnside's lemma 
$$
\Im(O_{L,\mathfrak{p}}[G_{\BQ}]\rightarrow \End_{O_{L,\mathfrak{p}}}(A_{\chi}[\fp]))=
\End_{O_{L,\mathfrak{p}}}(A_{\chi}[\fp]). 
$$
The assertion thus readily follows from Nakayama lemma. 

In the notation of \cite[\S2]{N}, we thus have $C_{4}=0$.
%We have $O_{L,\mathfrak{p}}[H_{\Gamma}]\subset \Im(O_{L,\mathfrak{p}}[G_{\BQ}]\rightarrow \End_{O_{L,\mathfrak{p}}}(T))$ for $T=A_{\chi}[\fp^{\infty}]$.
%It suffices to show that
%$$
%\bigg{\{}\matrixx{1}{0}{0}{0}, \matrixx{0}{1}{0}{0}, \matrixx{0}{0}{1}{0}, \matrixx{0}{0}{0}{1}\bigg{\}}
%\subset O_{L,\mathfrak{p}}[H_{\Gamma}].
%$$
%Let $\alpha,\beta \in \BZ_{p}^\times$ with $\alpha - \beta \in \BZ_{p}^\times$. Note that
%$\matrixx{\alpha}{0}{0}{1}, \matrixx{\beta}{0}{0}{1}\in H_{\Gamma}$. Thus,
%$\matrixx{1}{0}{0}{0} \in O_{L,\mathfrak{p}}[H_{\Gamma}]$. A similar argument applies for other elementary matrices.

\item[(iii)] We show that
$$
\Im(O_{L,\mathfrak{p}}[G_{H_{\chi}}]\rightarrow \End_{O_{L,\mathfrak{p}}}(A_{\chi}[\fp^{\infty}]))=
M_{2}(O_{L,\mathfrak{p}})
$$
for $p>3$ and $|D_{K}|\gg 0$. 
Here $H_{\chi}$ denotes the extension of $K$ cut out by the finite order Hecke character $\chi$ over $K$. 

In fact, we show the following. Suppose that $K$ is not a subfield of $\BQ(A_{\chi}[\mathfrak{p}^{\infty}])$.
Then,
$$
O_{L,\mathfrak{p}}[\rho_{\mathfrak{p}}(G_{H_{\chi}})]=M_{2}(O_{L,\mathfrak{p}}).
$$

Let $M$ denote the intersection of $\BQ(A_{\chi}[\mathfrak{p}^{\infty}])$ and $H_{\chi}$.
As $K$ is not a subfield of $\BQ(A_{\chi}[\mathfrak{p}^{\infty}])$, the extension $M/\BQ$ is abelian.
Moreover, the corresponding Galois group has exponent $2$ as $\Gal(M/\BQ)\simeq \Gal(MK/K)$
and the latter extension is both cyclotomic and anticyclotomic. It follows that
$$
O_{L,\mathfrak{p}}[H_{\Gamma}^{2}] \subset O_{L,\mathfrak{p}}[\rho_{\mathfrak{p}}(G_{H_{\chi}})].
$$
Here $H_{\Gamma}$ is as in the description of maximal image. 
Here and in what follows, for a subgroup $H' \subset \GL_{2}(O_{L,\fp})$, let $H'^{2}$ denotes the subgroup generated by the square of elements in $H'$.

For $p>3$, note that $\BZ_{p}[\GL_{2}(\BZ_{p})^{2}]=\BZ_{p}[\GL_{2}(\BZ_{p})]$. As $\GL_{2}(\BZ_p)\subset H_{\Gamma}$, we conclude that
$$
M_{2}(O_{L,\mathfrak{p}})=O_{L,\mathfrak{p}}[\GL_{2}(\BZ_{p})^{2}]\subset O_{L,\mathfrak{p}}[\rho_{\mathfrak{p}}(G_{H_{\chi}})].
$$

In the notation of \cite[\S2]{N}, we thus have $C_{8}=0$ for $|D_{K}| \gg 0$.
\item[(iv)] 
Let $m$ be a sufficiently large positive integer and $n=m+C_{0}$ for the constant $C_{0}$ in \cite[2.7.9]{N}. Let $H_{n,\chi}=\BQ(T/\mathfrak{p}^{n})H_{\chi}$ for $T=A_{\chi}[\fp^{\infty}]$.

Let $g_{n}\in\Gal(H_{n,\chi}/\BQ)$ be the restriction of $g_{K}$ to $H_{n,\chi}$. 
Here $g_{K}$ is an element as in the paragraph preceeding the vanishing of $C_{2}$.
Let $H_{n,\chi}'$ be the subfield of $H_{n,\chi}$ fixed by $\langle g_{n}^{2}\rangle$.
Let $\Res:H^{1}(H_{\chi},T/\mathfrak{p}^{m})\rightarrow H^{1}(H_{n,\chi},T/\mathfrak{p}^{m})$
and
$\Res':H^{1}(H_{\chi},T/\mathfrak{p}^{m})\rightarrow H^{1}(H_{n,\chi}',T/\mathfrak{p}^{m})$
be the restriction maps.
We show that
$$
\ker(\Res')=\ker(\Res)=0
$$
for $p>3$ and $|D_{K}|\gg 0$.

This is based on the proof of \cite[Prop. 9.1]{G}.
By inflation--restriction, we have
$$
\ker(\Res)=H^{1}(\Gal(H_{n,\chi}/H_{\chi}),A_{\chi}[\mathfrak{p}^m]).
$$
The Galois action gives rise to an embedding $\rho:\Gal(H_{n,\chi}/H_{\chi})\hookrightarrow \GL(A_{\chi}[\mathfrak{p}^n])$. Suppose that $K$ is not a subfield of $\BQ(A_{\chi}[\mathfrak{p}^{\infty}])$ and let $M$ be as in the discussion preceding the vanishing of $C_8$.
As $\Gal(M/\BQ)$ is abelian with exponent two,
$\rho((\Gal(\BQ(A_{\chi}[\mathfrak{p}^{n}])/\BQ))^2)$ is a subgroup of $\rho(\Gal(H_{n,\chi}/H_{\chi}))$, and therefore
$$
Z_{0} \subset Z(\Im(\rho)) \subset \Im(\rho),
$$
where $Z_0$ is the image of $(\BF_{p}^\times)^{2}$ under the Teichmuller lifting. The existence follows from the maximality of the image.
As $p>3$, the subgroup $Z_{0}$ is non-trivial with order prime to $p$. It follows that
$$
H^{i}(Z_{0},A_{\chi}[\mathfrak{p}^{m}])=0
$$
for $i \geq 0$. The spectral sequence
$$
H^{i}(\Gal(H_{n,\chi}/H_{\chi})/\rho^{-1}(Z_{0}), H^{j}(Z_{0},A_{\chi}[\mathfrak{p}^{m}]))
\implies H^{i+j}(\Gal(H_{n,\chi}/H_{\chi}), A_{\chi}[\mathfrak{p}^{m}]))
$$
thus implies the desired vanishing of $\ker(\Res)$. 

Recall that the elements $g_{n}^{2}$ does not belong to the center $Z(\Im(\rho))$. Going modulo $\langle g_{n}^2 \rangle$, the above argument thus also applies for $\ker(\Res')$.

In the notation of \cite[\S2]{N}, we thus have $C_{7}=0$ for $|D_{K}| \gg 0$. 

\item[(v).] In view of our hypothesis $p \nmid \prod_{\ell} c_{\ell}$ and the description of Tamagawa numbers in \cite[Prop. 3]{Kh}, we have $C_{6}=0$ in the notation of \cite[\S2]{N}. 

\item[(vi).] In view of our hypothesis $p \nmid \prod_{\ell} d_{\ell}$, we have $C_{5}=0$
in the notation of \cite[\S2]{N}. 
\end{itemize}

We now recall 
$$
C_{\chi}=
3C_{2}+12C_{4}+C_{5}+C_{6}+C_{7}+C_{8}+v_{p}(I \cdot [H_{\chi}:K])+v_{p}([H_{\chi}:K])
$$
for all $\chi \in \fX_{K}$ (\cite[2.9.6]{N}). 
As $p \nmid h_{K}$, we thus conclude 
$$
C_{\chi}=0
$$
for all $\chi \in \fX_{K}$ such that $|D_{K}| \gg 0$ (\cite[2.9.6]{N}).
This finishes the proof for $\Sha(A_{\chi})$.

In view of \cite{MRS}, the rest follows for $\Sha(A/H_{K})$ from the result for the Serre tensor $A_{\chi}$.
\end{proof}
\begin{remark} Based on an analysis of the Euler system method for toric periods, 
it may be checked that there exists a constant $C$ such that
$$
\# \Sha(A_{\chi})[\fp^{\infty}] \bigg{|} p^{C+v_{p}(P_{f}(\chi))}
$$
for all $\chi \in \fX_{K}$ as $|D_{K}|\gg 0$.
In this sense, the toric periods control the size of the Tate--Shafarevich groups.
\end{remark}

\subsection{Analytic Sha I} 
In this subsection, we describe bounds for analytic Sha in the analytic rank zero case in terms of a toric period based on an analysis of explicit Waldspurger formula due to Cai--Shu--Tian (\cite{CST}).
%In this subsection, we consider horizontal variation of the underlying central L-values. 
The consideration along with the main result in \S2.3 leads us to establish the $p$-part of corresponding Birch and Swinnerton-Dyer conjecture for a class of primes $p$.

We first introduce a definition. 

\begin{defn}
For an abelian variety $A$ over a number field $L$, we refer to the ideal $\CL(A)$ in the formula part of the BSD conjecture for $A$ over $L$ (Conjecture \ref{BSD}) as the analytic Sha. 
\end{defn}

We now introduce our setup. Let the notation and assumptions be as in \S2.3. In particular, $A$ denotes a $\GL_2$-type abelian variety over the rationals corresponding to a weight two newform $\phi \in S_{2}(\Gamma_{0}(N))$ and $\chi$ a finite order Hecke character over an imaginary quadratic field $K$ such that
$$
\epsilon(A,\chi)=1
$$
((RN)).
 Moreover,
$A_{\chi}$ denotes the Serre tensor of $A$ by $\chi$ which is an abelian variety over $K$ with $\cO_{L} \subset \End(A_{\chi})$.

As above, we refer to $\CL(A_{\chi})$ as the analytic Sha of $A_\chi$. 

In this subsection, we consider the conjecture in the analytic rank zero case.
We thus suppose that
 $$L(1, A_\chi)\neq 0.$$

To introduce an explicit Waldspurger formula, we begin with some notation. 
Let $B$ be a definite quaternion algebra as in \S2.1. 
Let $R \subset B$ be an order as in \cite[\S1.1]{CST}. 
Recall that $f$ is a toric form on $B_{\BA}^\times$ and
$U \subset (B_{\BA}^{(\infty)})^{\times}$
the level corresponding to $f$. 
Let $X_U$ be the corresponding Shimura set. 
Let $N=N^{+}N^{-}$ for $N^{+}$ (resp. $N^{-}$) precisely divisible by split (resp. non-split) primes in the extension $K/\BQ$. 
Recall that $S$ denotes the set of prime divisors of $N\infty$.

For any $f_i = \sum_x f_i(x) \cdot x \in \BC[X_{U}]$ with $i=1,2$, we put
   \[ \langle f_1,f_2\rangle = \sum_x w_x^{-1}  f_1(x)f_2(x \tau_B). \]
   Here $w_x$ is the cardinality of $(B^\times \cap g \wh{R}^\times g^{-1})/\{ \pm 1\}$
   with $g\in \wh{B}^\times$ any representative of $x$ and $\tau_B \in N_{\wh{B}^\times}(\wh{R}^\times)$
   such that
   \begin{itemize}
	   \item for $\ell|N^+$,
		 $\tau_{B,\ell} \in N_{B_\ell^\times}(R_\ell^\times) \setminus \BQ_\ell^\times R_\ell^\times$ the Atkin--Lehner operator;
	   \item for $\ell | N^-$, $\tau_{B,\ell} t_\ell = \ov{t_\ell} \tau_{B,\ell}$ for any $t_\ell \in K_\ell^\times$ where
	   $\ov{t_\ell}$ is the Galois conjugate of $t_\ell$.
   	   \item for $\ell \nmid N$, $\tau_{B,\ell} = 1$.
   \end{itemize}
    We consider Hermitian pairing $(\ ,\ )_{\widehat{R}^\times}$ on the space $\BC[X_{U}]$ of $\BC$-valued functions on the Shimura set $X_U$  is given by
    \[ (f_1,f_2)_{\widehat{R}^\times} = \sum_x w_x^{-1}  f_1(x)\ov{f_2(x)} \]
with $\overline{\cdot}$ being the complex conjugaton.

From \cite[Thm. 1.8]{CST}, we have an explicit Waldspurger formula given by 
$$L(1, A_\chi)=2^{-\# \Sigma_D+2}\cdot \frac{8\pi^2 (\phi,\phi)_{\Gamma_0(N)}}{[\CO_{K}^\times :\BZ^\times]^2\sqrt{|D_K|^2}}\cdot \frac{P_{f_{1}}(\chi) P_{f_{2}} (\chi^{-1})}{(f_1, f_2)_{\wh{R}^\times}}.$$ 
Here $\Sigma_{D}=\big{\{}\ell \big{|} \ell | (N,D_{K})\big{\}}$, $f_{1} \in V(\pi_{B},\chi)$ and $f_{2} \in V(\pi_{B},\chi^{-1})$ are non-zero vectors. For the definition of $V(\pi_{B},\chi)$, we refer to \cite[Def. 1.7]{CST}. 
Here we only mention that $f_1$ can be taken to be the toric form $f$ as in \S2.1.

%For the notation, we refer to \cite[\S1]{CST} (recall). We mention the following.

%If we identify $\wt{\pi}$ and $\ov{\pi}$, then we may take $f_2=\ov{f_1}$ in the above formula. As in \cite{BDP1}, there is a test vector $f$ such that
%$$P_\chi(f_1) P_{\chi^{-1}}(f_2)=CP_\chi(f)^2$$
%for some non-zero constant $C$ independent of $\chi$ and $K$.

In the analytic rank zero case, the analytic Sha of $A_\chi$ is thus given by
$$\CL(A_\chi)=
\frac{\pi^2 (\phi,\phi)_{\Gamma_0(N)}}{\sqrt{|D_K|^2}\Omega(A_\chi/K)}\cdot \frac{P_{f_{1}}(\chi) P_{f_{2}} (\chi^{-1})}{(f_1, f_2)_{\wh{R}^\times}}\cdot \frac{\Fitt (A_\chi(K)_\tor) \Fitt (A^\vee_{\chi^{-1}}(K)_\tor)}{2^{\# \Sigma_D -5}\cdot [\CO_K^\times :\BZ^\times]^2\cdot\prod_{v}\Fitt(c_v(A_\chi))}.
$$
%\cdot \prod_{v\in \Sigma} L_v(1, A_\chi).$$
 %By a scalar multiplication,  we may assume $f=f_1+f_1^\bot$ according to the $U_{0, S}$-action with $f_1$ as previous $\fp$-primitive test vector.  Note that $P_\chi(f_1)=P_\chi(f)$ for all $\chi\in \fX^+$.

 We have the following key proposition regarding the variation of the $\mathfrak{p}$-part of analytic Sha in terms of toric periods.

\begin{prop}\label{control0}
Let $\phi \in S_{2}(\Gamma_{0}(N))$ be a non-CM newform and $L_{\phi}$ the corresponding Hecke field.
  %Let $p$ be a prime and $\fp_\phi, \fp$ the primes above $p$ of the coefficient fields $L_\phi$.
 % and $L$ corresponding to the fixed embedding, respectively.
 Let $A$ be an abelian variety corresponding to $\phi$ such that $\cO_{\phi} \subset \End(A)$.
 %Let $c$ be a positive integer.

There exists a finite set $\Sigma_{A}$ of primes of $\CO_\phi$ only dependent on $A$,  such that for any prime $\fp_0\notin \Sigma_{A}$ of $\CO_\phi$ and any $\chi\in \fX$ with $L(1, A_\chi)\neq 0$, the ideal
$$\displaystyle{\frac{\CL(A_\chi)}{P_f(\chi)^2}}$$
is coprime to $\fp_0$.
Here $f$ is the $\chi_{0, S}$-toric vector as in last subsection.
%with $\fp$ the prime above $\fp_0$ in $\cO_{\phi,\chi}$ determined via the embeddings.
 In particular, if $P_f(\chi)$ is a $\fp$-unit, then $\CL(A_\chi)$ is also a $\fp$-unit as long $\fp$ lies above some $\fp_{0} \notin \Sigma_{A}$.
\end{prop}
\begin{proof} 
Let $f_{1}'=f$ be a $\mathfrak{p}$-minimal toric test vector satisfying (F1) and (F2) in \S2.1. Let $J \in B$ be as in \cite[\S2.1]{CH} and $f_{2}'$ the $J$-translate of $f_{1}'$ given by 
$$
f_{2}'(x)=f_{1}'(xJ).
$$
Locally, $f_{2}'$ differs from $f_{1}'$ only at the places dividing $N$. 

Note that 
$$
P_{f_{2}'}(\chi^{-1})=P_{f_{1}'}(\chi).
$$

We now recall the $S$-version of Waldspurger formula in \cite[Thm. 1.9]{CST}.
It states that
$$
L (1, \phi, \chi)=2^{-\#\Sigma_D+2}\cdot C_\infty \cdot \frac{\pair{\phi^0, \ov\phi^0}_{U_0(N)}}{[\cO_{K}^{\times}:\BZ^{\times}]^2 \sqrt{|D_K| }}\cdot \frac{P_{f_{1}'}(\chi) P_{f_{2}'}(\chi^{-1})}
{\pair{f'_1, f'_2}_{\wh{R}^\times}}\cdot \prod_{v\in S\backslash \{\infty\}} \frac{\beta^0(f_{1, v}, f_{2, v})}{ \beta^0(f'_{1, v}, f'_{2, v})}.
$$
Here $\phi^{0}$ is a normalised new vector fixed by level $U_{1}(N)$ and the pairing $\langle \cdot , \cdot \rangle_{U_{0}(N)}$ is as in \cite[\S1.3]{CST}.  
Further, $f_{1}=f \in V(\pi_{B},\chi)$ and $f_{2} \in V(\pi_{B},\chi^{-1})$ are non-zero vectors, the local pairing $\beta^{0}(\cdot, \cdot)$ is as in \cite[Thm. 1.6]{CST} 
and $C_{\infty}$ denotes the constant in \cite[Thm. 1.8]{CST}.
%We refer to \cite[\S1.3]{CST} for the notation (recall).

%We now choose $f_{1}'$ to be $\mathfrak{p}$-minimal toric test vector satisfying (F1) and (F2) in \S2.1. Let $J \in B$ be as in \cite[\S2.1]{CH}. We then have
%$$
%P_{\overline{\chi}}(f_{1,J}')=P_{\chi}(f_{1}')
%$$
%for $f_{1,J}'$ being the $J$-translate of $f_{1}'$.
It follows that the analytic Sha is given by
$$
\CL(A_\chi)=
\frac{C_{\infty} \cdot (\phi,\phi)_{\Gamma_0(N)}}{\sqrt{|D_K|^2}\Omega(A_\chi/K)}\cdot \frac{P_f(\chi)^{2}}{(f_1, f_2)_{\wh{R}^\times}}\cdot \frac{\Fitt (A_\chi(K)_\tor) \Fitt (A^\vee_{\chi^{-1}}(K)_\tor)}{2^{\# \Sigma_D -5}[\cO_{K}^{\times}:\BZ^{\times}]^{2}\cdot\prod_{v}\Fitt(c_v(A_\chi))}\cdot 
%\prod_{v\in \Sigma} L_v(1, A_\chi)
%\cdot 
\prod_{v\in S\backslash\{\infty\}} \frac{\beta^0(f_{1, v}, f_{2, v})}{ \beta^0(f'_{1, v}, f'_{1, J, v})}.
$$

We now discuss horizontal variation of the terms in the right hand side of the above expression except the toric period. It suffices to show that the terms have bounded prime divisors as the pair $(K,\chi)$-varies.

By strong approximation, we can choose
$$
J \in (\prod_{v|N}B_{v}^{\times}) \cdot \GL_{2}(\widehat{\BZ}^{(N)}) \cap B^{\times} .
$$
Since the $S$-type of $K$ is fixed, it thus follows that
$$
\prod_{v\in S\backslash\{\infty\}} \frac{\beta^0(f_{1, v}, f_{2, v})}{ \beta^0(f'_{1, v}, f'_{1, J, v})}
$$
is a fixed rational number as $K$ varies.

%Regarding archimedean periods,
We have
$$
\Omega(A_{\chi}/K)=\Omega(A/K)
$$
as fractional ideals up to ideals dividing the level $N$ (\cite[\S 4.4.1]{BCW}).
%up to primes dividing the greatest common divisor $(N,c)$ (\cite[\S 4.4.1]{BCW}).

%Compare $\Omega(A_\chi/K)$ and $\Omega_A$ by noting that the Neron model of $A_\chi/K_v$ for $v|N$ only depending on $A$ and $\chi_v$,
In view of the definition of the period $\Omega(A/K)$ and the description of the Tamagawa numbers for the Serre tensor $A_{\chi}$, it follows that there exists an nonzero ideal  $\fa_1$ such that
$$\fa_1 \subseteq \frac{C_{\infty} \cdot (\phi,\phi)_{\Gamma_0(N)}}{\sqrt{|D_K|^2}\Omega(A_\chi/K)}\cdot \prod_{v}\Fitt(c_v(A_\chi)).$$
Moreover, the support of prime divisors of $\mathfrak{a}_1$ is bounded as $K$ varies.

From Proposition \ref{torsion}, there exists an nonzero ideal  $\fa_2$ such that
$$
\fa_2 \subseteq \Fitt (A_\chi(K)_\tor) \Fitt (A^\vee_{\chi^{-1}}(K)_\tor)
$$
with support of prime divisors bounded as $K$ varies.
%There is also an nonzero ideal  $\fa_3$ such that
%$\fa_3 P_{\chi}(f_1)\subseteq P_{\chi^{-1}}(f_2).$

In view of the earlier expression for the analytic Sha,
%$\CL(A_{\chi})$,
this finishes the proof.
\end{proof}

The following finiteness is used in the proof. 
\begin{prop}
\label{torsion} 
Let $A$ be a non-CM $\GL_2$-type abelian variety over the rationals. 
Then, the set
$\displaystyle{\bigcup_{K}  A(H_{K})_\tor}$ is finite as $K$ varies imaginary quadratic fields (\cite[Prop. 3.13]{BT}).
\end{prop}

Along with Proposition \ref{BD}, we have the following immediate consequence for the BSD.
\begin{cor}\label{BSD0}
Let $\phi \in S_{2}(\Gamma_{0}(N))$ be a non-CM newform and $L_{\phi}$ the corresponding Hecke field.
  %Let $p$ be a prime and $\fp_\phi, \fp$ the primes above $p$ of the coefficient fields $L_\phi$.
 % and $L$ corresponding to the fixed embedding, respectively.
 Let $A$ be an abelian variety corresponding to $\phi$ such that $\cO_{\phi} \subset \End(A)$.
% Let $c$ be a positive integer.

There exists a finite set $\Sigma_{A}'$ of primes of $\CO_\phi$ only dependent on $A$,  such that the following holds. For any prime $\fp_0\notin \Sigma_{A}'$ of $\CO_\phi$ and for $\chi\in \fX$ with
the toric period $P_{f}(\chi)$ being a $\fp$-unit for $\fp$ above $\fp_{0}$, the $\fp$-part of BSD holds for the Serre tensor $A_{\chi}$ as long as $p \nmid h_{K}$ for $(p)=\fp_{0} \cap \BQ$.

\end{cor}

The exceptional set of primes $\Sigma_{A}$ in Proposition \ref{control0} admits an explicit description under mild hypothesis.

\begin{prop}\label{explicit0}
Let $\phi \in S_{2}(\Gamma_{0}(N))$ be a non-CM newform and $L_{\phi}$ the corresponding Hecke field.
%Let $A$ be an abelian variety associated to a newform $\phi\in S_2(\Gamma_0(N))$ without CM and $\End(A)$ is the ring of integers of $\BQ(\phi)$.
Let $K$ be an imaginary quadratic field and $\chi$ an unramified finite order anticyclotomic Hecke charcter over $K$ as above.
Let $N = N^+ N^-$ where
   $l|N^+$ (resp. $l|N^-$) if and only if $l|N$ and split in $K$ (resp. non-split in $K$).
%of conductor $c$.
Let $p$ be a prime and $\fp_0|p$ a prime of $L_{\phi}$.
Suppose that the following holds.
\begin{itemize}
\item [(1)] $(N, D_{K})=1$,
\item [(2)]  $N^-$ is square-free
   with odd number of prime factors,
   % and $\ord_\ell (N^{-})=1$ for any prime $\ell |N$,
 \item [(3)]$p\nmid 6ND_{K}$ and
  \item [(4)] The Galois representation $\rho_{A, \fp_0}$ has maximal image and the residual representation $\ov{\rho}_{A,\fp_{0}}$ is ramified at $\ell |N^-$ with $\ell^2\equiv 1\mod p$.
\end{itemize}
For all primes $\fp|\fp_0$ of $L=L_{\phi,\chi}$, we then have
$$
\ord_{\fp} \CL(A_\chi)=2\cdot \ord_\fp (P_\chi(f))-2 \cdot \sum_{\ell|N^+, \chi|_{D_\ell} =1}\ord_p(c_\ell).
$$
\end{prop}
\begin{proof}

%[State the explicit form of $S$-version of Waldspurger under the hypotheses.]
We first describe an explicit form of Waldspurger formula under the hypotheses that $N^-$ is square-free with odd number of prime factors.

We have
\[L(1,\phi,\chi) =  \frac{8\pi^2(\phi,\phi)_{\Gamma_0(N)}}
	   {[\CO_K^\times:\BZ^\times]^2\sqrt{|D_{K}|}}\frac{P_f(\chi)^2}{\langle f,f \rangle} \chi^{-1}_{\CN^+}(N^+).\]
	   Here the notation is as in the paragraphs following Conjecture \ref{BSD}. This follows from an analysis of an explicit version in \cite[\S1.3]{CST}.
	   We indicate the sketch below.
	
	   Let $J \in B^\times$ such that $J t = \bar{t} J$ for any $t \in K$ as in the proof of Proposition \ref{control0}.
	   %Then
	   %$P_{\chi^{-1}}(f) = P_\chi(f(\cdot J))$. The equivalence follows from the following equation
	   We have
	   \begin{equation}\label{J-action}
		   \frac{P_{f_{J}}(\chi)}{( f,f )} = \frac{P_{f}(\chi)}{\langle f,f \rangle}\chi^{-1}_{\CN^+}(N^+)
   	   \end{equation}
for $f_{J}$ being the $J$-translate as in the proof of Proposition \ref{control0}.

	   The equality can be deduced as follows. We may assume
	   \[J \in \iota_{N^+}^{-1}\left[\matrixx{N^+}{0}{0}{1}\right] \cdot \tau_B \cdot \wh{R}^\times.\]
	   For $\ell|N^+$, we thus have
	   \[f(tJ_\ell) = \varepsilon_\ell(\pi_\ell)f\left[t^{(\ell)}\iota_\ell^{-1} \matrixx{t_\ell N_\ell}{0}{0}{t_{\bar{\ell}}}\right],
             \quad t \in \wh{K}^\times\]
           and
	   \[P_{f_{J_{\ell}}}(\chi) = \chi_\ell^{-1}(N_\ell)P_{f_{\tau_{B,\ell}}}(\chi).\]
	   As $f_{\tau_{B}}$ and $f$ both in the one-dimensional space $V(\pi_{B},\chi)$,
	   there exists a constant $c \in S^1$ such that
	   \[f_{\tau_{B}} = cf.\]
	   We thus have
	   \[\frac{P_{f_{J}}(\chi)}{( f,f )} = \frac{P_{f_{\tau_B}}(\chi)}{(f,f)}\chi^{-1}_{\CN^+}(N^+)
	   = \frac{cP_{f}(\chi)}{(f,f)}\chi^{-1}_{\CN^+}(N^+) = \frac{P_{f}(\chi)}{\langle f,f\rangle}\chi^{-1}_{\CN^+}(N^+).\]
	   	
%In view of the explicit Waldspurger formula, we have
%$$
%L(1, A_\chi)=\frac{8\pi^2 (\phi, \phi)}{ \sqrt{|D|}} \cdot \frac{P_\chi(f_1)P_{\chi^{-1}}(f_2)}{(f_1, f_2)}.
%$$
It thus follows that
$$\CL(A_\chi)=\frac{8\pi^2 (\phi,\phi)_{\Gamma_0(N)}}{\sqrt{|D_K|^2}\Omega(A_\chi/K)}\cdot 
\frac{P_{f}(\chi)^{2}}{\langle f, f \rangle}\cdot \frac{\Fitt (A_\chi(K)_\tor) \Fitt (A^\vee_{\chi^{-1}}(K)_\tor)}{\prod_{v}\Fitt(c_v(A_\chi))} \cdot \chi^{-1}_{\CN^+}(N^+).$$

We now consider $\fp$-adic valuation of the terms in the right hand side of the above expression.

%Let $\CA/\BZ$ be the Neron-model of $A$ and $\omega\in \Gamma(\BZ, \wedge \Omega_{\CA/\BZ})$ a generator (unique up to $\pm 1$). Let
%$\Omega_A=\int_{A(\BC)} \omega$. Then there is a fraction ideal $\fb$ of $L$ supported on $Dc^2$ such that
%$$\Omega(A_\chi)=\Omega_A \fb.$$

Recall that we have
$$
\Omega(A_{\chi}/K)=\Omega(A/K)
$$
as fractional ideals up to primes dividing $N$.

 In view of the hypothesis (4) on Galois image, both $A_\chi(K)$ and $A^\vee_{\chi^{-1}}(K)$ have no non-trivial $\fp_0$-torsion point.

% Let $(f, f)$ be the pairing defined by Gross and identify $\pi$ and $\wt{\pi}$, then
%$$\frac{P_\chi(f)P_{\chi^{-1}}(f)}{\pair{f, f}}\sim \frac{P_\chi(f)^2}{(f, f)}$$
%up to a $\fp_0$-unit.
From (\cite[(12)]{PW}), we recall that
$$
\ord_{\fp_0} \left(\frac{8\pi^2 (\phi, \phi)}{\Omega_A (f, f)}\right)=\ord_{\fp_0}(\prod_{\ell |N^-} c_\ell).
$$

The proposition thus follows by noting that
$$
\ord_{\fp_0}\left(\prod_v c_{v}(A_\chi)\right)=\ord_p \left( \prod_{\ell |N^-} c_\ell \prod_{\ell |N^+, \chi|_{D_\ell}=1} c_\ell^2\right).
$$
\end{proof}

We have the following immediate consequence.

\begin{cor}\label{BSDE0}
Let $\phi \in S_{2}(\Gamma_{0}(N))$ be a non-CM newform and $L_{\phi}$ the corresponding Hecke field.
%Let $A$ be an abelian variety associated to a newform $\phi\in S_2(\Gamma_0(N))$ without CM and $\End(A)$ is the ring of integers of $\BQ(\phi)$.
Let $K$ be an imaginary quadratic field and $\chi$ an unramified finite order anticyclotomic Hecke charcater over $K$ as above.
Let $N = N^+ N^-$ where
   $l|N^+$ (resp. $l|N^-$) if and only if $l|N$ and split in $K$ (resp. non-split in $K$).
%of conductor $c$.
Let $p$ be a prime and $\fp_0|p$ a prime of $L_{\phi}$.
Suppose that the following holds.
\begin{itemize}
\item [(1)] $(N, D_{K})=1$,
\item [(2)]  $N^-$ is square-free
   with odd number of prime factors,
   % and $\ord_\ell (N^{-})=1$ for any prime $\ell |N$,
 \item [(3)]$p\nmid 6ND_{K}\cdot h_{K}$ and
  \item [(4)] The Galois representation $\rho_{A, \fp_0}$ has maximal image and the residual representation $\ov{\rho}$ is ramified at $\ell |N^-$ with $\ell^2\equiv 1\mod p$.
\end{itemize}
For all primes $\fp|\fp_0$ of $L=L_{\phi}(\chi)$, the $\fp$-part of BSD holds for the Serre tensor $A_{\chi}$ whenever the toric period $P_{f}(\chi)$ is a $\fp$-unit.
\end{cor}

\section{Heegner Points}

In this section, we consider Heegner points arising from a Shimura curve associated to an indefinite quaternion algebra over the rationals. In \S 3.1, we introduce the setup. 
In \S3.2, we recall generalities regarding CM points on a Shimura curve. 
In \S3.-3.4., we prove horizontal $p$-indivisibility of Heegner points. 
In \S3.5 and \S3.6, we deduce the consequences of the $p$-indivisibility for Tate--Shafarevich group and analytic Sha, respectively.

\subsection{Setup}

In this subsection, we introduce the setup regarding Heegner points.

Let $\phi\in S_2(\Gamma_0(N))$ be a newform with weight $2$ and level $\Gamma_{0}(N)$ for $N \geq 3$. In particular, we consider newforms with trivial central character. 
%Let $\phi\in S_2(\Gamma_0(N), \omega)$ be a newform of weight $2$, level $N$ and Neben character $\omega$. Let $\omega$ also denote its associated Hecke character on $\BA^\times$. Let $c\geq 1$ be an integer such that for $\ell |N$, $\ord_\ell \omega_\ell\leq \ord_\ell c$.
Let $L_\phi$ denote the Hecke field corresponding to the newform $\phi$.

Let $B$ be an indefinite quaternion algebra over $\BQ$ such that there is an irreducible automorphic  representation $\pi_{B}$ on $B_\BA^\times$ whose Jacquet-Langlands correspondence is the automorphic representation associated to $\phi$. Recall $B_{\BA}=B\otimes_{\BQ}\BA$.

Let $S=\Supp (N \infty)$.  Let $K_{0, S}\subset B_S$ be a $\BQ_S$-subalgebra such that $K_{0, \infty}=\BC$ and $K_{0, v}/\BQ_v$ is separable quadratic. For any $v\in S$, we say that $v$ is non-split if $K_{0, v}$ is a field and split otherwise. Let
$$ U_{0, S}:=\prod_{v\in S,\ \text{$v$  split}}\CO_{K_{0, v}}^\times \times \prod_{v\in S,\ \text{$v$ non-split}} K_{0, v}^\times.$$

Suppose we are given a finite order character $\chi_{0, S}: U_{0, S}\lra \ov{\BQ}^\times$ with conductor one such that the following holds.
\begin{itemize}
\item[(LC1)]
$
\omega \cdot \chi_{0, S}\big|_{\BQ_S^\times \cap U_{0, S}}=1.
$
\item[(LC2)]
%We also suppose that
$
\epsilon(\phi, \chi_{0,v})\chi_{0,v}\eta_v(-1)=\epsilon(B_v)
$
for all places $v|N $ with the local root number $\epsilon(\phi,\chi_{0,v})$ corresponding to the Rankin--Selberg convolution.
\end{itemize}

Fix a maximal order $R^{(S)}$ of $B_{\BA}^{(S)}\cong M_2(\BA^{(S)})$.
Let $U^{(S)}=R^{(S)\times}$. Note that $U^{(S)}$ is a maximal compact subgroup of
$B_{\BA}^{(S)^{\times}}\cong \GL_2(\BA^{(S)})$.

Let $p$ be an odd prime such that $p\nmid N$. 
As in the introduction, we fix an embedding $\iota_{p}:\ov{\BQ}\hookrightarrow \BC_p$. 
Let $v_p$ be a $p$-adic valuation. 
Let $\BF$ be an algebraic closure of $\BF_p$.
%Suppose we are given a finite order character $\chi_{0, S}: U_{0, S}\lra \ov{\BQ}^\times$ of conductor $c$ such that $\omega \cdot \chi_{0, S}\big|_{\BQ_S^\times \cap U_{0, S}}=1$. Assume that
%$\epsilon(\phi, \chi_v)\chi_{0,v}\eta_v(-1)=\epsilon(B_v)$
%for all places $v|N$.
%$v|N \infty$.

%Fix a maximal compact subgroup $U^{(S)}$ of $\wh{B}^{(S)}\cong M_2(\BA^{(S)})$. Let $p\nmid \prod_{\ell | c,\\{split}} (\ell-1) \prod_{\ell |c, \\{non-split}} (\ell+1)$ be a prime.

Let $\Theta_{S}$ denote the set of imaginary quadratic extensions as in Definition \ref{Theta}.
Recall that there exist infinitely many imaginary quadratic extensions $K/\BQ$ with $K \in \Theta_S$
(\cite{Br}, \cite{W} and \cite{Be}).
%For $K \in \Theta_{S}$, we fix such an embedding $\iota_{K}$.
%In the rest of the section, we let $\Theta$ denote $\Theta_{S}$.
%Let $\Theta$ denote the set of imaginary quadratic fields $K$ such that
%$p\nmid |\Pic_{K, c}|$ and there exists an embedding $\iota_{K}:K\ra B$ such that $K_S=K_{0, S}$ and
%$\wh{K}^{(S)}\cap U^{(S)}=\wh{\CO}_K^{(S)}$ under the embedding.
For $K \in \Theta_S$, we fix a torus embedding $\iota_{K}$ as in \S2.1.

For each $K\in \Theta_S$, let $\fX_{K,\chi_{0,S}}$ denote the set of finite order characters $\chi$ over $K$ as in Definition \ref{char}.
%such that (i) $\chi|_{\BA^\times} \cdot \omega=1$, (ii) $\chi_S=\chi_{0, S}$,  and (iii) $\chi$ is unramified outside $S$.
Recall that the conductor of $\chi$ equals one. Moreover, we have
\begin{itemize}
\item[(RN)] $\epsilon(\phi,\chi)=-1$.
\end{itemize}
Here $\epsilon(\phi,\chi)$ denotes the global root number of the Rankin--Selberg convolution corresponding to the pair $(\phi,\chi)$. 

In the rest of the section, we let $\Theta=\Theta_{S}$ and
$\fX_{K} = \fX_{K,\chi_{0,S}}$.

%Suppose we are given a finite order character $\chi_{0, S}: U_{0, S}/(\BQ_S^\times \cap U_{0, S}) \lra \ov{\BQ}^\times$ with conductor $c$. Assume that
%$B$ is precisely ramified at places $v$ satisfying
%$$
%\epsilon(\phi, \chi_v)\chi_{0,v}\eta_v(-1)= -1.
%$$

\begin{lem}\label{existence2}
The set $\fX_{K}$ is non-empty for all but finitely many imaginary quadratic fields $K$ with $K\in \Theta$. Moreover, it is a homogenous space for the class group $\Cl_{K}$.
\end{lem}
\begin{proof}
The same argument as in the proof of Lemma \ref{existence1} applies.
\end{proof}

Let $A=A_{\phi}$ be an abelian variety over $\BQ$ associated to $\phi$ such that $\cO_{L_{\phi}} \subset \End(A)$.

As in \cite{YZZ}, we have the representation of $B_{\BA}^{(\infty),\times}$ over the field $M:=\End^0(A_{/\BQ})$ given by
$$\pi^{B}=\varinjlim_{U\subset B_{\BA}^{(\infty),\times}} \Hom_\xi^0(X_U, A).$$
Here $\xi$ is a Hodge class on $X=\varprojlim_{U} X_{U}$.
Recall that
$$
\pi^{B} \otimes_{M}\BC \simeq \pi_{B,f}
$$
for $\pi_{B,f}$ being the finite part of $\pi_{B}$.

Let $\CO$ be the ring of integers of the field $\BQ(\phi, \chi_{0, S})$ generated over $\BQ$ by the Hecke eigenvalues of $\phi$ and the values of $\chi_{0, S}$. Let $\fp$ be the prime ideal of $\CO$ corresponding to $\iota_p$. 
Let $\cO_{(\fp)}$ be the localisation of the integer ring $\cO$ at the prime ideal $\fp$.

We have the following existence of toric test vectors. 

\begin{lem}
\label{TV1}
There exists a non-zero form $f\in \pi_B$ defined over $\cO_{(\fp)}$
satisfying the following.
\begin{itemize}
\item[(F1)] The subgroup $U_{0, S}$ acts on $f$ via $\chi_{0, S}$ and
\item[(F2)] $f \in \pi_{B}^{U^{(S)}}$.
\end{itemize}
\end{lem}
\begin{proof}
The same argument as in the proof of Lemma \ref{TV0} applies.
\end{proof}
Note that, $f$ is spherical outside $S$.

We further normalise $f$ to be $\fp$-primitive in the sense arising from geometric definition of modular forms (\S3.2.2).

%We further normalise $f$ to be $\fp$-primitive in the sense that the ideal generated by the image of $f$ is prime to $\fp$.

Let $U$ be an open compact subgroup of ${B}_{\BA}^{(\infty),\times}$ such that 
$U=U_{S}U^{(S)}$ with $U_{S}\subset B_{S}^\times$ and $U_{S}=\prod_{v \in S} U_{v}$ such that\begin{itemize}
\item[(L1)] $f\in \pi_{B}^U$ and
\item[(L2)] $\CO_{0, v}^\times \subset U_v$ for all $v\in S \backslash \{\infty\}$. 
\end{itemize}

Let $X_{U}$ be the corresponding Shimura set of level
$U$.
%that the ideal generated by the image of $f$ is prime to $\fp$.%Let $f\in \pi^B$ be a modular parametrisation which is $\CO$-valued when considered as a modular form
%and also satisfies the following.
%\begin{itemize}
%\item[(P1)] The subgroup $U_{0, S}$ acts on $f$ via $\chi_{0, S}$ and
%\item[(P2)] $f \in \pi^{U^{(S)}}$. In particular, it is spherical outside $S$.
%\end{itemize}
%The existence follows from the existence of test vector corresponding to the pair $(\pi,\chi)$.
As a modular parametrisation, $f$ corresponds to a morphism
$$
f:X_{U}\rightarrow A \otimes_{\cO_{\phi}} \cO
$$
and it maps the Hodge class $\xi$ to the identity element.
Here $A \otimes_{\cO_{\phi}} \cO$ denotes the Serre-tensor.
%We further normalise $f$ to be $p$-optimal.
%In what follows,
We may also regard $f$ as a $\overline{\BZ}_{(p)}$-valued weight two modular form at CM points on the Shimura curve $X_{U}$.
For generalities on CM points, we refer to \S3.2.

Let $f^{(p)}$ be the corresponding $p$-depletion.
Let $d$ be the Katz $p$-adic differential operator.
In view of the $p$-adic Waldspurger formula (\cite{BDP1} and \cite{LZZ}), the CM values of weight zero $p$-adic modular form $d^{-1}(f^{(p)})$ are closely related to $p$-adic logarithm of Heegner points arising from CM points on the Shimura curve.

\subsection{Shimura curve}
In this subsection, we briefly describe generalities on Shimura curves, modular forms, CM points and the underlying Igusa tower. For a detailed treatment, we refer to \cite[Ch. 7]{Hi5} and \cite[\S2]{Bu}.

As in \S3.1, $p$ denotes an odd prime and $\BF$ an algebraic closure of $\BF_p$.
Let $W$ denote the Witt ring $W(\BF)$.
Let $\mathcal{W}$ be the strict Henselisation of $\BZ_{(p)}$ given by $\iota_{p}^{-1}(W)$.

\subsubsection{Shimura curve} 
We describe generalities regarding Shimura curves. 

Let $B/\BQ$ be an indefinite quaternion algebra as in \S3.1.
In particular, $p \nmid \disc(B)$.
Let $(X_{U})_{U}$ be the corresponding tower of Shimura curves indexed by open compact subgroups $U \subset B_{\BA}^{(\infty),\times}$ and $X=\varprojlim_{U} X_{U}$ the underlying Shimura variety.

We recall the following (\cite{De2}).
\begin{thm}
\label{Sh}
Let the notation and assumptions be as above.
\begin{itemize}
\item[(1).] For a sufficiently small level $U$, the Shimura curve $X_U$ admit a smooth canonical model over the rationals.
\item[(2).] For a sufficiently small level $U$ prime to $p$, the Shimura curve $X_{U/\BQ}$ admits a smooth $p$-integral model $X_{U/\BZ_{(p)}}$.
\end{itemize}
\end{thm}

We also recall that $X_U$ represents a moduli functor classifying a class of elliptic curves or abelian surfaces with level $U$-structure. The difference in the moduli functors depends on whether $B$ is split. For a description of the moduli problem, we refer to \cite[\S2.2]{Bu}. Let $f_{U}:\mathcal{A}_{U}\ra X_{U/\BZ_{(p)}}$ be the universal abelian scheme. 

For level $U$ as in \S3.1, the Shimura curve $X_{U/\BQ}$ admits a $p$-integral model smooth at CM points with complex multiplication by $\cO_K$. This follows from the hypothesis (L2) and part (iv) of Definition \ref{Theta}.

In what follows, we fix a sufficiently small level $U$ prime to $p$ unless otherwise stated. 
\subsubsection{Modular forms} We briefly describe the notion of geometric modular forms. 

For an odd prime $p$, let $A$ be a $\BZ_{(p)}$-algebra. A weight two, $p$-integral modular form with level $U$ over $A$ is an element of $H^{0}(X_{U/A},\mathcal{L}_{U/A})$ for a line bundle 
$\mathcal{L}_{U/\BZ_{(p)}}$ on the Shimura curve $X_{U/\BZ_{(p)}}$ (\cite{Ha1} and \cite{HMF}). 
%Here $U$ is a sufficiently small prime to $p$ level.

When $A$ is nothing but $\cO_{(\fp)}$ for $\cO$ an integer ring of a number field and $\fp$ a prime above $p$ in $\cO$, we say that a modular form over $A$ is $\fp$-primitive if its image in 
$H^{0}(X_{U/\BF_{\fp}},\mathcal{L}_{U/\BF_{\fp}})$ is non-zero. Here $\BF_\fp$ denotes the residue field of $\cO$ at the prime $\fp$.
\subsubsection{CM points} We describe generalities on CM points. 

We identify $\BA_{K}^\times$ as a subgroup of $B_{\BA}^\times$ under the embedding $\iota_{K}: K \hookrightarrow B$ for an imaginary quadratic field $K$ as in \S3.1. 
Let $U$ be an open compact subgroup of ${B}_{\BA}^{(\infty),\times}$ as in \S3.1 satisfying (L1) and (L2).

As $\iota_{K}(\cO_{K}) \subset U$,
the choice of embedding $\iota_{K}$ gives rise to a map
$$
\varphi_{K}:\Cl_{K}\rightarrow X_{U/\ov{\BQ}}.
$$
For $\sigma \in \Cl_{K}$, let $x_{\sigma} \in X_{U/\ov{\BQ}}$ be the corresponding point on the Shimura curve.
This is usually referred as a CM point on the Shimura curve. In what follows, the map
$\varphi_{K}$ plays an underlying role.

We now suppose that $K$ is an imaginary quadratic field as above such that
\begin{itemize}
\item[(ord)] $p$ splits in $K$.
\end{itemize}

%Let $c$ be an integer with $p \nmid c$. For an ideal class $[\mathfrak{a}] \in \Pic_{K,c}$, let $x(\mathfrak{a})$ be the corresponding CM point on
%the Shimura curve $X_{U}$. By abuse of notation, we also denote it by $x_\sigma$ with $\sigma \in \Pic_{K,c}$.
Based on the CM theory of Shimura--Taniyama--Weil and the hypothesis (ord), we have the following (\cite[\S2.4]{Bu}).
\begin{lem}\label{CM}
Let the notation and assumptions be as above.
Then, the CM points $x_{\sigma}$ correspond to $p$-ordinary CM abelian varieties. Moreover, the CM points $x_{\sigma}$ are defined over the Henselisation $\mathcal{W}$.
\end{lem}
In view of the lemma, we can consider the mod $p$ reduction $x_{\sigma/\BF} \in X_{U/\BF}$ of the CM point $x_{\sigma}$.
We have the following key proposition based on the Serre--Tate deformation theory.
\begin{prop}
\label{CM}
Let the notation and assumptions be as above.
For $\sigma, \tau \in \Cl_{K}$, we have
$$
x_{\sigma/\BF}=x_{\tau/\BF} \iff \sigma=\tau.
$$
\end{prop}
\begin{proof}
For simplicity, let us consider the case when the quaternion $B$ is split.

The condition (ord) guarantees a categorical equivalence between
ordinary CM elliptic curves over the Witt ring $W$
and ordinary elliptic curves over $\BF$ (\cite[\S6.3]{Hi5}).
The equivalence is via the canonical lift in the deformation space over $W$.

For $\sigma\in \Cl_{K}$, recall that the CM point $x_{\sigma}$ is defined over $\mathcal{W}$, in particular over $W$.
It is thus the canonical lift of $x_{\sigma/\BF}$. This finishes the proof.

\end{proof}

\subsubsection{Igusa tower} We describe generalities on the Igusa tower. 

Let $U$ be a sufficiently small level prime to $p$ as in \S3.2.1 and $X_{U/\BZ_{(p)}}$ corresponding Shimura curve. 
Let $X_{U/W}^{\ord}$ be the $p$-ordinary locus.
Adding $p^\infty$-level structure to the moduli problem corresponding to the Shimura curve on the ordinary locus, we obtain an \'{e}tale Igusa tower
$$
\pi: Ig \rightarrow X_{U/W}^{\ord}
$$
(\cite[\S2.5]{Bu}). Here we only mention that this amounts to trivialise the connected component $\mathcal{A}_{U}[p^{\infty}]^{\circ}$ or equivalently the \'etale quotient $\mathcal{A}_{U}[p^{\infty}]^{et}$ of the $p$-divisible group $\mathcal{A}_{U}[p^{\infty}]$. 

For an imaginary quadratic field $K$ satisfying (ord), recall that $x_{\sigma} \in X_{U/W}^{\ord}$ for $\sigma \in \Cl_{K}$. 
In view of (ord), the $\fp_{K}^\infty$-divisible subgroup of the $p$-divisible group corresponding to $x_{\sigma}$ gives rise to a canonical point in the Igusa tower above $x_{\sigma}$. Here $\fp_K$ denotes the prime of $K$ above $p$ determined via the embedding $\iota_p$. 
By abuse of notation, let $x_{\sigma}$ also denote a point in the Igusa tower lying over the CM point.

For a $W$-algebra $B$, recall that a $p$-adic modular form $g$ over $B$ with tame level $U$ is given by
$$
g \in H^{0}(Ig_{/B},\mathcal{O}_{Ig/B}).
$$
Here we mention that the geometric theory of $p$-adic modular forms plays an essential role in the $p$-indivisibility of Heegner points. 

Regarding non-constancy of a mod $p$ modular form on the Igusa tower, we have the following immediate consequence of Proposition \ref{CM} and the Brauer-Siegel lower bound for the size of class groups.
\begin{cor}
\label{bound}
Let $g$ be a non-constant $p$-adic modular form over $\BF$ with weight zero and tame level $U$. For $c \in \BF$,
the size of the CM fiber
$$
 \left\{\ x_{\sigma} \ \Big|  \sigma \in \Cl_{K}, g(x_{\sigma})=c\right\}
$$
is bounded independent of $K$.
\end{cor}

This corollary plays a key role in the mod $p$ non-vanishing of toric periods of a $p$-adic modular form.

\subsection{Non-vanishing}
In this subsection, we prove the horizontal mod $p$ non-vanishing of toric periods of a $p$-adic modular form on a Shimura curve.

Let the notation and assumptions be as in \S3.1. In particular, $f$ is a weight two modular form on the Shimura curve $X_{U}$ and $f^{(p)}$ the $p$-depletion.
Let $g$ the weight zero $p$-adic modular form given by
 $$
g=d^{-1}( f^{(p)}).
$$

Let $\Theta$ be the set of imaginary quadratic fields as in \S3.1. Let $K \in \Theta$.
For $\sigma \in \Cl_{K}$, let $x_{\sigma} \in Ig_{/W}$ be the corresponding CM point as in \S3.2.4.
For $\chi \in \widehat{\Cl}_{K}$,
let $P_{g}(\chi)$ be the toric period given by
\begin{equation}
P_{g}(\chi)=\frac{1}{h_{K}}\cdot\sum_{\sigma \in \Cl_{K}} \chi(\sigma)^{-1}g(x_{\sigma}).
\end{equation}
%for $g=d^{-1}(f^{(p)})$.
%Note that the period depends on the choice of the torus embedding.
% (Lemma \ref{embedding1}).

%In what follows, let $g$ denote $d^{-1}(f^{(p)})$.
Note that $g$ is defined over a finite flat discrete valuation ring $W'$ over the Witt ring $W=W(\BF)$ as a $p$-adic modular form.
Here $\BF$ denotes an algebraic closure of $\BF_p$.
Let $\overline{g}$ denote the reduction of $g$ modulo the maximal ideal of the local ring $W'$.

As $g$ is $p$-integral and $p\nmid h_{K}$, the period $P_{g}(\chi)$ is $p$-integral.
More precisely, $P_{g}(\chi)\in W'[\chi]$ as the underlying CM points are defined over $W$. 
Here $W'[\chi]$ denotes the extension of $W'$ obtained by adjoining the values of $\chi$.
Note that the period
possibly depends on the choice of the torus embedding $\iota_K$.

\begin{remark}
(1). For $\chi \in \widehat{\Cl}_{K}$, the non-vanishing of $P_{g}(\chi)$ implies that $\chi \in \fX_{K}$.
This follows from $g$ being $\chi_{0,S}$-toric (part (1) of Lemma \ref{non-constancy1}).
The observation will be used in the proof of Theorem \ref{toric2}. 

(2). We restrict to newforms with trivial central character and unramified Hecke characters over the imaginary quadratic fields for simplicity. 
\end{remark}

Let $\fp_{0}$ be the prime above $p$ in the Hecke field $L_\phi$ corresponding to the newform determined via the embedding $\iota_p$. 
Let $\rho_{\phi,\fp_{0}}: G_{\BQ} \ra \GL_{2}(L_{\phi,\fp_{0}})$ be the corresponding $p$-adic Galois representation. 
Let $\overline{\rho}_{\phi,\fp_{0}}$ be the reduction modulo $\fp_{0}$.

We have the following lemma regarding $g$.
\begin{lem}
\label{non-constancy1}
Let the notation and assumptions be as above.
\begin{itemize}
\item[(1).] The $p$-adic modular form $g$ is $\chi_{0,S}^{-1}$-invariant under the action of $K_{0,S}^{\times}$.
\item[(2).] Assume that the mod $p$ Galois representation $\overline{\rho}_{\phi,\fp_{0}}$
associated to the newform $\phi$ is absolutely irreducible. Then the weight zero $p$-adic modular form $g$ is non-constant modulo $p$.
\end{itemize}
\end{lem}
\begin{proof}
For the first part, note that the $p$-depletion only changes the Hecke action at $p$. In particular, the action of $K_{0,S}$ is unchanged. As
$$
d^{-1}(f^{(p)})=\lim_{j\rightarrow 0}d^{-1+j}(f^{(p)})
$$
and each $d^{-1+j}(f^{(p)})$ is $\chi_{0,S}^{-1}$-invariant under $K_{0,S}^{\times}$, this finishes the proof.

For the second part, we first recall the non-constancy of the $p$-depletion $f^{(p)}$ under irreducibility
of the underlying mod $p$ Galois representation (\cite[Prop. 5.3]{Bu}). Recall that the $p$-adic differential operator acts as $t\frac{d}{dt}$ on the $t$-expansion of a $p$-adic modular form around a CM point. The non-constancy of $d^{-1}(f^{(p)})$ thus follows. Here $t$ denotes the Serre--Tate co-ordinate of the deformation space around the CM point (\cite[\S3.1]{Bu}).
\end{proof}

Our main result regarding the mod $p$ non-vanishing of toric periods is the following.
\begin{thm}
\label{toric2}
Let $\phi \in S_{2}(\Gamma_{0}(N))$ be a newform.
Let $p\nmid N$ be a prime such that
\begin{itemize}
\item[(irr)] the mod $p$ Galois representation $\ov{\rho}_{\phi, \fp}$ is absolutely irreducible.
\end{itemize}
Let $f$ be a toric test vector on an indefinite quaternion algebra as above,
$f^{(p)}$ the corresponding $p$-depletion and $g=d^{-1}(f^{(p)})$
with $d$ being the Katz $p$-adic differential operator.
Let $\Theta=\Theta_S$ be the set of imaginary quadratic fields as above.
For $K \in \Theta$, let $\fX_K$ be the set of finite order Hecke character over $K$ as in Definition \ref{char}. For $\chi \in \fX_K$, let $P_{g}(\chi)$ be the toric period corresponding to the pair $(g,\chi)$. 
% satisfying (ord).

%Let the notation and assumptions be as above.
For $\epsilon > 0$, we have
$$
 \#\left\{\chi\in\fX_{K}\ \Big|\ v_{p}(P_{g}(\chi))=0 \right\}
\gg_{\epsilon} \log(|D_{K}|)^{1-\epsilon}.
 $$
 as $K \in \Theta$ varies over imaginary quadratic fields with
 \begin{itemize}
% \item[(i)] $p \nmid |\Cl_{K}|$ and
 %\item[(ii)]
\item[(ord)] $p$ split.
 %and
% $K_S\cong K_{0, S}$.
\end{itemize}
\end{thm}
\begin{proof}
Recall that $g$ is a function on the Igusa tower $Ig$ defined over $W'$ and $\overline{g}$ its  reduction modulo the maximal ideal of $W'$.
Note that $\overline{g}$ is a function on the mod $p$ Igusa tower $Ig_{\BF}$.
In view of the map
$\varphi_{K}:\Cl_{K}\rightarrow Ig$, we may regard $g$ and $\overline{g}$ as functions on the abelian group $\Cl_{K}$.

The approach is based on the Fourier analysis on $\Cl_{K}$ and its relation to the mod $p$ Igusa tower $Ig_{\BF}$.

%For $\sigma \in \Pic_{K,c}$, the element $x_{\sigma}$ as above is nothing but the corresponding Gross point on $X$.

To begin with, the toric periods $P_{g}(\chi)$ modulo $\fm_{\chi}$ equal the $\chi$-th Fourier coefficients of the function
$$\sigma \mapsto g(x_{\sigma}) \mod{\fm'}$$ on the finite abelian group $\Cl_{K}$.
Here $\fm_{\chi}$ (resp. $\fm'$) is the maximal ideal of $W'[\chi]$ (resp. $W'$).
%determined via the embeddings.
Here $W'$ denotes a finite flat extension of $W$ and $W'[\chi]$ its extension obtained by adjoining the values of $\chi$ as above.
Let $\fm$ denote the maximal ideal of $W$.

In what follows, we thus consider Fourier analysis of the function
$$\overline{g}:\Cl_{K}\rightarrow \BF$$
and its variation for $K \in \Theta$.
As $p \nmid h_{K}$, the Fourier inversion works as usual.

Let us first consider the non-vanishing of at least one twist.
In view of the Fourier inversion, it suffices to show that
\begin{equation}
\label{pCM1}
%v_{p}(g(x_{\sigma}))= 0
\overline{g}(x_{\sigma})\neq 0.
\end{equation}
for at least one $\sigma \in \Cl_{K}$ for all but finitely many $K \in \Theta$. By abuse of notation,  $x_{\sigma}$ denotes the reduction of the CM point $x_{\sigma}$ modulo $\fm$.

In view of Proposition \ref{CM} and the Brauer--Siegel lower bound, the number of CM points
$$
\left\{\ x_{\sigma} \ \Big|  \sigma \in \Pic_{K,c} \right\} \subset Ig_{\BF}
$$
increases with the discriminant of $K$.
Thus, the $p$-primitivity of $g$ readily implies \ref{pCM1} (part (2) of Lemma \ref{non-constancy1}).

We now consider the growth in the number of non-vanishing twists.
%In what follows, we consider $g$ as a mod $p$ modular form.
Suppose that there exists a positive integer $\ell_{K}$ such that exactly $\ell_{K} -1$ of the toric periods
$P_{g}(\chi)$
are non-vanishing modulo $\fm_{\chi}$, for $K$ with sufficiently large discriminant.
Let
$$
\Xi_{K}=\{ \chi_{K,1}, ..., \chi_{K,\ell_{K}-1}\}
$$
 be the set of the non-vanishing modulo $\fm_\chi$ twists.

From Shimura's reciprocity law,
$$
\overline{P_{g}(\chi)} \neq 0 \iff \overline{P_{g}(\chi^{\sigma})}\neq 0
$$
for $\sigma \in \Gal(\BF/\BF_{p}(\overline{\phi},\overline{\chi_{0,S}}))$. 
Here $``\ov{\cdot}"$ denotes the reduction modulo a prime ideal above $p$ determined via the embedding $\iota_p$. 
We conclude that $\Xi_{K}$ is invariant under the action of
$\Gal(\BF/\BF_{p}(\overline{\phi},\overline{\chi_{0,S}}))$.
%Let $G_{K} \subset \widehat{\Pic_{K,c}}$ the subgroup generated by these characters. Let $H_{K} \subset \Pic_{K,c}$ be the orthogonal complement of $G_{K}$. Recall that $|H_{K}|=|\Pic_{K,c}|/|G_{K}|$.

Let $C_{K} \subset \widehat{\Cl}_{K}$ the subgroup generated by these characters and let $H_{K} \subset \Cl_{K}$ be the orthogonal complement of $C_{K}$. The orthogonal complement is with respect to the pairing on $\widehat{\Cl}_{K}$ and its character group.
Recall that $|H_{K}|=|\Cl_{K}|/|C_{K}|$.

In view of the Fourier inversion on $\Cl_{K}$ and the hypothesis on the periods, we have
\begin{equation}
\overline{g}(x_{\sigma})=\sum_{i=1}^{\ell_{K}-1} c_{K,i}\chi_{K, i}(\sigma).
\end{equation}
Here $c_{K,i}\in \BF$ are the Fourier coefficients possibly dependent on K and determined by
the pairs $(\overline{g},\chi_{K,i})$.

From the above formula on $H_{K}$, it follows that $g$ is constant on $H_{K}$.
In view of part (2) of \ref{non-constancy1} and Corollary \ref{bound}, there exists a constant $c$ such that
$$
|H_{K}| < c
$$
as $|D_{K}|\ra \infty$.

In particular,
$$|C_{K}| \gg_{\epsilon'} |D_{K}|^{\frac{1}{2}-\epsilon'}$$
for $\epsilon'>0$.

From the Galois-stability of $\Xi_{K}$ and Lemma \ref{stability}, we conclude
$$
\ell_{K}\gg_{\epsilon} \log(|D_{K}|)^{1-\epsilon}
$$
for any $\epsilon > 0$.

This finishes the proof.
\end{proof}
\noindent
\begin{remark}
%(1). There exist infinitely many imaginary quadratic extensions $K/\BQ$ with $K \in \Theta$
%(\cite{Br} and \cite{W}).

%(2)
Under the assumptions in Theorem \ref{toric2}, it seems tempting to predict
$$
 \#\left\{\chi\in\fX_{K}\ \Big|\ v_{p}(P_{g}(\chi))=0 \right\}
\gg_{\epsilon} |\Cl_{K}|^{1-\epsilon}
 $$
 for any $\epsilon > 0$.
 \end{remark}

\subsection{$p$-indivisibility of Heegner points}
In this subsection, we prove the horizontal $p$-indivisibility of Heegner points based on the horizontal mod $p$ non-vanishing of toric periods.

Let the notation and assumptions be as in \S3.1. In particular, we have modular parametrisation $f: X_{U} \rightarrow A_{\phi}\otimes_{\cO_{\phi}} \cO$ for the Shimura curve $X_{U}$ and the abelian variety $A_{\phi}$ associated to a  newform $\phi \in S_{2}(\Gamma_{0}(N))$. 
For choice of a base point for Albanese map of the Shimura curve $X_{U}$ into its Jacobian, we follow 
\cite[Rem. 6]{Z}. Here we only remark that existence of such a normalisation follows from maximality of the $p$-adic Galois representation ${\rho}_{\phi, \fp}$ 
associated to $\phi$. 

In what follows, we choose the parametrisation $f$ to be $(\cO_{\phi},\fp_{\phi})$-optimal as in \cite[\S3.7]{Z}. Here $\fp_\phi$ denotes the prime of $\cO_\phi$ above $p$ arising from the embedding $\iota_p$.

Recall, $\Theta$ denotes the set of imaginary quadratic fields in \S3.1.
For $K \in \Theta$, $\fX_K$ denotes the set of finite order Hecke character over $K$ as in Definition \ref{char}.
For $\sigma \in \Cl_K$, we have the same CM point $x_{\sigma} \in X_{U/\ov{\BQ}}$ (\S3.2.3).
 Let $\chi \in \fX_K$.
Associated to the pair $(f,\chi)$, we have Heegner point $P_{f}(\chi)$ given by

\begin{equation}
P_{f}(\chi):=\sum_{\sigma\in \Cl_{K}} \chi(\sigma)^{-1}f(x_{\sigma})
\end{equation}

%For $p$-integrality of the Heegner point, we consider the normalisation as in \cite[Rem. 7 and \S3.7]{Z}. Here we only remark that existence of such a normalisation follows from maximality of the $p$-adic Galois representation ${\rho}_{\phi, \fp}$ 
%associated to $\phi$.
% (Lemma \ref{coeff}).

We have the following immediate.
\begin{lem}
\label{rationality}
Let the notation and assumptions be as above. Then,
$$
P_{f}(\chi) \in A(H_{\chi})
$$
for the extension $H_{\chi}/K$ cut out by $\chi$ and $A=A_{\phi}\otimes_{\cO_{\phi}} \cO$.
\end{lem}

%For $g$ as in \S3.3 and $\chi \in \widehat{\Pic_{K,c}}$, let
%$$
%P_{g}'(\chi)=|\Pic_{K,c}|P_{g}(\chi).
%$$
Fundamental $p$-adic invariants associated to a Heegner point are its $p$-adic height and $p$-adic logarithm. 
Regarding the $p$-adic logarithm of the Heegner point, we have the following fundamental $p$-adic Waldspurger formula due to Bertolini--Darmon--Prasanna (\cite{BDP1}, \cite{BDP2}), Brooks (\cite{Bro}) and Liu--Zhang--Zhang (\cite{LZZ}).
\begin{thm}
\label{log}
Let $\phi \in S_{2}(\Gamma_{0}(N))$ be a newform.
Let $p\nmid N$ be a prime and $a_{p}(\phi)$ the corresponding Hecke eigenvalue.
Let $f$ be a toric test vector on an indefinite quaternion algebra as above and $g=d^{-1}(f^{(p)})$
with $d$ being the Katz $p$-adic differential operator. Let $K/\BQ$ be an imaginary quadratic extension with $p$-split. 
Let $\fp_K$ be the prime above $p$ in $K$ determined via the embedding $\iota_p$ and $\ov{\fp}_K$ its conjugate. 
Let $\chi$ be a finite order unramified Hecke character over $K$ as above.
Let $P_{f}(\chi)$ be the Heegner point corresponding to the pair $(f,\chi)$.

Then, we have
$$ h_{K}\cdot P_{g}(\chi) = (1-\chi^{-1}(\overline{\mathfrak{p}}_{K})p^{-1} a_{p}(\phi)+\chi^{-2}(\overline{\mathfrak{p}}_{K})p^{-1})\log_{\omega_{f}}(P_{f}(\chi))$$
for the $p$-adic logarithm $\log_{\omega_{f}}(\cdot)$.
\end{thm}
\begin{proof} This is essentially proven in \cite{BDP1}. We recall the steps and refer to \cite[\S3]{BDP1} for details.

Note that
\begin{equation}
\label{$p$-Hecke}
T_{p}(f)= a_{p}(\phi)f, \langle p \rangle (f)=f
\end{equation}
for $\langle p \rangle$ being the diamond operator associated with $p$. Indeed, this follows from the construction of the toric test vector $\phi$ in \S3.1.

Let $F_{f}$ the Coleman primitive of the differential $\omega_{f}$ associated with $f$ (\cite[\S3.7 and \S3.8]{BDP1}).
In view of \cite[Lem. 3.23 and Prop. 3.24]{BDP1}, we have
$$
F_{f}(x_{\sigma})=g( x_{\sigma}) - a_{p}(\phi) g(\mathfrak{p}_{K}*x_{\sigma})+\frac{1}{p}g(\fp_{K}^{2}*x_{\sigma}).
$$
%Here $\fp_K$ is the prime above $p$ in the imaginary quadratic field $K$ determined via $\iota_p$ and  
For the $``*"$-action on the set of CM points, we refer to \cite[(1.4.8)]{BDP1}.

From \cite[Lem. 3.23]{BDP1} and \cite[Thm. 3.12]{BDP2}, we have
$$
F_{f}(x_{\sigma})=\log_{\omega_{f}}(f(x_{\sigma})).
$$

The summation of the above identity over $\sigma \in \Cl_{K}$ along with \ref{$p$-Hecke} and \cite[Proof of Thm. 5.13]{BDP1} finishes the proof.

\end{proof}
\begin{remark}
Based on Waldspurger formula, the left hand side of the above formula turns out to be a value of an anticyclotomic Rankin--Selberg $p$-adic L-function outside its interpolation range (\cite{BDP1}, \cite{Bro} and \cite{LZZ}). However, the above version only relies on Coleman integration and it suffices for our application to $p$-indivisibility of Heegner points. 
\end{remark}

In view of the $p$-adic Waldspurger formula and the $p$-indivisibility of toric periods (Theorem \ref{toric2}), we have the following $p$-indivisibility of Heegner points.

\begin{thm}\label{main3}
Let $\phi \in S_{2}(\Gamma_{0}(N))$ be a non-CM newform.
Let $p\nmid N$ be a
%ordinary
prime such that
\begin{itemize}
\item[(irr)] the mod $p$ Galois representation $\ov{\rho}_{\phi, \fp}$ is absolutely irreducible.
\end{itemize}
Let $A_{\phi}$ be the corresponding $\GL_2$-type abelian variety such that $\cO_{L_{\phi}} \subset \End(A)$.
Let $K$ be an imaginary quadratic field with $K \in \Theta$ for $\Theta$ as above.
Let $\chi$ be a finite order Hecke character over $K$ as above. 
Let $\fp$ the prime above $p$ of the coefficient field $L$ corresponding to the fixed embedding $\iota_p$.
Let $P_{f}(\chi)$ be the Heegner point corresponding to the pair $(f,\chi)$.

For $\epsilon > 0$, we have
$$
 \#\left\{\chi\in\fX_{K}\ \Big|
 \text{$P_{f}(\chi)$ $\fp$-indivisible in $A(H_{\chi})$ }\right\}
\gg_{\epsilon} \log(|D_{K}|)^{1-\epsilon}
 $$
 as $K$ varies over imaginary quadratic fields with
 \begin{itemize}
% \item[(i)] $p \nmid |\Cl_{K}|$ and
 \item[(ord)] $p$ split.
 %and
 %$K_S\cong K_{0, S}$.
 \end{itemize}
\end{thm}

\begin{proof}
Recall that $p$ is unramified in the anticyclotomic extension $H_{\chi}$. In view of Lemma \ref{rationality} and the $p$-integrality of normalised $p$-adic logarithm, the $p$-indivisibility thus follows from Theorem \ref{toric2} and Theorem \ref{log}.
\end{proof}

\begin{remark}
Based on Jochnowitz congruence (\cite{V} and \cite{BD}), above theorem can be deduced from Theorem \ref{toric1} for variation over imaginary quadratic fields satisfying extra hypothesis that there exists a fixed prime inert in them. 
\end{remark}

\subsection{Tate--Shafarevich groups II} 
In this subsection, we describe bounds for the Tate--Shafarevich group in the analytic rank one case in terms of a Heegner points based on an analysis Euler system of Heegner points due to Kolyvagin (\cite{Ko}) as refined by Nekov\'a\v{r} (\cite{N}).

Let the notation be as in \S3.1. 
In particular, $L_\phi$ denotes the Hecke field of the newform $\phi\in S_2(\Gamma_0(N))$.
%and $\CO_{L_\phi}$ its ring of integer.
Let $\pi$ the cuspidal automorphic representation of $\GL_{2}(\BA)$ corresponding to the newform 
$\phi$ so that we have the equality
$$L(s,\pi)=L(s,\phi)$$
of complex L-functions. 
Let $A$ be a $\GL_2$-type abelian variety  over $\BQ$ associated to  $\phi$ such that $\CO_{L_\phi}\subset \End (A)$.

%Let $c\geq 1$ be a positive integer such that for each $\ell |N$, we have
%$\ord_\ell(c)\geq \cond(\omega_\ell)$ for the local component $\omega_\ell$ of $\omega$.

Let $B$ be a indefinite quaternion algebra and $\chi\in \fX_K$ a character over an imaginary quadratic field $K$ as in \S2.1. Recall that $f$ is a toric form on $B_{\BA}^\times$ and
$U \subset (B_{\BA}^{(\infty)})^{\times}$
the level corresponding to $f$. 
We have the underlying Heegner point $P_{f}(\chi) \in A(H_{\chi})$.

In the rest of this subsection, let $L=L_{\phi,\chi}$. Let $A_\chi=A_K\otimes_{\CO_{L_\phi}} \CO_L$ be the Serre tensor where the absolute Galois group $G_K$ acts on $\CO_L$ via $\chi$. In regards to certain aspects of the arithmetic corresponding the pair $(A,\chi)$, it turns out that $A_{\chi}$ is a relevant abelian variety.

The main result of this subsection regarding the Tate--Shafarevich group of the Serre tensor $A_{\chi}$ is the following.
  \begin{prop} \label{Ko}
   Let $\phi \in S_{2}(\Gamma_{0}(N))$ be a non-CM newform.
 Let $K/\BQ$ be an imaginary quadratic extension and $\chi$ an unramified finite order Hecke character over $K$ as above.
 Let $p$ be a prime and $\fp_\phi$ (resp. $\fp$) the prime above $p$ of the coefficient fields $L_\phi$ (resp. $L$) corresponding to the fixed embedding $\iota_p$. Let $A$ be an abelian variety corresponding to $\phi$ such that $\cO_{L_{\phi}} \subset \End(A)$. For $\ell |N$, let $c_{\ell}$ be the corresponding Tamagawa number. Suppose that the following holds.

% Let $p$ be a prime and $\fp_\phi, \fp$ primes above $p$ of $L_\phi$ and $L$ corresponding the fixed embedding, respectively. Assume that $A$ is non-CM abelian variety and the following:
 \begin{enumerate}
 \item The $p$-adic Galois representation $\rho_{A, \fp_\phi}$ has maximal image,
 \item $p\nmid 6N \cdot (\prod_\ell c_\ell )\cdot h_K$,
 \item There exists a polarisation $\varphi:A \rightarrow A^{\vee}$ with $p \nmid \deg(\varphi)$ for $A^{\vee}$ the dual abelian variety.
 \item The Heegner point $P_f(\chi)$ is $\fp$-indivisible with $f$ the $p$-primitive toric form in last subsection.
 \end{enumerate}
 Then, 
$$\Sha(A_\chi/K)[\fp^\infty]=0 \qquad \text{for $|D_K|\gg 0$}.$$ 
As $p\nmid h_{K}$, we also have
$$
\Sha(A/H_{K})^\chi[\fp^\infty]\cong \Sha(A_\chi/K)[\fp^\infty]=0 \qquad \text{for $|D_K|\gg 0$}.
$$

 \end{prop}
 \begin{proof}

 When the Heegner point $P_{f}(\chi)$ is non-torsion and $\fp_\chi$-indivisible, the Kolyvagin method to bound
$\Sha(A_\chi/K)[\fp^\infty]$ shows that
$$
\fp^{C_{\chi}}\Sha(A_\chi/K)[\fp^\infty]=0
$$
for a non-negative integer $C_{\chi}$ (\cite[7.5.3 \& 7.6.5]{N0}). The constant $C_{\chi}$ possibly depends on $p$.

Under the hypotheses (1)--(3), we show that
$$
C_{\chi}=0
$$
for all $\chi \in \fX_{K}$ as $|D_{K}|\gg 0$.

In the following exposition, we suppose some familiarity with \cite[\S2]{N0}.

Strictly speaking, the approach is slightly different which we now describe. Let $\cO=\cO_\fp$
and $M$ a positive integer with $h_{L}| M$.

Let $Sel_{M,\chi}$ be the classical Selmer group $\Sel(A_{\chi}/H_{\chi},\fp^{M})$ for $\fp^M$-descent on the abelian variety $A_{\chi}$. Multiplication by $\fp^M$ along with a choice of generator of $\fp^M$ gives rise to the Kummer exact sequence
$$
0 \rightarrow A_{\chi}(H_{\chi})/\fp^{M} \xrightarrow{\delta} Sel_{M,\chi} \rightarrow \Sha(A_{\chi}/H_{\chi})[\fp^{M}] \rightarrow 0.
$$

Let $\varpi \in \cO_L$ be a uniformiser corresponding to $\fp$. Let
$$
\kappa = \varpi^{C_{1}}\delta(P_{f}(\chi)) \in Sel_{M,\chi}
$$
for a non-negative integer $C_1$ considered below. Then, there exists a non-negative integer $C_{\chi}$ such that
$$
\fp^{C_{\chi}}(Sel_{M,\chi}/\cO\cdot \kappa)=0.
$$

Under the hypotheses (1)--(3), we show that
$$
C_{1}=C_{\chi}=0
$$
for all $\chi \in \fX_{K}$ as $|D_{K}|\gg 0$. It would then follow
$$
\Sha(A_{\chi}/H_{\chi})[\fp^\infty]=0.
$$

In view of the description of $C_{1}$ and $C_{\chi}$, we consider the following.
\begin{itemize}

\item[(i)] 
For a finite place $v$ of $H_\chi$, let $\widetilde{A_{\chi,v}}$ be the special fiber of the N\'{e}ron model of
$A_{\chi}$ over $\cO_{H_{\chi}}$ at $v$. Let $\pi_{0}(\widetilde{A_{\chi,v}})$ be the $\cO_L$-module of the connected components of $\widetilde{A_{\chi,v}}$. Let $C_{1,v}$ be smallest non-negative integer such that
$$
\fp^{C_{1,v}}\pi_{0}(\widetilde{A_{\chi,v}})_{\fp}=0.
$$
Let
$$
C_{1}=\max_{v}C_{1,v}.
$$

In view of the hypothesis (2), we have
$$
C_{1}=0
$$
(\cite[7.2.1]{N0}).

%\item[(ii)] As in the proof of Proposition \ref{BD}, we have
%$$
%\Im(O_{L,\mathfrak{p}}[G_{\BQ}]\rightarrow \End_{O_{L,\mathfrak{p}}}(A_{\chi}[\fp^{\infty}]))=
%M_{2}(O_{L,\mathfrak{p}}).
%$$

\item[(ii)] As in the proof of Proposition \ref{BD}, we also have
$$
\Im(O_{L,\mathfrak{p}}[G_{H_{\chi}}]\rightarrow \End_{O_{L,\mathfrak{p}}}(A_{\chi}[\fp^{\infty}]))=
M_{2}(O_{L,\mathfrak{p}})
$$
for $p>3$ and $|D_{K}|\gg 0$. 

In the notation of \cite{N0}, we thus have $C_{3}=0$.

\item[(iii)] Let $M_0$ be the constant in \cite[5.1]{N0}.
Let $H_{M,\chi}=\BQ(A[\mathfrak{p}^{M+M_{0}}])H_{\chi}$ and
$U_{M}=\Gal(H_{M,\chi}/H_{\chi})$.
Let $\Res:H^{1}(H_{\chi},A[\mathfrak{p}^{M}])\rightarrow H^{1}(H_{M,\chi},A[\mathfrak{p}^{M}])^{U_{M}}$
be the restriction map.
Based on the proof of (iv) in the proof of Proposition \ref{BD}, we have
$$
\ker(\Res)=0.
$$
for $p>3$ and $|D_{K}|\gg 0$. 

In the notation of \cite{N0}, we thus have $C_{3}=0$.

\item[(iv)] As $p\nmid h_{K}$, we have $C_{4}=C_{5}=0$ in the notation of \cite{N0}.

\item[(v)] In view of our hypothesis $p \nmid \deg(\varphi)$, we have $C_{6}=0$ in the notation of \cite{N0}.
\end{itemize} 

We now recall 
$$
C_{\chi}=2C_{1}+4C_{2}+4C_{3}+C_{5}+C_{6}.
$$
if $\chi^{2}=1$ and 

$$
C_{\chi}=4C_{1}+7C_{2}+7C_{3}+5C_{4}+2C_{5}+2C_{6}.
$$
otherwise  (\cite[7.5.3 \& 7.6.5]{N0}). 

We thus conclude 
$$
C_{\chi}=0
$$
for all $\chi \in \fX_{K}$ such that $|D_{K}| \gg 0$ (\cite[2.9.6]{N}).
This finishes the proof for $\Sha(A_{\chi})$.

In view of \cite{MRS}, the rest follows for $\Sha(A/H_{K})$ from the result for the Serre tensor $A_{\chi}$.
\end{proof}

\begin{remark}
Based on an analysis of Euler system method for Heegner points, it may be checked that there exists a constant $C$ such that
$$
\# \Sha(A_{\chi})[\fp^{\infty}] \bigg{|} p^{C+v_{p}(P_{f}(\chi))}
$$
for all $\chi \in \fX_{K}^{-}$ as $|D_{K}|\gg 0$.
Here $v_{p}(P_{f}(\chi))$ measures the $\fp$-divisibility of the Heegner point $P_{f}(\chi)$.
In this sense, the Heegner points control the size of the Tate-Shafarevich groups.
\end{remark}
\subsection{Analytic Sha II}
In this subsection, we describe bounds for analytic Sha in the analytic rank one case in terms of a Heegner point based on an analysis of explicit Waldspurger formula due to Cai--Shu--Tian (\cite{CST}).
%In this subsection, we consider horizontal variation of the underlying central L-values. 
The consideration along with the main result in \S3.5 leads us to establish the $p$-part of corresponding Birch and Swinnerton-Dyer conjecture for a class of primes $p$.

Let the notation and assumptions be as in \S3.1. In particular, $A$ denotes a $\GL_2$-type abelian variety over the rationals corresponding to a weight two newform $\phi \in S_{2}(\Gamma_{0}(N))$ and $\chi$ a finite order Hecke character over an imaginary quadratic field $K$ such that
$$
\epsilon(A,\chi)=-1
$$
((RN)).
 Moreover,
$A_{\chi}$ denotes the Serre tensor of $A$ by $\chi$ which is an abelian variety over $K$ with $\cO_{L} \subset \End(A_{\chi})$.

Recall that, we refer to $\CL(A_{\chi})$ in \S2.3 as the analytic Sha of $A_\chi$. 

In this subsection, we consider the BSD conjecture in the analytic rank one case.
We thus suppose that
 $$L'(1, A_\chi)\neq 0.$$

To introduce an explicit Gross--Zagier formula, we begin with some notation. 
Let $B$ be a definite quaternion algebra as in \S3.1. 
Let $R \subset B$ be an order as in \cite[\S1.1]{CST}. 
Recall that $f$ is a toric form on $B_{\BA}^\times$ and
$U \subset (B_{\BA}^{(\infty)})^{\times}$
the level corresponding to $f$. 
Let $X_U$ be the corresponding Shimura curve. 
Let $N=N^{+}N^{-}$ for $N^{+}$ (resp. $N^{-}$) precisely divisible by split (resp. non-split) primes in the extension $K/\BQ$. 
Recall that $S$ denotes the set of prime divisors of $N\infty$. 

For any $f_i \in \pi^{B}$ factoring through $X_{U}$ with $i=1,2$ (\S3.1), we put
   \[ \langle f_1,f_2\rangle = \frac{1}{\vol(X_{U})}\deg  ((f_{2}\circ \tau_{B})^{\vee}\circ f_{1}). \]
   Here a canonical polarisation is used on the Jacobian and
   % $w_x$ is the cardinality of $(B^\times \cap g \wh{R}^\times g^{-1})/\{ \pm 1\}$
   %with $g\in \wh{B}^\times$ any representative of $x$ and
   $\tau_B \in N_{\wh{B}^\times}(\wh{R}^\times)$
   such that
   \begin{itemize}
	   \item for $\ell|N^+$,
		 $\tau_{B,\ell} \in N_{B_\ell^\times}(R_\ell^\times) \setminus \BQ_\ell^\times R_\ell^\times$ the Atkin--Lehner operator;
	   \item for $\ell | N^-$, $\tau_{B,\ell} t_\ell = \ov{t_\ell} \tau_{B,\ell}$ for any $t_\ell \in K_\ell^\times$ where
	   $\ov{t_\ell}$ is the Galois conjugate of $t_\ell$.
   	   \item for $\ell \nmid N$, $\tau_{B,\ell} = 1$.
   \end{itemize}
   We consider another 
    %Hermitian
    pairing $(\ ,\ )_{\widehat{R}^\times}$ is given by
    \[ (f_1,f_2)_{\widehat{R}^\times} = \frac{1}{\vol(X_{U})} \deg (f_{2}^{\vee}\circ f_{1}). \]
    
From \cite[Thm. 1.5]{CST}, we have an explicit Gross--Zagier formula given by 
$$
L'(1, A_\chi)=2^{-\# \Sigma_D+2}\cdot \frac{8\pi^2 (\phi,\phi)_{\Gamma_0(N)}}{[\CO_{K}^\times :\BZ^\times]^2\sqrt{|D_K|^2}}\cdot 
\frac{\langle P_{f_1}(\chi), P_{f_2}(\chi^{-1})\rangle}{(f_1, f_2)_{\wh{R}^\times}}.
$$
Here $\Sigma_{D}=\big{\{}\ell \big{|} \ell | (N,D_{K})\big{\}}$, $f_{1} \in V(\pi_{B},\chi)$ and $f_{2} \in V(\pi_{B},\chi^{-1})$ are non-zero vectors. For the definition of $V(\pi_{B},\chi)$, we refer to \cite[Def. 1.7]{CST}. 
Here we only mention that $f_1$ can be taken to be the toric form $f$ as in \S3.1.

In the analytic rank one case, the analytic Sha of $A_\chi$ is thus given by
$$
\CL(A_\chi)=
\frac{\pi^2 (\phi,\phi)_{\Gamma_0(N)}}{\sqrt{|D_K|^2}\Omega(A_\chi/K)}\cdot
\frac{\langle P_{f_1}(\chi), P_{f_2}(\chi^{-1})\rangle}{(f_1, f_2)_{\wh{R}^\times}}\cdot 
\frac{\Fitt (A_\chi(K)_\tor) \Fitt (A^\vee_{\chi^{-1}}(K)_\tor)}{2^{\# \Sigma_D -5}\cdot [\CO_K^\times :\BZ^\times]^2\cdot\prod_{v}\Fitt(c_v(A_\chi))}.
$$
 %By a scalar multiplication,  we may assume $f=f_1+f_1^\bot$ according to the $U_{0, S}$-action with $f_1$ as previous $\fp$-primitive test vector.  Note that $P_\chi(f_1)=P_\chi(f)$ for all $\chi\in \fX^+$.

 We have the following key proposition regarding the variation of the $\mathfrak{p}$-part of analytic Sha in terms of Heegner points.

\begin{prop}\label{control1}
Let $\phi \in S_{2}(\Gamma_{0}(N))$ be a non-CM newform and $L_{\phi}$ the corresponding Hecke field.
  %Let $p$ be a prime and $\fp_\phi, \fp$ the primes above $p$ of the coefficient fields $L_\phi$.
 % and $L$ corresponding to the fixed embedding, respectively.
 Let $A$ be an abelian variety corresponding to $\phi$ such that $\cO_{\phi} \subset \End(A)$.
 %Let $c$ be a positive integer.

There exists a finite set $\Sigma_{A}$ of primes of $\CO_\phi$ only dependent on $A$,  such that for any prime $\fp_0\notin \Sigma_{A}$ of $\CO_\phi$ and any $\chi\in \fX$ with $L'(1, A_\chi)\neq 0$, the ideal
$$
\displaystyle{\frac{\CL(A_\chi)}{[A_{\chi}(K):\cO_{\phi,\chi} P_{f}(\chi)]^{2}}}
$$
is coprime to $\fp_0$.
Here $f$ is the $\chi_{0, S}$-toric vector as in last subsection.
%with $\fp$ the prime above $\fp_0$ in $\cO_{\phi,\chi}$ determined via the embeddings.
 In particular, if $P_f(\chi)$ is a $\fp$-indivisible, then $\CL(A_\chi)$ is a $\fp$-unit as long $\fp$ lies above some $\fp_{0} \notin \Sigma_{A}$.

\end{prop}
\begin{proof} 

Let $f_{1}'=f$ be a $\mathfrak{p}$-minimal toric test vector satisfying (F1) and (F2) in \S3.1. Let $J \in B$ be analogue of the construction in \cite[\S2.1]{CH}. Let $f_{2}'$ the $J$-translate of $f_{1}'$ given by 
$$
f_{2}'(x)=f_{1}'(xJ).
$$
Locally, $f_{2}'$ differs from $f_{1}'$ only at the places dividing $N$. 
We have
$$
P_{f_{1}'}(\chi), P_{f_{2}'}(\chi^{-1})
\in A_{\chi}(K).
$$

Let $\alpha: A \rightarrow A^\vee$ be an isogeny.
It gives rise to an isogeny $\alpha:A_{\chi}\rightarrow A^{\vee}_{\chi^{-1}}$.
Note that
$$
P_{\alpha \circ f_{2}'}(\chi^{-1})=\alpha \circ P_{f_{1}'}(\chi).
$$

%We refer to \cite[\S1.3]{CST} for the notation (recall).

%We now choose $f_{1}'$ to be $\mathfrak{p}$-minimal toric test vector satisfying (F1) and (F2) in \S2.1. Let $J \in B$ be as in \cite[\S2.1]{CH}. We then have
%$$
%P_{\overline{\chi}}(f_{1,J}')=P_{\chi}(f_{1}')
%$$
%for $f_{1,J}'$ being the $J$-translate of $f_{1}'$.

We now recall the $S$-version of Gross--Zagier formula in \cite[Thm. 1.6]{CST}.
It states that
$$
L' (1, \phi, \chi)=
2^{-\#\Sigma_D+2}\cdot C_\infty \cdot \frac{\pair{\phi^0, \ov\phi^0}_{U_0(N)}}{[\cO_{K}^{\times}:\BZ^{\times}]^2 \sqrt{|D_K| }}\cdot
\frac{\langle P_{f'_1}(\chi), P_{f'_2}(\chi^{-1})\rangle}{\pair{f'_1, f'_2}_{\wh{R}^\times}}
\cdot 
\prod_{v\in S\backslash \{\infty\}} \frac{\beta^0(f_{1, v}, f_{2, v})}{ \beta^0(f'_{1, v}, f'_{2, v})}.
$$
Here $\phi^{0}$ is a normalised new vector fixed by level $U_{1}(N)$ and the pairing $\langle \cdot , \cdot \rangle_{U_{0}(N)}$ is as in \cite[\S1.3]{CST}.  
Further, $f_{1}=f \in V(\pi_{B},\chi)$ and $f_{2} \in V(\pi_{B},\chi^{-1})$ are non-zero vectors, the local pairing $\beta^{0}(\cdot, \cdot)$ is as in \cite[Thm. 1.6]{CST} 
and $C_{\infty}$ denotes the constant in \cite[Thm. 1.8]{CST}.

It follows that the analytic Sha is given by
$$
\CL(A_\chi)=
\frac{C_{\infty} \cdot (\phi,\phi)_{\Gamma_0(N)}}{\sqrt{|D_K|^2}\Omega(A_\chi/K)}\cdot 
 \frac{\langle P_{f_1'}(\chi), P_{f_1'}(\chi)\rangle_{A_{\chi}}}{(f_1, f_2)_{\wh{R}^\times}}\cdot 
 \frac{\Fitt (A_\chi(K)_\tor) \Fitt (A^\vee_{\chi^{-1}}(K)_\tor)}{2^{\# \Sigma_D -5}[\cO_{K}^{\times}:\BZ^{\times}]^{2}\cdot\prod_{v}\Fitt(c_v(A_\chi))}\cdot 
%\prod_{v\in \Sigma} L_v(1, A_\chi)
%\cdot 
\prod_{v\in S\backslash\{\infty\}} \frac{\beta^0(f_{1, v}, f_{2, v})}{ \beta^0(f'_{1, v}, f'_{1, J, v})}.
$$
Here $\langle \cdot, \cdot \rangle_{A}$ is the pairing given by
$\langle P, Q \rangle_{A}=\langle P, \alpha(Q)\rangle .$

We now discuss horizontal variation of the terms in the right hand side of the above expression except the Heegner point. It suffices to show that the terms have bounded prime divisors as the pair $(K,\chi)$-varies.

Note that the terms are identical to the ones appearing in the proof of Proposition \ref{control0} for analytic Sha in the analytic rank zero case. 

In view of the expression for the analytic Sha,
%$\CL(A_{\chi})$,
we thus finish the proof.
\end{proof}
Along with Proposition \ref{Ko}, we have the following immediate consequence for the BSD.
\begin{cor}\label{BSD1}
Let $\phi \in S_{2}(\Gamma_{0}(N))$ be a non-CM newform and $L_{\phi}$ the corresponding Hecke field.
  %Let $p$ be a prime and $\fp_\phi, \fp$ the primes above $p$ of the coefficient fields $L_\phi$.
 % and $L$ corresponding to the fixed embedding, respectively.
 Let $A$ be an abelian variety corresponding to $\phi$ such that $\cO_{\phi} \subset \End(A)$.
% Let $c$ be a positive integer.

There exists a finite set $\Sigma_{A}'$ of primes of $\CO_\phi$ only dependent on $A$,  such that the following holds. For any prime $\fp_0\notin \Sigma_{A}'$ of $\CO_\phi$ and for $\chi\in \fX$ with
the Heegner point $P_{f}(\chi)$ being a $\fp$-indivisible for $\fp$ above $\fp_{0}$, the $\fp$-part of BSD holds for the Serre tensor $A_{\chi}$ as long as $p \nmid h_{K}$ for $(p)=\fp_{0} \cap \BQ$.

\end{cor}

The exceptional set of primes $\Sigma_{A}$ in Proposition \ref{control1} admits an explicit description under mild hypothesis.

\begin{prop}\label{explicit1}

Let $\phi \in S_{2}(\Gamma_{0}(N))$ be a non-CM newform and $L_{\phi}$ the corresponding Hecke field.
%Let $A$ be an abelian variety associated to a newform $\phi\in S_2(\Gamma_0(N))$ without CM and $\End(A)$ is the ring of integers of $\BQ(\phi)$.
Let $K$ be an imaginary quadratic field and $\chi$ an unramified finite order anticyclotomic Hecke charcter over $K$ as above.
Let $N = N^+ N^-$ where
   $l|N^+$ (resp. $l|N^-$) if and only if $l|N$ and split in $K$ (resp. non-split in $K$).
%of conductor $c$.
Let $p$ be a prime and $\fp_0|p$ a prime of $L_{\phi}$.
Suppose that the following holds.
\begin{itemize}
\item [(1)] $(N, D_{K})=1$,
\item [(2)]  $N^-$ is square-free
   with even number of prime factors,
   % and $\ord_\ell (N^{-})=1$ for any prime $\ell |N$,
 \item [(3)]$p\nmid 6ND_{K}$, 
  \item[(4)] There exists an isogeny $\varphi : A \rightarrow A^{\vee}$ with $p \nmid \deg(\varphi)$ and 
    \item [(5)] The Galois representation $\rho_{A, \fp_0}$ has maximal image and the residual representation $\ov{\rho}_{A,\fp_{0}}$ is ramified at $\ell |N^-$ with $\ell^2\equiv 1\mod p$.
\end{itemize}
For all primes $\fp|\fp_0$ of $L=L_{\phi,\chi}$, we then have
$$\ord_{\fp} \CL(A_\chi)=2\cdot \ord_\fp [A_{\chi}(K):\cO_{\phi,\chi}P_{f_{1}}(\chi)]-2\cdot \sum_{\ell|N^+, \chi|_{D_\ell} =1}\ord_p(c_\ell). $$
\end{prop}

\begin{proof}

We first describe an explicit form of Gross--Zagier formula under the hypotheses that $N^-$ is square-free with odd number of prime factors.

We have
\[L'(1,\phi,\chi) =  \frac{8\pi^2(\phi,\phi)_{\Gamma_0(N)}}
	   {[\CO_K^\times:\BZ^\times]^2\sqrt{|D_{K}|}}
	   \frac{ \langle P_f(\chi), P_{f}(\chi)\rangle_{A_{\chi}}}{\langle f,f \rangle}
	   	   \chi^{-1}_{\CN^+}(N^+).\]
	   Here the notation is as in the earlier part of this subsection. This follows from an analysis of an explicit version in \cite[\S1.3]{CST}.
	   As the proof is analogous to the proof of explicit version of the Waldspurger formula in Proposition \ref{explicit0}, we skip the details.

%[State the explicit form of $S$-version of Gross--Zagier under the hypotheses.]
%In view of the explicit Gross--Zagier formula, we have
%$$
%L'(1, A_\chi)=\frac{8\pi^2 (\phi, \phi)_{\Gamma_{0}(N)}}{ \sqrt{|D_{K}|}} \cdot
%\frac{\langle P_\chi(f), P_{\chi}(f)\rangle_{A_{\chi}}}{(f_1, f_2)}.
%$$
It thus follows that
$$\CL(A_\chi)=\frac{8\pi^2 (\phi,\phi)_{\Gamma_0(N)}}{\sqrt{|D_K|^2}\Omega(A_\chi/K)}\cdot
\frac{\langle P_f(\chi), P_{f} (\chi) \rangle_{A_{\chi}}}{\langle f, f \rangle}\cdot \deg(\varphi)\cdot\frac{\Fitt (A_\chi(K)_\tor)
\Fitt (A^\vee_{\chi^{-1}}(K)_\tor)}{\prod_{v}\Fitt(c_v(A_\chi))}
\cdot \chi^{-1}_{\CN^+}(N^+).
$$

We now consider $\fp$-adic valuation of the terms in the right hand side of the above expression. 

%Recall that we have
%$$
%\Omega(A_{\chi}/K)=\Omega(A/K)
%$$
%as fractional ideals up to primes dividing $N$ (proof of Proposition \ref{control0}).

 %In view of the hypothesis (5) on Galois image, both $A_\chi(K)$ and $A^\vee_{\chi^{-1}}(K)$ have no non-trivial $\fp_0$-torsion point.

% Let $(f, f)$ be the pairing defined by Gross and identify $\pi$ and $\wt{\pi}$, then
%$$\frac{P_\chi(f)P_{\chi^{-1}}(f)}{\pair{f, f}}\sim \frac{P_\chi(f)^2}{(f, f)}$$
%up to a $\fp_0$-unit.

Under the hypotheses, the  test vector $f$ is nothing but the new vector.
From \cite[Proof of Lem. 10.1]{Z},
we thus have
$$
\ord_{\fp_0} \left(\frac{8\pi^2 (\phi, \phi)_{\Gamma_{0}(N)}}{\Omega(A/K) (f, f)}\right)=\ord_{\fp_0}(\prod_{\ell |N^-} c_\ell).
$$

The other terms in the above expression for analytic Sha except the one involving Heegner point are identical to the ones appearing in the proof of Proposition \ref{explicit0}. The proposition thus follows by noting that
$$
\ord_{\fp_0}\left(\prod_v c_{v}(A_\chi)\right)=\ord_p \left( \prod_{\ell |N^-} c_\ell \prod_{\ell |N^+, \chi|_{D_\ell}=1} c_\ell^2\right).
$$

%Let $\CA/\BZ$ be the Neron-model of $A$ and $\omega\in \Gamma(\BZ, \wedge \Omega_{\CA/\BZ})$ a generator (unique up to $\pm 1$). Let
%$\Omega_A=\int_{A(\BC)} \omega$. Then there is a fraction ideal $\fb$ of $L$ supported on $Dc^2$ such that
%$$\Omega(A_\chi)=\Omega_A \fb.$$

\end{proof}

\begin{remark} 
As in \cite{CST}, there is a striking resemblance among explicit Waldspurger formula and explicit Gross--Zagier formula in our setup. Essentially, the terms other than the ones involving toric period and Heegner point turn out to be identical. 
\end{remark}

We have the following immediate consequence.

\begin{cor}\label{BSDE1}
Let $\phi \in S_{2}(\Gamma_{0}(N))$ be a non-CM newform and $L_{\phi}$ the corresponding Hecke field.
%Let $A$ be an abelian variety associated to a newform $\phi\in S_2(\Gamma_0(N))$ without CM and $\End(A)$ is the ring of integers of $\BQ(\phi)$.
Let $K$ be an imaginary quadratic field and $\chi$ an unramified finite order anticyclotomic Hecke charcter over $K$ as above.
Let $N = N^+ N^-$ where
   $l|N^+$ (resp. $l|N^-$) if and only if $l|N$ and split in $K$ (resp. non-split in $K$).
%of conductor $c$.
Let $p$ be a prime and $\fp_0|p$ a prime of $L_{\phi}$.
Suppose that the following holds.
\begin{itemize}
\item [(1)] $(N, D_{K})=1$,
\item [(2)]  $N^-$ is square-free
   with even number of prime factors,
   % and $\ord_\ell (N^{-})=1$ for any prime $\ell |N$,
 \item [(3)]$p\nmid 6ND_{K}\cdot h_{K}$, 
  \item[(4)] There exists an isogeny $\varphi : A \rightarrow A^{\vee}$ with $p \nmid \deg(\varphi)$ and 
    \item [(5)] The Galois representation $\rho_{A, \fp_0}$ has maximal image and the residual representation $\ov{\rho}_{A,\fp_{0}}$ is ramified at $\ell |N^-$ with $\ell^2\equiv 1\mod p$.
\end{itemize}

For all primes $\fp|\fp_0$ of $L=L_{\phi}(\chi)$, the $\fp$-part of BSD holds for the Serre tensor $A_{\chi}$ whenever the toric period $P_{f}(\chi)$ is $\fp$-indivisible.
\end{cor}

\section{Main results}
In this section, we describe the main results. 

\subsection{Torus embedding}
In this subsection, we describe a result on the existence of a suitable torus embedding into a quaternion algebra. The embedding plays a role in the proof of the main results in \S5.2.

Let us first introduce the setup.

Let $\phi\in S_2(\Gamma_0(N))$ be a non-CM newform and $L_{\phi}$ the corresponding Hecke field. 
Let $\pi$ the cuspidal automorphic representation of $\GL_{2}(\BA)$ corresponding to the newform 
$\phi$ so that we have the equality
$$L(s,\pi)=L(s,\phi)$$
of complex L-functions. 
Let $A$ the abelian variety corresponding to $\phi$ such that $\cO_{L_{\phi}} \subset \End(A)$.
%Let $c\geq 1$ be a positive integer such that for each $\ell |N$, we have
%$\ord_\ell(c)\geq \cond(\omega_\ell)$ for the local component $\omega_\ell$ of $\omega$.

Let $B$ be a quaternion algebra over $\BQ$ such that there exists an irreducible automorphic  representation $\pi_{B}$ on $B_\BA^\times$ whose Jacquet-Langlands transfer is the automorphic representation $\pi$ of $\GL_2(\BA)$ associated to $\phi$.

Let $S=\Supp (N \infty)$.  Let $K_{0, S}\subset B_S$ be a $\BQ_S$-subalgebra such that $K_{0, \infty}=\BC$ and $K_{0, v}/\BQ_v$ is semi-simple quadratic.
Fix a maximal order $R^{(S)}$ of $B_{\BA}^{(S)}\cong M_2(\BA^{(S)})$.
Let $U^{(S)}=R^{(S)\times}$. Note that $U^{(S)}$ is a maximal compact subgroup of
$B_{\BA}^{\times (S)}\cong \GL_2(\BA^{(S)})$.

\begin{lem}
\label{embedding}
Let $K$ be an imaginary quadratic field with $K_S\cong K_{0, S}$.
Then, there exists an embedding $\iota: K \hookrightarrow B$ such that
\begin{itemize}
\item[(i).] $\iota(K_{S}) = K_{0,S} $ and
\item[(ii).] $\iota(\BA_{K}^{S}) \cap U=\iota(\cO_{K}^{S})$.
\end{itemize}
\end{lem}
\begin{proof}
Let $\iota_{0}:K \hookrightarrow B$ be an embedding such that
$$
\iota(\BA_{K}^{S}) \cap U = \iota(\widehat{\cO}_{K}^{S}).
$$
For $v \in S$, we have $\iota(K_{v}) \simeq K_{v}$. Thus, there exists $b_{v}\in B_{v}^{\times}$
such that
$$
b_{v}^{-1}\iota(K_{v})b_{v}=K_{0,v}.
$$
Let $b_{S}=(b_{v})_{v \in S}\in B_{S} \subset B_{\BA}^{\times}$.

In view of the choice of level $U$, strong approximation for the subgroup $B^{1} \subset B^\times$ consisting of elements with identity norm implies that 
$$
B_{\BA}^{\times}=B^{\times}\cdot U \prod_{v \in S} K_{0,v} .
$$
There exists $b \in \widehat{B}^{\times}$ such that
$b_{S}=bu$ with $b \in B^{\times}$ and $u \in U \prod_{v \in S} K_{0,v} $.
We can take $\iota$ to be $\iota_{0}^{b}:=b^{-1}\iota_{0}b$.
\end{proof}

Let $p$ 
%$$
%p\nmid \prod_{\ell | c,\\{split}} (\ell-1) \prod_{\ell |c, \\{non-split}} (\ell+1)
%$$ 
be an odd prime.
\begin{defn}\label{Theta'}
Let $\Theta_{S}$ denote the set of imaginary quadratic fields $K$ such that
\begin{itemize}
\item[(i)] $K_{S}\cong K_{0,S}$ and
\item[(ii)] $p\nmid h_{K}$.
%\item[(ii)] there exists an embedding $\iota_{K}:K\ra B$ with $K_S=K_{0, S}$ and
%\item[(iii)] $\wh{K}^{(S)}\cap R^{(S)}=\wh{\CO}_K^{(S)}$ under the embedding.
\end{itemize}
\end{defn}
For any imaginary quadratic field $K$ with $K_S\cong K_{0, S}$, we fix an embedding $\iota_K$ as above from now.

\subsection{Main results} In this section, we describe and prove the main results. The proof is essentially a compilation of the results in \S2, \S3 and \S4.1.

Let the notation and assumptions be as in \S3.1. In particular, $\phi \in S_{2}(\Gamma_{0}(N))$ is a non-CM newform and $L_\phi$ the Hecke field. Let $A$ be the corresponding abelian variety with $\cO_{L_{\phi}} \subset \End(A)$.

Let $p$ be an odd prime. As in the introduction, we fix the embeddings $\iota_{\infty}:\overline{\BQ}\hookrightarrow \BC$ and
$\iota_{p}:\overline{\BQ}\hookrightarrow \BC_p $. 
For $\Theta$ in \S4.1, let $K \in \Theta$. For $\chi\in \wh{\Cl}_{K}$,
Let $\fp_\chi$ be the prime above $p$ of $L_{\phi,\chi}$ determined via the embedding $\iota_p$.

For $\chi\in \wh{\Cl}_{K}$, let
$A_\chi=A_K\otimes_\BZ \BZ[\chi]$ denote the Serre tensor where the absolute Galois group $G_K$ acts on $\BZ[\chi]$ via $\chi$ (\cite{MRS}). Then $A_\chi$ is an abelian variety over $K$ with the following properties:
 $$\Sha(A_\chi/K)\otimes_\BZ \BZ[h_{K}^{-1}]\cong \Sha(A/H_{K,c})^\chi\otimes_\BZ \BZ[h_{K}^{-1}], \qquad L(s, A_\chi)=L(s, A, \chi)\in \BC\otimes_\BQ L_{\phi,\chi}.$$
 Let $\CL(A_\chi)$ be the analytic Sha of $A_\chi$ (part (2) of Conjecture \ref{BSD}), which is conjectured to be a non-zero integral ideal of $L_{\phi}(\chi)$ as predicted by the BSD conjecture for the abelian variety $A_{\chi}$ over $K$.
  \begin{defn}
  A character $\chi\in \wh{\Cl}_{K}$ is said to be $p$-regular for $A$ if the following conditions hold.
 \begin{itemize}
 \item $\rank_{\cO_{L_{\phi,\chi}}} A_\chi(K)=\ord_{s=1}L(s, A_\chi/K)$, and
\item  Let $\fp_\chi|p$ be the prime ideal of $\cO_{L_{\phi,\chi}}$ induced by the fixed embedding $\iota_p$, then
\begin{enumerate}
\item[(a)]  $\Sha(A_\chi)$ is finite with trivial $\fp_\chi$-part, and
        \item[(b)]  $\CL(A_\chi)$ has trivial $\fp_\chi$-part.
            \end{enumerate}
            \end{itemize}
            \end{defn}
In particular, for a $p$-regular $\chi\in \wh{\Cl}_{K}$, the $\fp_\chi$-part of the BSD conjecture holds for the Serre tensor $A_\chi$.

Our result regarding $p$-regular characters is the following.
\begin{thm}\label{main'1}
Let $A$ be a non-CM $\GL_2$-type abelian variety over the rationals. Then, there exists an explicit finite subset $\Sigma_{A}$ of primes only dependent on $A$, such that for any prime $p \notin \Sigma_{A}$ and $\epsilon > 0$ we have
$$
\#\bigg{\{} \chi \in \widehat{\Cl}_{K} \bigg{|}\ \chi\ \text{is $p$-regular for $A$}\ \bigg{\}}\gg_\epsilon (\ln |D_K|)^{1-\epsilon}.
$$
Here $K$ varies over imaginary quadratic fields satisfying
\begin{itemize}
\item[(i)] $p\nmid h_K$, and
\item[(ii)] either $p$ splits in $K$ or the root number $\epsilon(A, \chi_0)$ of the Rankin--Selberg convolution $L(s, A, \chi_0)$ equals $+1$ for some $\chi_0\in \wh{\Cl}_K$.
\end{itemize}

\end{thm}
\begin{proof} 
Let $S=\supp(N\infty)$ as before. 

Let $K_{0, S}$ be a $\BQ_S$-subalgebra such that $K_{0, \infty}=\BC$ and $K_{0, v}/F_v$ is semi-simple quadratic. 
Let $U_{0,S}$ be as in \S2.1. 
%$$ U_{0, S}:=\prod_{v\in S,\ \text{$v$  split}}\CO_{K_{0, v}}^\times \times \prod_{v\in S,\ \text{$v$ non-split}} K_{0, v}^\times.$$
Suppose that there exists a finite order character $\chi_{0, S}: U_{0, S}\lra \ov{\BQ}^\times$ with conductor one satisfying (LC1) in \S2.1 such that 
$$
\epsilon(\pi, \chi_{0,v})\chi_{0,v}\eta_{0,v}(-1)=\pm1
$$
for $v \in S$. 

From now, let $K \in \Theta_{S}$ (Definition \ref{Theta'}).
%For any $v\in S$, recall that we say $v$ is non-split if $K_{0, v}$ is a field and split otherwise. 
Let $\fX_{K,\chi_{0,S}}$ denote the set of finite order Hecke characters over $K$ satisfying the conditions (i) and (iii) in Definition \ref{char}.
We have 
$$
\fX_{K,\chi_{0,S}}= \fX_{K,\chi_{0,S}}^{+} \bigcup \fX_{K,\chi_{0,S}}^{-}.
$$
Here $\fX_{K,\chi_{0,S}}^{+}$ (resp. $\fX_{K,\chi_{0,S}}^{-}$) denotes the subset of Hecke characters $\chi$ over $K$ such that $\epsilon(\pi,\chi)=1$ (resp. $\epsilon(\pi,\chi)=-1$). 
%In view of the hypothesis, $\fX_{K,\fc}^{-}$ is non-empty. 

Note that
$$
\Theta_{S} = \Theta_{S}^{+} \bigcup \Theta_{S}^{-}.
$$
Here $\Theta_{S}^{+}$ (resp. $\Theta_{S}^{-}$) denotes the subset of imaginary quadratic fields $K$ such that there exists $\chi \in \fX_{K,\chi_{0,S}}^{+}$ (resp. $\fX_{K,\chi_{0,S}}^{-}$) with $\chi_{S}=\chi_{0,S}$ for some $\chi_{0,S}$.
% satisfying (LC1) and (LC2) such that the Rankin--Selberg convolution corresponding to the pair $(\phi,\chi)$ is self-dual with root number $1$ (resp. $-1$). 
The above union need not be disjoint.

%Let $S=\supp(\fN\fc\infty)$.
%For any $v\in S$, recall that we say $v$ is non-split if $K_{0, v}$ is a field and split otherwise. 
%Let $U_{0,S}$ be as in \S3.1.
%$$ U_{0, S}:=\prod_{v\in S,\ \text{$v$  split}}\CO_{K_{0, v}}^\times \times \prod_{v\in S,\ \text{$v$ non-split}} K_{0, v}^\times.$$
%Suppose we are given a finite order character $\chi_{0, S}: U_{0, S}\lra \ov{\BQ}^\times$ with conductor $\fc$ satisfying (LC1)-(LC3) in \S3.1. 
%As the conductor of Hecke characters in $\fX_{K,\fc}$ is fixed,
%Moreover, there are only finitely many choices for $(S,\chi_{0,S})$.
%Here $\fX_{K,\chi_{0,S}}$ denotes the set of Hecke characters as in Definition \ref{char}. 

Note that there are only finitely many choices for $\chi_{0,S}$. For any such $\chi_{0,S}$, we consider a quaternion algebra $B$ over $\BQ$ with $\ram(B) \subset S$ such that 
$$
\epsilon(\pi, \chi_{0,v})\chi_{0,v}\eta_{0,v}(-1)=\epsilon(B_{v})
$$
for $v \in S$. 

As $K$ varies over imaginary quadratic extensions satisfying the hypotheses in the Theorem, there are only finitely many choices for the characters $\chi_{0,S})$. So, it suffices to consider the case of a fixed character $\chi_{0,S}$.

In the case of root number $1$, the result thus follows from Lemma \ref{embedding}, Theorem \ref{toric1}, Proposition \ref{BD} and Proposition \ref{control0}.

In the case of root number $-1$, the result thus follows from Lemma \ref{embedding}, Theorem \ref{main3}, Proposition \ref{Ko} and Proposition \ref{control1}.
\end{proof}

\begin{remark} In view of the proof, the set of exceptional primes $\Sigma_A$ arises from the ones in Proposition \ref{BD} (resp. Proposition \ref{Ko}) and Proposition \ref{control0} (resp. Proposition \ref{control1}). Note that the set $\Sigma_A$ can be larger than the set of exceptional primes for the $p$-minimality of $\Sha$ or analytic Sha when considered individually.
\end{remark}

Based on the proof, we have the following explicit version of $p$-minimality in individual cases.

\begin{cor}\label{explicit'0}
Let $\phi \in S_{2}(\Gamma_{0}(N))$ be a non-CM newform with $N$ square-free and $L_{\phi}$ the corresponding Hecke field.
%Let $A$ be an abelian variety associated to a newform $\phi\in S_2(\Gamma_0(N))$ without CM and $\End(A)$ is the ring of integers of $\BQ(\phi)$.
%of conductor $c$.
Let $p$ be a prime and $\fp_\phi|p$ a prime of $L_{\phi}$.
%Let $p$ be a prime and $\fp_\phi, \fp$ the primes above $p$ of the coefficient fields $L_\phi$ and $L$ corresponding to the fixed embedding, respectively.
Let $A$ be an abelian variety corresponding to $\phi$ such that $\cO_{L_{\phi}} \subset \End(A)$. Suppose that the following holds.
 \begin{enumerate}
 \item The Galois representation $\rho_{A, \fp_\phi}$ has maximal image and
 \item $p\nmid 6N\cdot (\prod_\ell c_\ell)$ with $c_\ell$ the Tamagawa number at $\ell$.
 %\item The toric period $P_\chi(f)$ is $\fp$ unit with $f$ the $\fp$-primitive toric form in last subsection.
 \end{enumerate}
For any $\epsilon > 0$, we have
$$
\#\big{\{} \chi \in \widehat{\Cl}_{K} \big{|}\ \chi\ \text{is $p$-regular for $A$}\ \big{\}}\gg_\epsilon (\ln |D_K|)^{1-\epsilon}.
$$
Here $K$ varies over imaginary quadratic fields satisfying

%Let $K$ be an imaginary quadratic field and $\chi$ an unramified anticyclotomic charcater over $K$ as above.
%Let $N = N^+ N^-$ where
%   $l|N^+$ (resp. $l|N^-$) if and only if $l|N$ and split in $K$ (resp. non-split in $K$).
%Suppose that the following holds.
\begin{itemize}
\item[(i)] $p \nmid h_{K}$
\item[(ii)] $(N, D_{K})=1$,
\item[(iii)]  $N^-$ has odd number of prime factors and $\ord_\ell (N^{-})=1$ for any prime $\ell |N$.
   Here $N = N^+ N^-$ where
   $\ell |N^+$ (resp. $\ell|N^-$) if and only if $\ell|N$ and split in $K$ (resp. non-split in $K$).
 \item[(iv)] $p\nmid 6ND_{K}$ and
  \item[(v)] The residual representation $\ov{\rho}_{A,\fp_{\phi}}$ is ramified at $\ell |N^-$ with $\ell^2\equiv 1\mod p$.
\end{itemize}

\end{cor}
\begin{proof}
This follows from the proof of Theorem \ref{main'1} along with Proposition \ref{BD} and Proposition \ref{explicit0}.
\end{proof}

\begin{cor}\label{explicit'1}
Let $\phi \in S_{2}(\Gamma_{0}(N))$ be a non-CM newform with $N$ square-free and $L_{\phi}$ the corresponding Hecke field.
%Let $A$ be an abelian variety associated to a newform $\phi\in S_2(\Gamma_0(N))$ without CM and $\End(A)$ is the ring of integers of $\BQ(\phi)$.
%of conductor $c$.
Let $p$ be a prime and $\fp_\phi|p$ a prime of $L_{\phi}$.
%Let $p$ be a prime and $\fp_\phi, \fp$ the primes above $p$ of the coefficient fields $L_\phi$ and $L$ corresponding to the fixed embedding, respectively.
Let $A$ be an abelian variety corresponding to $\phi$ such that $\cO_{L_{\phi}} \subset \End(A)$. Suppose that the following holds.
 \begin{enumerate}
 \item The Galois representation $\rho_{A, \fp_\phi}$ has maximal image,
 \item $p\nmid 6N\cdot (\prod_\ell c_\ell)$ with $c_\ell$ the Tamagawa number at $\ell$ and
 \item There exists a polarisation $\varphi:A \rightarrow A^{\vee}$ with $p \nmid \deg(\varphi)$.

 %\item The toric period $P_\chi(f)$ is $\fp$ unit with $f$ the $\fp$-primitive toric form in last subsection.
 \end{enumerate}
For any $\epsilon > 0$, we have
$$
\#\big{\{} \chi \in \widehat{\Cl}_{K} \big{|}\ \chi\ \text{is $p$-regular for $A$}\ \big{\}}\gg_\epsilon (\ln |D_K|)^{1-\epsilon}.
$$
Here $K$ varies over imaginary quadratic fields satisfying

%Let $K$ be an imaginary quadratic field and $\chi$ an unramified anticyclotomic characater over $K$ as above.
%Let $N = N^+ N^-$ where
%   $l|N^+$ (resp. $l|N^-$) if and only if $l|N$ and split in $K$ (resp. non-split in $K$).
%Suppose that the following holds.
\begin{itemize}
\item[(i)] $p \nmid h_{K}$
\item[(ii)] $p$ split in $K$ and $(N, D_{K})=1$,
\item[(iii)]  $N^-$ is has even number of prime factors.
   Here $N = N^+ N^-$ where
   $\ell| N^+$ (resp. $\ell |N^-$) if and only if $\ell |N$ and split in $K$ (resp. non-split in $K$).
 \item[(iv)] $p\nmid 6ND_{K}$ and
  \item[(v)] The residual representation $\ov{\rho}_{A,\fp_{\phi}}$ is ramified at $\ell |N^-$ with $\ell^2\equiv 1\mod p$.

 \end{itemize}

\end{cor}
\begin{proof}
This follows from the proof of Theorem \ref{main'1} along with Proposition \ref{Ko} and Proposition \ref{explicit1}.
\end{proof}

We end with results in the introduction.
\begin{remark}
(1). Theorem \ref{main0} and Theorem \ref{main1} follow from Theorem \ref{main'1}.

%(2). Theorem \ref{main1} follows from Theorem \ref{Sha0} along with analogous version of Proposition \ref{control0} and Theorem \ref{Sha1} along with analogous version of Proposition \ref{control1}.

(2). Theorem \ref{main2} follows from Corollary \ref{explicit'0} and Corollary \ref{explicit'1}.
\end{remark}

\end{document}